\documentclass[twoside,11pt]{article}

\usepackage{jmlr2e}

\usepackage{amsmath}
\usepackage{amssymb}
\usepackage[utf8]{inputenc} 
\usepackage[T1]{fontenc}    
\usepackage{url}            
\usepackage{booktabs}       
\usepackage{amsfonts}       
\usepackage{nicefrac}       
\usepackage{microtype}      
\usepackage{algorithm}
\usepackage{algorithmic} 
\usepackage{tikz}
\usetikzlibrary{automata, positioning}
\usepackage{wrapfig}
\usepackage{hyperref} 

\usepackage{footnote}
\makesavenoteenv{tabular}
\makesavenoteenv{table}

\definecolor{myblue}{RGB}{6,0,255}
\definecolor{mygreen}{RGB}{0,139,0}
\definecolor{myred}{RGB}{255,4,0}

\newcommand{\DefinedAs}[0]{\mathrel{\mathop:}=}
\newcommand{\AsDefined}[0]{=\mathrel{\mathop:}}

\newcommand{\rbr}[1]{\left(#1\right)}
\newcommand{\sbr}[1]{\left[#1\right]}
\newcommand{\cbr}[1]{\left\{#1\right\}}

\newcommand{\abr}[1]{\left|#1\right|}
\newcommand{\calP}{{\mathcal{P}}}
\newcommand{\field}[1]{\mathbb{#1}}
\newcommand{\Ind}[1]{ \field{I}{\{{#1}\}} }

\newcommand{\norm}[1]{\left\|{#1}\right\|}

\DeclareMathOperator*{\argmin}{argmin}
\DeclareMathOperator*{\argmax}{argmax}
\DeclareMathOperator*{\minimize}{minimize}
\DeclareMathOperator*{\maximize}{maximize}
\DeclareMathOperator*{\subject}{subject~to}

\usepackage{lastpage}
\jmlrheading{26}{2025}{1-\pageref{LastPage}}{6/22; Revised
	8/24}{10/25}{22-0622}{Dongsheng Ding, Kaiqing Zhang, Jiali Duan, Tamer Ba\c{s}ar, Mihailo R.\ Jovanovi\'c}
\ShortHeadings{Natural Policy Gradient Primal-Dual Methods for Constrained MDPs}{Ding, Zhang, Duan, Ba\c{s}ar, and Jovanovi\'c}
\firstpageno{1}

\begin{document}
	
	\title{Convergence and Sample Complexity of Natural Policy Gradient Primal-Dual Methods for Constrained MDPs}
	
	\author{\name Dongsheng~Ding
		\email dongshed@utk.edu \\
		\addr 
		University of Tennessee, Knoxville \\
		Knoxville, TN 37996, USA
		\AND
		\name Kaiqing~Zhang \email kaiqing@umd.edu \\
		\addr 
		University of Maryland, College Park\\
		College Park, MD 20742, USA
		\AND
		\name Jiali~Duan \email jduan24@apple.com \\
		\addr 
		Apple Inc. 
		\\
		Cupertino, CA 95014, USA
		\AND
		\name Tamer~Ba\c{s}ar \email basar1@illinois.edu \\
		\addr 
		University of Illinois Urbana-Champaign\\
		Champaign, IL 61820, USA
		\AND
		\name Mihailo~R.\ Jovanovi\'c \email mihailo@usc.edu \\
		\addr  
		University of Southern California\\
		Los Angeles, CA 90089, USA
	}
	
	\editor{Ambuj~Tewari}
	
	\maketitle
	
	\begin{abstract}%
		We study the sequential decision making problem of  maximizing the expected total reward while satisfying a constraint on the expected total utility. We employ the natural policy gradient method to solve the discounted infinite-horizon optimal control problem for Constrained Markov Decision Processes (constrained MDPs). Specifically, we propose a new Natural Policy Gradient Primal-Dual (NPG-PD) method that updates the primal variable via natural policy gradient ascent and the dual variable via projected subgradient descent. Although the underlying maximization involves a nonconcave objective function and a nonconvex constraint set, under the softmax policy parametrization, we prove that our method achieves global convergence with sublinear rates regarding both the optimality gap and the constraint violation. Such convergence is independent of the size of the state-action space, i.e., it is~dimension-free. Furthermore, for log-linear and general smooth policy parametrizations, we establish sublinear convergence rates up to a function approximation error caused by restricted policy parametrization. We also provide convergence and finite-sample complexity guarantees for two sample-based NPG-PD algorithms. We use a set of computational experiments to showcase the effectiveness of our approach.
		
	\end{abstract}
	
	\vspace*{0.15cm}
	
	\begin{keywords} Constrained Markov decision processes, Natural policy gradient, Constrained nonconvex optimization, Method of Lagrange multipliers, Primal-dual algorithms
	\end{keywords}
	
	\section{Introduction}
	\label{sec.introduction}
    
	Reinforcement learning (RL) studies sequential decision-making problems with the objective of maximizing the expected total reward while interacting with an unknown environment~\citep{sutton2018reinforcement}. Markov Decision Processes (MDPs) are typically used to model the dynamics of the environment. However, in many safety-critical applications, e.g., in autonomous driving~\citep{fisac2018general}, robotics~\citep{ono2015chance}, cyber-security~\citep{zhang2019non}, and financial management~\citep{abe2010optimizing}, the control system is also subject to constraints on its utilities or costs. In this setting, constrained Markov Decision Processes (constrained MDPs) are used to model the environment dynamics~\citep{altman1999constrained} and, in addition to maximizing the expected total reward it is also important to take into account the constraints on the expected total utility/cost as an extra learning objective. 
	
	Policy gradient (PG)~\citep{sutton2000policy} and natural policy gradient (NPG)~\citep{kakade2002natural} methods have enjoyed substantial empirical successes in solving MDPs~\citep{schulman2015trust,lillicrap2015continuous,mnih2016asynchronous,schulman2017proximal,sutton2018reinforcement}. PG methods,~or more generally {\em direct policy search\/} methods, have also been used to solve constrained MDPs~\citep{uchibe2007constrained,borkar2005actor,bhatnagar2012online,chow2017risk,tessler2018reward,liang2018accelerated,paternain2022safe,achiam2017constrained,spooner2020natural}, but most existing theoretical guarantees are asymptotic and/or only provide local convergence around stationary-point policies. On the other hand, it is desirable to show that, for arbitrary initial condition, a solution that enjoys $\epsilon$-optimality gap and $\epsilon$-constraint violation is computed using a finite number of iterations and/or samples. It is thus imperative to establish global convergence guarantees for PG methods when solving constrained MDPs. 
	
	In this work, we provide a theoretical foundation for the non-asymptotic global convergence of PG methods  for solving  constrained MDPs, and answer the following questions: 
	
	\begin{itemize}
		\item[(i)] Can we employ PG methods to solve optimal control problems for constrained MDPs? 
		
		\item[(ii)] Do PG methods converge to the globally optimal solution that satisfies constraints?
		
		\item[(iii)] What is the convergence rate of PG methods and the effect of the function approximation error caused by a restricted policy parametrization?
		
		\item[(iv)] What is the sample complexity of model-free PG methods? 
	\end{itemize}
	
	\subsection{Preview of key contributions} 
	
	Our key contributions are:
	
	\begin{itemize}
		\item[(i)] We propose a simple but effective primal-dual  policy gradient algorithm for solving discounted infinite-horizon optimal control problems for constrained MDPs. Our Natural Policy Gradient Primal-Dual (NPG-PD) method employs natural policy gradient ascent to update the primal variable and projected subgradient descent to update the dual variable.
		
		\item[(ii)] We exploit the structure of the softmax policy parametrization to establish global convergence guarantees in spite of the fact that the objective function in the maximization problem is not concave and the constraint set is not convex. In particular, we prove that our NPG-PD method achieves global convergence at a rate of $O(1/\sqrt{T})$ for both the optimality gap and the constraint violation, where $T$ is the total number of iterations. Our convergence guarantees are dimension-free, i.e., the rate is independent of the size of the state-action space. 
		
		\item[(iii)] We establish sublinear convergence at a rate of $O(1/\sqrt{T})$ in both the optimality gap and the constraint violation for log-linear and general smooth policy parametrizations, up to a function approximation error caused by restricted policy parametrization. This is accomplished by providing a new regret-type primal-dual analysis in the function approximation case, thereby eliminating the need for in-policy class comparison.
		
		\item[(iv)] We provide convergence and finite-sample complexity guarantees for two sample-based NPG-PD algorithms. The new sample complexity of  $O(1/\epsilon^4)$ for generating an $\epsilon$-optimal policy results from our new regret-type primal-dual analysis, accompanied by a practical stochastic gradient ascent method. We utilize a set of computational experiments to showcase the effectiveness of our approach. 
	\end{itemize}
	
	At this point it is worth highlighting the main differences between the results of this paper (as an extended version of our earlier NeurIPS paper~\citep{ding2020natural}) and those in~\cite{ding2020natural}. Although our algorithmic framework here builds on the NPG-PD method of~\cite{ding2020natural}, our new characterization of function approximation error, which is based on the estimation-transfer error decomposition in both the optimality gap and the constraint violation, facilitates the derivation of convergence and sample complexity results described in (iii) and (iv) above. In contrast, the earlier function approximation study of~\cite{ding2020natural} utilizes the classical notion of \emph{compatible function approximation} which is often challenging to control. Furthermore,~\cite{ding2020natural} adopts a standard drift analysis of constraint violation in online optimization, assuming in-policy class feasibility. This approach not only yields sub-optimal rates relative to the tabular case, but it also leaves the optimality of a comparison policy within policy class unjustified. 
	
	Our results as summarized in (iii) and (iv) above extend the PG methods studied in~\cite{agarwal2021optimality} and provide a novel contribution to the constrained MDP setting. In contrast to our earlier work~\citep{ding2020natural}, we establish here the optimal rate for log-linear and general smooth policy parameterizations up to a function approximation error, and eliminate the in-policy class feasibility assumption. By providing a new regret-type primal-dual analysis, we show that the derived rate for the function approximation case matches the optimal rate for the tabular case. Furthermore, in contrast to the sample complexity result of~\cite{ding2020natural}, which holds only when the estimates of value functions are bounded, we employ and analyze a more practical version of stochastic gradient method that does not require this boundedness assumption. Our new analysis allows us to establish an improved sample complexity of $O(1/\epsilon^4)$, compared to the previous $O(1/\epsilon^8)$, where $\epsilon$ is the desired level of accuracy.
	
	In addition to these technical differences, we also characterize the zero constraint violation performance (on average) of our method in both the tabular and the function approximation settings, and conduct computational experiments on a set of benchmark robotic simulation tasks to demonstrate the effectiveness of our approach. These are all new results compared to our conference version~\citep{ding2020natural}.
	
	\subsection{Related work} 
	
	Our work builds on Lagrangian-based constrained MDP algorithms~\citep{altman1999constrained,abad2002self,abad2003policy,borkar2005actor}. However, convergence guarantees of these algorithms are either local (to stationary-point or locally optimal policies)~\citep{bhatnagar2012online,chow2017risk,tessler2018reward} or asymptotic~\citep{borkar2005actor}. In the tabular setting, we compare the convergence rates in Table~\ref{tab.comparison_tabular} by assuming the exact evaluation of policy gradients. When function approximation is used for policy parametrization,~\cite{yu2019convergent} recognized the lack of convexity and showed asymptotic convergence (to a stationary point) of a method based on successive convex relaxations. In contrast, we establish convergence to a globally optimal solution in spite of the lack of convexity. References~\citep{paternain2019constrained,paternain2022safe} are closely related to our work.~\citet{paternain2019constrained} provided a duality analysis for constrained MDPs in the policy space and proposed a provably convergent dual descent algorithm by assuming access to a nonconvex optimization oracle. However, it is not clear how to obtain a solution to a primal nonconvex optimization problem and the global convergence guarantees are not established.~\citet{paternain2022safe} proposed a primal-dual algorithm and provided empirical results, but did not offer a convergence analysis. In spite of the lack of convexity, our work provides global convergence guarantees for a new primal-dual algorithm without using any optimization oracles. For the function approximation setting, we compare the convergence rates and sample complexities in Table~\ref{tab.comparison_function_approximation}. Other related policy optimization methods include CPG~\citep{uchibe2007constrained}, accelerated PDPO~\citep{liang2018accelerated}, CPO~\citep{achiam2017constrained,yang2020projection}, FOCOPS~\citep{zhang2020first}, IPPO~\citep{liu2020ipo}, and CUP~\citep{yang2022constrained} but theoretical guarantees for these algorithms are still lacking. Recently, optimism principles have been used for efficient exploration in constrained  MDPs~\citep{singh2022learning,zheng2020constrained,ding2020provably,qiu2020upper,efroni2020exploration,bai2023provably,yu2021provably,liu2021learning,wei2021provably}. In comparison, our work focuses on the optimization landscape within a primal-dual framework in both model-based and model-free settings.
	
	Our work is also pertinent to the global convergence results of PG methods.~\citet{fazel2018global,malpanbhakhabarwai20,mohzarsoljovCDC19,mohsoljovL4DC20,mohsoljovLCSS21,mohzarsoljovTAC22} provided global convergence guarantees and quantified sample complexity of (natural) PG methods for the nonconvex linear quadratic regulator problem of both discrete- and continuous-time systems.~\cite{zhang2020global} showed that locally optimal policies for MDPs are achievable using PG methods with reward reshaping.~\citet{wang2019neural} demonstrated that (natural) PG methods converge to the globally optimal value when overparametrized neural networks are used. A variant of NPG, trust-region policy optimization (TRPO)~\citep{schulman2015trust}, converges to the globally optimal policy with overparametrized neural networks~\citep{liu2019neural} and for regularized MDPs~\citep{shani2019adaptive}.~\citet{bhandari2024global,bhandari2020note} studied global optimality and convergence of PG methods from a policy iteration perspective.~\citet{agarwal2021optimality} characterized global convergence properties of (natural) PG methods and studied computational, approximation, and sample size issues. Additional recent advances along these lines include~\citep{mei2020global,zhang2020variational,cen2022fast,liu2020improved,khodadadian2022linear}. While all these references handled the lack of convexity in the objective function, additional effort is required to deal with \emph{nonconvex constraints}  that arise in 
	constrained MDPs. Our paper addresses \mbox{this challenge.}
	
	We also remark on some recent work on Lagrangian-based policy optimization.~\citet{liu2021policy,li2024faster,ying2022dual} examined a two-timescale scheme for updating the primal-dual variables by updating policy in an inner loop via an NPG-style subroutine for each dual iterate. In spite of the improved convergence that results from the proposed modifications of the Lagrangian and the dual update, the double-loop scheme often increases computational cost and introduces difficulty in parameter tuning. Additionally,~\citet{liu2021policy,liu2021learning,bai2021achieving} proposed policy optimization algorithms that offer zero constraint violation at the end of training, and~\cite{ding2023last,muller2024truly,montenegro2024last} introduced regularization to single-timescale primal-dual algorithms that achieve policy-iterate convergence, which is orthogonal to our work.
	
	\renewcommand{\arraystretch}{2.0}
	
	\begin{table}
		\begin{center}
			\begin{tabular}{ c||c } 
				\hline
				\textbf{Algorithm} & \textbf{Iteration/Sample complexities}  \\ 
				\hline
				\hline
				PG-PD~\citep{abad2002self}  & asymptotic \\ 
				\hline
				PG-PD~\citep{borkar2005actor} & asymptotic \\ 
				\hline
				NPG-PD (Theorem~\ref{thm.convergence.softmax}, Theorem~\ref{thm.samplecomplexity.loglinear}) & $O\left(\, {1}/{\epsilon^2}\, \right)$ / $O\left(\, {1}/{\epsilon^4} \,\right)$   \\ 
				\hline
			\end{tabular}
		\end{center}
		\caption{Complexity comparison of our NPG-PD method with closely related algorithms for the tabular case with finitely many states/actions. The iteration complexity is determined by the number of gradient-based updates that an algorithm takes to achieve $\epsilon$-optimality gap and $\epsilon$-constraint violation, $\frac{1}{T} \sum_{t\,=\,0}^{T-1} \big( V_r^\star (\rho) - V_r^{(t)} (\rho) \big) \leq \epsilon$ and $\big[\,\frac{1}{T} \sum_{t\,=\,0}^{T-1} \big(b- V_g^{(t)} (\rho) \big)\,\big]_+ \leq \epsilon$, and the sample complexity is determined by the number of trajectory rollouts. }
		\label{tab.comparison_tabular}
	\end{table}
	
	\begin{table}
		\begin{center}
			\begin{tabular}{ c||c } 
				\hline
				\textbf{Algorithm} & \textbf{Iteration/Sample complexities} \\ 
				\hline
				\hline
				PDO~\citep{chow2017risk} &  asymptotic  \\ 
				\hline
				RCPO~\citep{tessler2018reward} & asymptotic  \\ 
				\hline
				CBP~\citep{jain2022towards}
				&  $O\left(\, {1} / {\epsilon^2} \,\right)$ /  --- \\ 
				\hline
				C-NPG-PD~\citep{bai2023achieving} &  $O\left( \, {1} / {\epsilon^2} \, \right)$ /  $O\left(\, {1}/{\epsilon^4} \,\right)$ \\ 
				\hline
				NPG-PD (Theorem~\ref{thm.convergence.loglinear}, Theorem~\ref{thm.samplecomplexity.loglinear}) & $O\left(\, {1}/ {\epsilon^2}\, \right)$ / $O\left(\, {1} / {\epsilon^4} \,\right)$ \\
				\hline
				NPG-PD (Theorem~\ref{thm.convergence.general}, Theorem~\ref{thm.samplecomplexity.general}) & $O\left(\, {1}/ {\epsilon^2} \,\right)$ / $O\left(\, {1}/{\epsilon^4} \,\right)$  \\
				\hline
			\end{tabular}
		\end{center}
		\caption{Complexity comparison of our NPG-PD method with closely related algorithms for the function approximation case with potentially infinitely many states/actions. The iteration complexity is determined by the number of gradient-based iterations an algorithm takes to ensure $\epsilon$-optimality gap and $\epsilon$-constraint violation up to a function approximation error $\epsilon_{\text{fa}}$, $\mathbb{E}\big[\,\frac{1}{T} \sum_{t\,=\,0}^{T-1} \big( V_r^\star (\rho) - V_r^{(t)} (\rho) \big) \,\big]\leq \epsilon +\sqrt{\epsilon_{\text{fa}}}$ and $\mathbb{E}\big[\,\big[\,\frac{1}{T} \sum_{t\,=\,0}^{T-1} \big(b- V_g^{(t)} (\rho) \big)\,\big]_+\,\big] \leq \epsilon +\sqrt{\epsilon_{\text{fa}}}$. The sample complexity is determined by the number of trajectory rollouts to ensure $\epsilon$-optimality gap and $\epsilon$-constraint violation up to a function approximation error $\epsilon_{\text{fa}}$ that is given either by the bias error $\bar{{\epsilon}}_{\text{bias}}$ (C-NPG-PD) or the transfer error $\epsilon_{\text{\normalfont  bias}}$ (NPG-PD). The bias error $\bar{{\epsilon}}_{\text{bias}}$ contains the transfer error $\epsilon_{\text{\normalfont  bias}}$, which captures how well an approximation function class covers the true value function, and the error of policy representation. 
		}
		\label{tab.comparison_function_approximation}
	\end{table}
	
	\subsection{Paper outline}
	
	In Section~\ref{sec.formulation}, we formulate an optimal control problem for constrained Markov decision processes and provide necessary background material. In Section~\ref{sec.npgpd}, we describe our natural policy gradient primal-dual method. We provide convergence guarantees for our algorithm under the tabular softmax policy parametrization in Section~\ref{sec.softmax} and under log-linear and general smooth policy parametrizations in Section~\ref{sec.fa}. We establish convergence and finite-sample complexity guarantees for two model-free primal-dual algorithms in Section~\ref{sec.sample} and provide computational experiments in Section~\ref{sec.experiments}. We close the paper with remarks in Section~\ref{sec.conclusion}.
	
	\section{Problem setup }
	\label{sec.formulation}
	
	In Section~\ref{subsec.cmdps}, we introduce constrained Markov decision processes. In Section~\ref{subsec.Lagrange}, we present the method of Lagrange multipliers, formulate a saddle-point problem for the constrained policy optimization, and exhibit several problem properties: strong duality, boundedness of the optimal dual variable, and constraint violation. In Section~\ref{subsec.policy}, we introduce a parametrized formulation of the constrained policy optimization problem, provide an example of a constrained MDP that is not convex, and present several useful policy parametrizations.
	
	\subsection{Constrained Markov decision processes}
	\label{subsec.cmdps}
	
	We consider an infinite-horizon discounted Constrained Markov Decision Process~\citep{piunovskiyoptimal,altman1999constrained} 
	\[
	\text{CMDP}\,( \, S, \, A, \, P, \, r, \, g, \, b, \, \gamma, \, \rho \, )
	\]
	where $S$ is a finite state space, $A$ is a finite action space, $P$ is a transition probability measure which specifies the transition probability $P(s'\,\vert\,s,a)$ from state $s$ to the next state $s'$ under action $a\in A$, $r$: $S\times A\to[0,1]$ is a reward function, $g$: $S\times A\to[0,1]$ is a utility function, $b$ is a constraint offset, $\gamma \in [0, 1)$ is a discount factor, and $\rho$ is an initial  distribution over $S$.
	
	For any state $s_t$, a stochastic policy $\pi$: $S\to\Delta_A$ is a function in the probability simplex $\Delta_A$ over the action space $A$, i.e., $a_t\sim \pi(\cdot\,\vert\,s_t)$ at time $t$. Let $\Pi$ be a set of all possible policies. A policy $\pi \in \Pi$, together with the initial state distribution $\rho$, induces a distribution over trajectories $\tau=\{(s_t,a_t,r_t,g_t)\}_{t\,=\,0}^\infty$, where $s_0\sim\rho$, $a_t\sim\pi(\cdot\,\vert\,s_t)$ and $s_{t+1} \sim P(\cdot\,\vert\,s_t,a_t)$ for all $t\geq 0$. 
	
	Given a policy $\pi$, the value functions $V_{r}^{\pi}$, $V_{g}^{\pi}$: $S\to\mathbb{R}$ associated with the reward $r$ or the utility $g$ are determined by the expected values of total discounted rewards or utilities received under policy $\pi$:
	\[
	V_{r}^{\pi}(s) 
	\; \DefinedAs \;
	\mathbb{E} \sbr{\,\displaystyle{\sum_{t \, = \,0 }^{\infty}} \gamma^t r(s_t, a_t) \,\big\vert\,\pi, s_0 =s\,},
	~~
	V_{g}^{\pi}(s)
	\; \DefinedAs \;
	\mathbb{E} \sbr{\,\displaystyle{\sum_{t \, = \,0 }^{\infty}} \gamma^t g(s_t, a_t) \,\big\vert\,\pi, s_0 =s\,}
	\]
	where the expectation $\mathbb{E}$ is taken over the randomness of the trajectory $\tau$ induced by $\pi$. Starting from an arbitrary state-action pair $(s,a)$ and following a policy $\pi$, we also introduce the state-action value functions $Q_{r}^{\pi}(s,a)$, $Q_{g}^{\pi}(s,a)$: $S\times A\to\mathbb{R}$ together with their advantage functions $A_r^\pi$, $A_g^\pi$: $S\times A\to\mathbb{R}$:
	\[
	\begin{array}{rcl}
		Q_{\diamond}^{\pi}(s,a) 
		& \DefinedAs &
		\displaystyle{\mathbb{E} \sbr{\,\sum_{t\,=\,0}^{\infty}\gamma^t\diamond(s_t,a_t ) \,\big\vert\,\pi,  s_0 =s,a_0=a\,}}
		\\[0.25cm]
		A_\diamond^\pi 
		& \DefinedAs &
		Q_{\diamond}^{\pi}(s,a) \,-\, V_{\diamond}^{\pi}(s)
	\end{array}
	\] 
	where the symbol $\diamond$ represents either $r$ or $g$. Since $r$, $g \in [0,1]$, we have
	\[
	V_{\diamond}^{\pi}(s) \, \in \, \left[\, 0, \, \frac{1}{1-\gamma} \,\right]
	\] 
	and their expected values under the initial distribution $\rho$ are determined by
	\[
	V_{\diamond}^{\pi}(\rho) 
	\; \DefinedAs \; 
	\mathbb{E}_{s_0\,\sim\,\rho} \left[ \, V_{\diamond}^{\pi}(s_0) \, \right]. 
	\]
	
	Having defined a policy as well as the state-action value functions for the discounted constrained MDP, the objective is to find a policy that maximizes the expected reward value over all policies, subject to a constraint on the expected utility value:
	\begin{equation}\label{eq.cmdp}
		\begin{array}{rl}
			\maximize\limits_{\pi\,\in\,\Pi}
			&
			V_{r}^{\pi}(\rho)
			\\[0.2cm]
			\subject 
			& 
			V_{g}^{\pi}(\rho) 
			\;\geq\; 
			b.
		\end{array}
	\end{equation}
	In view of the aforementioned boundedness of $V_{r}^{\pi}(s)$ and $V_{g}^{\pi}(s)$, we set the constraint offset $b\in ( 0,{1}/(1-\gamma)]$ to make Problem~\eqref{eq.cmdp} meaningful. 
	
	\begin{remark}
		For notational convenience, we consider a single constraint in Problem~\eqref{eq.cmdp}, but our convergence guarantees are readily generalizable to problems with multiple constraints. 
	\end{remark}
	
	\subsection{Method of Lagrange multipliers}
	\label{subsec.Lagrange}
	
	By dualizing constraints~\citep{luenberger1984linear,bertsekas2014constrained}, we cast Problem~\eqref{eq.cmdp} into the following max-min problem:
	\begin{equation}\label{eq.Lagrangian}
		\maximize_{\pi\,\in\,\Pi}\;\,\minimize_{\lambda\,\geq\, 0} \;\, V_{r}^{\pi}(\rho) 
		\,+\, 
		\lambda \left(\,V_{g}^{\pi}(\rho)\,-\,b\,\right)
	\end{equation}
	where $V_L^{\pi,\lambda}(\rho) \DefinedAs V_{r}^{\pi}(\rho) +\lambda\, (V_{g}^{\pi}(\rho)-b)$ is the Lagrangian of Problem~\eqref{eq.cmdp}, $\pi$ is the primal variable, and $\lambda$ is the Lagrange multiplier or dual variable which is nonnegative. The associated dual objective function is defined as 
	\[
	V_D^{\lambda} (\rho)
	\; \DefinedAs \;
	\maximize_{\pi \,\in\, \Pi} \; V_L^{\pi,\lambda}(\rho). 
	\]
	
	Instead of utilizing the linear-programming-based  method~\citep{piunovskiyoptimal,altman1999constrained}, we employ {\em direct policy search\/} method to solve Problem~\eqref{eq.Lagrangian}. Direct policy search is attractive for three reasons: (i) it allows us to directly optimize/monitor the value functions that we are interested in; (ii) it can deal with large state-action spaces via policy parameterization, e.g., neural nets; and (iii) it can utilize policy gradient estimates via simulations of the policy. Since Problem~\eqref{eq.cmdp} is a nonconcave constrained maximization problem and the policy space is often infinite-dimensional, Problems~\eqref{eq.cmdp} and~\eqref{eq.Lagrangian} are challenging. 
	
	In spite of these challenges, Problem~\eqref{eq.cmdp} has nice properties in the policy space when it is strictly feasible. We adapt the standard Slater condition in constrained optimization~\citep{bertsekas2014constrained} and assume strict feasibility of Problem~\eqref{eq.cmdp} throughout the paper.
	
	\begin{assumption}[Slater condition]
		\label{as.slater}
		There exist a constant $\xi>0$ and a policy $\bar{\pi} \in \Pi$ such that $V_{g}^{\bar{\pi}}(\rho) -b \geq \xi$ holds.
	\end{assumption}
	
	The Slater condition is mild in practice because we usually have {\em a priori\/} knowledge on a strictly feasible policy, e.g., the minimal utility is achievable by a particular policy so that the constraint becomes loose.
	
	Let $\pi^\star$ denote an optimal solution to Problem~\eqref{eq.cmdp}, let $\lambda^\star$ be an optimal dual variable:
	\[
	\lambda^\star 
	\; \in \; 
	\argmin_{\lambda\,\geq\, 0} \; V_D^{\lambda} (\rho)
	\] 
	and let the set of all optimal dual variables be $\Lambda^\star$. We use the shorthand notation $V_r^{\pi^\star}(\rho) = V_r^{\star}(\rho)$ and $V_D^{\lambda^\star} (\rho) = V_D^{\star} (\rho)$ whenever it is clear from the context. We recall the strong duality for constrained MDPs and we prove boundedness of an optimal dual variable $\lambda^\star$. 
	
	\begin{lemma}[Strong duality and boundedness of $\lambda^\star$]
		\label{lem.duality}
		Let Assumption~\ref{as.slater} hold. Then 
		\begin{itemize}
			\item[\normalfont(i)] $V_r^{\star}(\rho) \, = \, V_D^{\star} (\rho)$; 
			\item[\normalfont(ii)] $ 0 \, \leq \, \lambda^\star \, \leq \, \left(V_r^{\star}(\rho)-V_r^{  \bar{\pi}} (\rho) \right) / {\xi} .
			$
		\end{itemize}
	\end{lemma}
	\begin{proof}
		The proof of (i) is standard; e.g., see~\citet[Theorem~3.6]{altman1999constrained} or~\citet[Theorem~3]{paternain2022safe}.
		The proof of (ii) builds on the constrained convex optimization~\cite[Section~8.5]{beck2017first}. 	
		Let $\Lambda_a\DefinedAs\{ \lambda\geq 0\,\vert\, V_D^\lambda(\rho) \leq a \}$ be a sublevel set of the dual objective function for $a\in\mathbb{R}$. For any $\lambda \in\Lambda_a$, we have
		\[
		a 
		\; \geq \; 
		V_D^\lambda(\rho) 
		\; \geq \; 
		V_{r}^{\bar{\pi}}(\rho) 
		\,+\, 
		\lambda \left(V_{g}^{\bar{\pi}}(\rho) \,-\, b\right) 
		\; \geq \; 
		V_{r}^{\bar{\pi}}(\rho)
		 \,+\, 
		 \lambda\,\xi
		\]
		where $\bar{\pi}$ is a Slater point. Thus, $\lambda \leq \rbr{a -V_r^{  \bar{\pi}} (\rho)}/\xi$.	
		If we take $a = V_r^{\star}(\rho) = V_D^{\star}(\rho)$, then $\Lambda_a=\Lambda^\star$, which proves (ii).
	\end{proof}
	
	\begin{remark}
		In the proof of Lemma~\ref{lem.duality}~(ii), we choose a particular sublevel set of the dual objective function based on the strong duality~(i). However, the boundedness of an optimal dual variable $\lambda^\star$ does not necessarily depend on the strong duality~(i). In general, we have weak duality $V_D^\star(\rho)\geq  V_r^\star(\rho)$. In this case, the choice of $a=V_r^\star(\rho)$ yields an empty sublevel set $\Lambda_a = \emptyset$, and the choice of $a=V_D^\star(\rho)$ leads to  $0\leq \lambda^\star \leq \rbr{V_D^\star(\rho) -V_r^{  \bar{\pi}} (\rho)}/\xi$, which depends on an optimal dual variable.
	\end{remark}

	Let the value function associated with Problem~\eqref{eq.cmdp} be determined by
	\[
	v(\tau) 
	\; \DefinedAs \; 
	\maximize\limits_{\pi \,\in \, \Pi}
	\; 
	\left\{ \, V_{r}^{\pi}(\rho) \,\big\vert\,V_{g}^{\pi}(\rho) \geq b \; + \; \tau \, \right\}. 
	\]
	Using the concavity of $v(\tau)$ (e.g., see~\citet[Proposition~1]{paternain2019constrained}), in Lemma~\ref{lem.constraint} we establish a bound on the constraint violation, thereby extending a result from the constrained convex optimization~\cite[Section~8.5]{beck2017first} to a constrained nonconvex setting.
	
	\begin{lemma}[Constraint violation]\label{lem.constraint}
		Let Assumption~\ref{as.slater} hold. For any $C\geq 2\lambda^\star$, if there exists a policy $\pi\in\Pi$ and $\delta>0$ such that $V_r^\star(\rho)-V_r^\pi(\rho)+C \, [\,b-V_g^\pi(\rho)\,]_{+}\leq \delta$, then $[\,b-V_g^\pi(\rho)\,]_{+}\leq 2\delta/C$, where $[ \,x\, ]_+ \DefinedAs \max(x,0)$.
	\end{lemma}
	\begin{proof}
		By the definition of $v(\tau)$, we have $v(0) = V_r^\star(\rho)$. 
		We also note that $v(\tau)$ is concave (see the proof of~\citet[Proposition~1]{paternain2019constrained}). First, we show that $-\lambda^\star\in\partial v(0)$. By the definition of  $V_L^{\pi,\lambda}(\rho) $ and the strong duality in Lemma~\ref{lem.duality},
		\[
		V_L^{\pi,\lambda^\star}(\rho) 
		\;\leq\; 
		\maximize_{\pi\,\in\,\Pi}\; V_L^{\pi,\lambda^\star}(\rho) 
		\;=\; 
		V_D^{\star}(\rho) 
		\;=\;
		V_r^{\star}(\rho) 
		\;=\;
		v(0),\; \text{ for all } \pi\in\Pi.
		\]
		Hence, for any $\pi\in\{ \pi\in\Pi \,\vert\,V_{g}^{\pi}(\rho) \geq b + \tau \}$,
		\[
		\begin{array}{rcl}
			v(0) \,-\, \tau \lambda^\star &\geq& V_L^{\pi,\lambda^\star}(\rho) \,-\, \tau\lambda^\star
			\\[0.2cm]
			&=& V_r^{\pi}(\rho) \,+\, \lambda^\star \left(\,V_g^\pi(\rho)\,-\,b\,\right) \,-\, \tau\lambda^\star
			\\[0.2cm]
			&=& V_r^{\pi}(\rho) \,+\, \lambda^\star \left(\,V_g^\pi(\rho)\,-\,b\,-\,\tau\,\right) 
			\\[0.2cm]
			&\geq& V_r^{\pi}(\rho).
		\end{array}
		\]
		Maximizing the right-hand side of the inequality above over $\{ \pi\in\Pi \,\vert\,V_{g}^{\pi}(\rho) \geq b + \tau \}$ yields
		\begin{equation}\label{eq.opt1}
			v(0) \,-\, \tau\lambda^\star 
			\;\geq\; 
			v(\tau)
		\end{equation}
		and, thus, $-\lambda^\star\in\partial v(0)$. 
		
		On the other hand, if we take $\tau = -\left[b - V_g^{{\pi}}(\rho)\right]_+$, then
		\begin{equation}\label{eq.opt2}
			V_r^{{\pi}}(\rho)
			\;\leq\;
			V_r^\star(\rho)
			\;=\;v(0) 
			\;\leq\; 
			v({\tau}).
		\end{equation}
		Combining ~\eqref{eq.opt1} and~\eqref{eq.opt2} yields 
		$
		V_r^{{\pi}}(\rho)-V_r^\star(\rho)  \leq-{\tau}{\lambda^\star}.
		$
		Thus,
		\[
		\rbr{\,C \,-\, \lambda^\star\,} \abr{{\tau}} 
		\; = \;  
		-\, \lambda^\star\abr{{\tau}} \,+\, C \abr{{\tau}}
		\; = \; 
		{\tau} \lambda^\star \,+\, C\abr{{\tau}}
		\; \leq \; 
		V_r^\star(\rho) \,-\, V_r^{{\pi}}(\rho) \,+\, C \abr{{\tau}}
		\]
		which completes the proof by applying the assumed condition on $\pi$.
	\end{proof}
	
	Aided by the above properties implied by the Slater condition, we target the max-min Problem~\eqref{eq.Lagrangian} using the primal-dual method. 
	
	\subsection{Policy parametrization}
	\label{subsec.policy}
	
	Introduction of a set of parametrized policies $\{\pi_\theta \,\vert\,\theta\in \Theta\}$ brings Problem~\eqref{eq.cmdp} into a constrained optimization problem over a finite-dimensional parameter space $\Theta$:
	\begin{equation}\label{eq.cmdp.p}
		\begin{array}{rl}
			\maximize\limits_{\theta\,\in\,\Theta}
			&
			V_{r}^{\pi_\theta}(\rho)
			\\[0.2cm]
			\subject 
			& 
			V_{g}^{\pi_\theta}(\rho) 
			\;\geq\; 
			b.
		\end{array}
	\end{equation}
	A parametric version of Problem~\eqref{eq.Lagrangian} is given by
	\begin{equation}\label{eq.Lagrangian.p}
		\maximize_{\theta \,\in\,\Theta}\;\,\minimize_{\lambda\,\geq\, 0} 
		\;\,  	
		V_{r}^{\pi_\theta}(\rho) \,+\, \lambda\, \big(\,V_{g}^{\pi_\theta}(\rho) \, - \, b\,\big)
	\end{equation}
	where 
	$
	V_L^{\pi_\theta,\lambda}(\rho) 
	\DefinedAs
	V_{r}^{\pi_\theta}(\rho) + \lambda (V_{g}^{\pi_\theta}(\rho)-b)
	$
	is the associated Lagrangian and $\lambda$ is the Lagrange multiplier. The dual objective function is determined by $V_D^{\lambda} (\rho) \DefinedAs \maximize_{\theta}  V_L^{\pi_\theta,\lambda}(\rho)$. The primal maximization problem~\eqref{eq.cmdp.p} is finite-dimensional but not concave, even in the absence of a constraint~\citep{agarwal2021optimality}. In Lemma~\ref{lem.nonconvex} we prove that, in general, Problem~\eqref{eq.cmdp.p} is not convex because it involves maximization of a nonconcave objective function over a nonconvex constraint set. The proof is provided in Appendix~\ref{pf.nonconvex} and it utilizes an example of a constrained MDP in Figure~\ref{fig.cmdp}. 
	
	\begin{lemma}[Lack of convexity]\label{lem.nonconvex}
		There exists a constrained MDP for which the objective function $V_r^{\pi_\theta}(s)$ in Problem~\eqref{eq.cmdp.p} is not concave and the constraint set $\{\theta \in\Theta\,\vert\, V_g^{\pi_\theta}(s) \geq b\}$ is not convex. 
	\end{lemma}
	
	\begin{figure}[h]
		\centering
		\begin{tikzpicture}
			\node[state]             (s1) {$s_1$};
			\node[state, right=1.6cm of s1] (s2) {$s_2$};
			\node[state, right=1.6cm of s2] (s3) {$s_3$};
			\node[state, above=1.6cm of s1] (s4) {$s_4$};
			\node[state, above=1.6cm of s2] (s5) {$s_5$};
			
			\draw[every loop]
			(s1) edge[right, auto=right]  node {$(0,0)$} (s2)
			(s2) edge[right, auto=right]  node {$(0,0)$} (s3)
			(s3) edge[loop above]             node {$(0,0)$} (s3)
			(s1) edge[above, auto=right] node {$(0,1)$} (s4)
			(s4) edge[loop above]             node {$(0,0)$} (s4)
			(s2) edge[above, auto=right] node {$(1,1)$} (s5)
			(s5) edge[loop above]             node {$(0,0)$} (s5);
		\end{tikzpicture}
		\caption{An example of a constrained MDP for which the objective function $V_r^{\pi_\theta}(s)$ in Problem~\eqref{eq.cmdp.p} is not concave and the constraint set $\{\theta \in\Theta\,\vert\,V_g^{\pi_\theta}(s)\geq b\}$ is not convex. A pair $(r,g)$ associated with a directed arrow represents the $(\text{reward, utility})$ received when an action in a certain state is taken. This example is utilized in the proof of Lemma~\ref{lem.nonconvex}.} 
		\label{fig.cmdp}
	\end{figure}
	
	In general, the Lagrangian $V_L^{\pi_\theta,\lambda}(\rho)$ in Problem~\eqref{eq.Lagrangian.p} is convex in $\lambda$ but not concave in $\theta$. While many algorithms for solving max-min optimization problems, e.g., those proposed in~\citet{lin2019gradient,nouiehed2019solving,yang2020global}, require extra assumptions on the max-min structure or only guarantee convergence to a stationary point, we exploit problem structure and propose a new primal-dual method to compute a globally optimal solution to Problem~\eqref{eq.Lagrangian.p}. Before doing that, we first introduce several useful classes of policies. 
	
	\subsubsection{Direct policy parametrization}
	
	A direct parametrization of a policy is a probability distribution:
	\[
	\pi_\theta (a\,\vert\,s)  
	\; = \;	
	\theta_{s,a} \;\text{ for all } \theta \, \in \, \Delta_{A}^{|S|}
	\]
	where $\theta_s \in \Delta_{A}$ for any $s\in S$, i.e., $\theta_{s,a} \geq 0$ and $\sum_{a\,\in\,A} \theta_{s,a} = 1$. This policy class is complete since it directly represents any stochastic policy. Even though it is a challenging class of policy to work with from both theoretical and computational viewpoints~\citep{mei2020global,agarwal2021optimality}, it offers a useful sanity check for many policy search methods.
	
	\subsubsection{Softmax policy parametrization} 
	
	This class of policies is parametrized by a softmax function:
	\begin{equation}\label{eq.softmax}
		\pi_\theta(a\,\vert\,s) 
		\;=\;
		\frac{\exp(\theta_{s,a})}{\sum_{a'\,\in\, A} \exp(\theta_{s,a'})} \;\text{ for all }\; \theta \in\mathbb{R}^{|S||A|}.
	\end{equation}
	The softmax policy can be used to represent any stochastic policy, and its closure contains all stationary policies. It has been utilized to study the convergence properties of many RL algorithms~\citep{bhandari2024global,agarwal2021optimality,mei2020global,cen2022fast,khodadadian2022linear}, and it offers several algorithmic advantages: (i) it equips the policy with a special structure so that the NPG update works like the classical multiplicative weights update in online learning~\citep{freund1997decision,cesa2006prediction}; (ii) it can be used to interpret the function approximation error~\citep{agarwal2021optimality}. 
	
	\subsubsection{Log-linear policy parametrization} 
	
	A log-linear policy is given by
	\begin{equation}\label{eq.loglinear}
		\pi_\theta(a\,\vert\,s) 
		\;=\;
		\frac{\exp(  \theta^\top \phi_{s,a} )}{\sum_{a'\,\in\, A} \exp( \theta^\top \phi_{s,a'}  )} \; \text{ for all }\; \theta \in\mathbb{R}^{d}
	\end{equation}
	where $\phi_{s,a} \in \mathbb{R}^d$ is the feature map at a state-action pair $(s,a)$. The log-linear policy builds on the softmax policy by applying the softmax function to a set of linear functions in a given feature space. More importantly, it exactly characterizes the linear function approximation via policy parametrization~\citep{agarwal2021optimality}; see~\cite{miryoosefi2022simple,amani2021safe} for solving constrained MDPs with linear function approximation.
	
	\subsubsection{General policy parametrization} 
	
	A general class of stochastic policies is given by $\{\pi_\theta \,\vert\,\theta\in \Theta \}$ with $\Theta\subset\mathbb{R}^d$ without specifying a parametric structure of the policy $\pi_\theta$. The parameter space has dimension $d$ and this policy class covers the settings that utilize nonlinear function approximation, such as (deep) neural networks~\citep{liu2019neural,wang2019neural}. 
	
	When we choose the dimension $d$ such that $d\ll |S||A|$ in either the log-linear policy or the general nonlinear policy, the policy class has limited expressiveness and may not contain all stochastic policies. Motivated by this observation, the theory that we develop in Section~\ref{sec.fa} establishes global convergence up to a function approximation error caused by restricted policy parametrization.
	
	\section{Natural policy gradient primal-dual method}
	\label{sec.npgpd}
	
	In Section~\ref{subsec.warmup}, we provide a brief summary of three basic optimization methods that have been used to solve the constrained policy optimization problem~\eqref{eq.cmdp.p}. In Section~\ref{subsec.npgpd}, we propose a natural policy gradient primal-dual method which represents an extension of the natural policy gradient method to the constrained problems.
	
	\subsection{Constrained policy optimization methods}
	\label{subsec.warmup}
	
	We summarize three basic optimization methods that can be employed to solve the primal problem~\eqref{eq.cmdp.p}. We assume that the value function and the policy gradient can be evaluated exactly for any given policy. 
	
	We first introduce some useful definitions. The discounted visitation distribution $d_{s_0}^\pi$ of a policy $\pi$ and its expectation over the initial distribution $\rho$ are, respectively, given by
	\begin{equation}\label{eq.visitation}
		\begin{array}{rcl}
			d_{s_0}^\pi (s)
			& \DefinedAs & 
			(1-\gamma) \displaystyle{\sum_{t\,=\,0}^{\infty}} \gamma^t P^\pi(s_t=s \,\vert\,s_0)
			\\[0.2cm]
			d_\rho^\pi(s) 
			& \DefinedAs &
			\mathbb{E}_{s_0\,\sim\,\rho} \left[\,d_{s_0}^\pi (s)\,\right]
		\end{array}
	\end{equation}
	where  $P^\pi(s_t=s \,\vert\,s_0)$ is the probability of visiting state $s$ at time $t$ under the policy $\pi$ with an initial state $s_0$. When the use of parametrized policy $\pi_\theta$ is clear from the context, we use $V_{r}^{\theta}(\rho)$ to denote $V_{r}^{\pi_\theta}(\rho)$. When $\pi_\theta(\cdot\,\vert\,s)$ is differentiable and when it belongs to the probability simplex, i.e., $\pi_\theta\in \Delta_A^{|S|}$ for all $\theta$, the policy gradient of the Lagrangian~\eqref{eq.Lagrangian.p} is determined by
	\begin{equation}\label{eq.policy gradient}
		\begin{array}{rcl}
			\nabla_\theta V_L^{\theta,\lambda}(s_0) 
			&  =  &
			\nabla_\theta V_r^{\theta} (s_0) 
			\, + \,
			 \lambda \,\nabla_\theta V_g^{\theta} (s_0)
			\\[0.2cm]
			&  =  &
			\dfrac{1}{1 \, - \, \gamma}\, \mathbb{E}_{s\,\sim\, d_{s_0}^{\pi_\theta}}\,\mathbb{E}_{a\,\sim\,\pi_\theta(\cdot\,\vert\,s)} \sbr{\,A_L^{\theta,\lambda} (s,a) \nabla_\theta \log \pi_\theta(a\,\vert\,s)\,}
		\end{array}
	\end{equation}
	where $A_L^{\theta,\lambda} (s,a)  \, \DefinedAs\,  A_r^{\theta}(s,a)  + \lambda A_g^{\theta}(s,a)$. 
	
	\subsubsection{Dual method} 
	
	When the strong duality in Lemma~\ref{lem.duality} holds, it is convenient to work with the dual formulation of Problem~\eqref{eq.cmdp.p}:
	\begin{equation}\label{eq.cmdp.d}
		\minimize_{\lambda\,\geq\, 0} \;   V_D^{\lambda} (\rho).
	\end{equation}
	While the dual objective function is convex regardless of the concavity of the primal maximization problem~\eqref{eq.cmdp.p}, it is often non-differentiable~\citep{bertsekasnonlinear08}. Thus, a projected dual subgradient descent is used to solve the dual problem:
	\[
	\lambda^{(t+1)} 
	\; = \; 
	\mathcal{P}_+ \left( \lambda^{(t)} 
	\, - \, 
	\eta\, g_\lambda^{(t)} \right)
	\]
	where  $\calP_+( \cdot )$ is the projection to the non-negative real axis, $\eta>0$ is the stepsize, and  $g_\lambda^{(t)}$ is a subgradient of the dual objective  function evaluated at $\lambda  = \lambda^{(t)}$, i.e., $g_\lambda^{(t)} \in \partial_\lambda V_D^{\lambda^{(t)}} (\rho)$.
	
	The dual method works in the space of dual variables and it requires efficient evaluation of the subgradient of the dual objective function. We note that computing the dual objective function $V_D^\lambda(\rho)$ for a given $\lambda = \lambda^{(t)}$ in each step amounts to a standard unconstrained RL problem~\citep{paternain2019constrained}. In spite of global convergence guarantees for several policy search methods in the tabular setting, it is challenging to obtain the dual objective function and/or to compute its subgradient, e.g., when the problem dimension is high and/or when the state space is continuous. Although the primal problem can be approximated using the first-order Taylor expansion~\citep{achiam2017constrained,yang2020projection}, inverting Hessian matrices becomes a computational burden,  and it is costly to implement the dual method.
	
	\subsubsection{Primal method} 
	
	In the primal method, a policy search strategy works directly on the primal problem~\eqref{eq.cmdp.p} by seeking an optimal policy in a feasible region. The key challenge is to ensure the feasibility of the next policy iterate in the search direction, which is similar to the use of the primal method in nonlinear programming~\citep{luenberger1984linear}.
	
	An intuitive approach is to check the feasibility of each policy iterate and determine whether the constraint is active~\citep{xu2021crpo}. If the policy iterate is feasible or the constraint is inactive, we move towards maximizing the single objective function; otherwise, we look for a feasible direction. For the softmax policy parametrization~\eqref{eq.softmax}, this can be accomplished using a simple first-order gradient-based  method:
	\begin{equation}\label{eq.primal}
		\begin{array}{rcl}
			\theta_{s,a}^{(t+1)} 
			& = &
			\theta_{s,a}^{(t)}  
			\, + \, 
			\eta\, G_{s,a}^{(t)}(\rho)
			\\[0.2cm]
			G_{s,a}^{(t)}(\rho) 
			& \DefinedAs &
			\begin{cases}
				\dfrac{1}{1-\gamma} \, A_g^{(t)}(s,a), \; \text{ when } \; V_g^{(t)}(\rho) \, < \, b - \epsilon_b
				\\[0.2cm]
				\dfrac{1}{1-\gamma} \, A_r^{(t)}(s,a), \; \text{ when } \; V_g^{(t)}(\rho) \, \geq \, b - \epsilon_b
			\end{cases}
		\end{array}
	\end{equation} 
	where we use shorthand $A_r^{(t)}(s,a)$ and $A_g^{(t)}(s,a)$ to denote $A_r^{\theta^{(t)}}(s,a)$ and $A_g^{\theta^{(t)}}(s,a)$, respectively, $G_{s,a}^{(t)}(\rho)$ is the gradient-ascent direction determined by the scaled version of the advantage functions, and $\epsilon_b$ is the relaxation parameter for the constraint $V_{g}^{\pi_\theta}(\rho) \geq b$. When the iterate violates the relaxed constraint, $V_{g}^{\pi_\theta}(\rho) \geq b-\epsilon_b$ with $\epsilon_b>0$, it maximizes the constraint function to gain feasibility. More reliable evaluation of the feasibility often demands a more tractable characterization of the constraint, e.g., by utilizing Lyapunov function~\citep{chow2018lyapunov}, Gaussian process modeling~\citep{sui2018stagewise}, backward value function~\citep{satija2020constrained}, and logarithmic penalty function~\citep{liu2020ipo}. Hence, the primal method offers the adaptability of adjusting a policy to satisfy the constraint, which is desirable in safe training applications. However, global convergence theory is still lacking, and recent progress~\citep{xu2021crpo} requires a careful relaxation of the constraint. 
	
	\subsubsection{Primal-dual method}
	
	The primal-dual method simultaneously updates the primal and dual variables~\citep{arrowstudies58}. With the direct parametrization $\pi_{\theta}(a\,\vert\,s) = \theta_{s,a}$, a basic primal-dual method~\citep{abad2003policy} performs the following Policy Gradient Primal-Dual (PG-PD) update:
	\begin{equation}\label{eq.PGAGD}
		\begin{array}{rcl}
			\theta^{(t+1)}  &=&  \calP_\Theta\left(\theta^{(t)} \,+\, 
			\eta_1\,\nabla_\theta V_L^{\theta^{(t)},\lambda^{(t)}} (\rho)\right)
			\\[0.15cm]
			\lambda^{(t+1)} &=& \calP_\Lambda\left( \lambda^{(t)} \, - \, \eta_2\, \big(\,{V_g^{\theta^{(t)}}(\rho)\,-\,b}\,\big)\right)
		\end{array}
	\end{equation}
	where $\nabla_\theta V_L^{\theta^{(t)},\lambda^{(t)}} (\rho)\DefinedAs\nabla_\theta V_r^{\theta^{(t)}} (\rho) + \lambda^{(t)} \nabla_\theta V_g^{\theta^{(t)}} (\rho)$, $\eta_1>0$ and $\eta_2>0$ are the stepsizes, $\calP_\Theta$ is the projection onto a probability simplex $\Theta \DefinedAs \Delta_{A}^{|S|}$, and $\calP_\Lambda$ is the projection that will be specified later. For the max-min formulation~\eqref{eq.Lagrangian.p}, the PG-PD method~\eqref{eq.PGAGD} directly performs projected gradient ascent in the policy parameter $\theta$ and descent in the dual variable $\lambda$, both over the Lagrangian $V_L^{\theta,\lambda}(\rho)$. The primal-dual method overcomes the disadvantages of the primal and dual methods either by relaxing the precise calculation of the subgradient of the dual objective function, or by changing the descent direction via tuning of the dual variable. While this simple method provides a foundation for solving constrained MDPs~\citep{chow2017risk,tessler2018reward}, the lack of convexity in~\eqref{eq.Lagrangian.p} makes it challenging to establish the global convergence theory for the primal-dual method, which is our primary objective. 
	
	We first leverage the structure of the constrained policy optimization problem~\eqref{eq.cmdp.p} to provide a positive result in terms of both the optimality gap and the constraint violation. 
	
	\begin{theorem}[Restrictive convergence: direct policy parametrization]\label{thm.convergence.direct}
		Let Assumption~\ref{as.slater} hold with a policy class $ \{\pi_\theta = \theta\,\vert\,\theta\in\Theta\}$ and let $\Lambda=[\,0, {2}/\rbr{(1-\gamma)\xi}\,]$, $\rho>0$, $\lambda^{(0)}=0$, and $\theta^{(0)}$ be such that $V_r^{\theta^{(0)}}(\rho) \geq V_r^\star(\rho)$. If we choose
		$\eta_1 =\Theta(1)$ and $\eta_2=\Theta(1/\sqrt{T})$, then the iterates $ \theta^{(t)}$ generated by the PG-PD method~\eqref{eq.PGAGD} satisfy
		\[
		\begin{array}{rcl}
			\text{\normalfont(Optimality gap)}\;\; 
			\displaystyle 
			\frac{1}{T} \sum_{t\,=\,0}^{T-1} \big(\,V_r^\star(\rho) \,-\, V_r^{(t)}(\rho) \,\big)
			& \leq & \displaystyle
			C_1
			\frac{|A|\, |S| }{(1-\gamma)^6\, {T}^{1/4}} \left\|{d_\rho^{\pi^\star}/{\rho}} \right\|_\infty^2
			\\[0.4cm]
			\displaystyle
			\text{\normalfont(Constraint violation)}\;\;
			\sbr{\,
				\frac{1}{T}\sum_{t \, = \,0 }^{T-1} \big(\,b \,-\, V_g^{(t)}(\rho)\,\big)
				\,}_+  
			& \leq & \displaystyle
			C_2\frac{|A|\, |S| }{(1-\gamma)^6 \, T^{1/4}} \left\|{d_\rho^{\pi^\star}/{\rho}}\right\|_\infty^2
		\end{array}
		\]
		where $C_1$ and $C_2$ are absolute constants that are independent of $T$.
	\end{theorem}
	
	For tabular constrained MDPs with direct policy parametrization, Theorem~\ref{thm.convergence.direct} guarantees that, on average, the optimality gap $V_r^\star(\rho) - V_r^{(t)}(\rho)$ and the constraint violation $b-V_g^{(t)}(\rho)$ decay to zero at a sublinear rate $1/{T}^{1/4}$. However, this rate explicitly depends on the sizes of state/action spaces $|S|$ and $|A|$, and the distribution shift $\|{{d_\rho^{\pi^\star}}/{\rho}}\|_\infty$, which characterizes the exploration factor. A careful initialization $\theta^{(0)}$ that satisfies $V_r^{\theta^{(0)}}(\rho) \geq V_r^{\theta^\star}(\rho)$ is also required. We leave it as future work to prove a tighter rate for this tabular setting.
	
	The proof of Theorem~\ref{thm.convergence.direct}, which first appeared in~\cite{ding2022convergence}, is provided in Appendix~\ref{pf.convergence.direct} for completeness. We exploit the problem structure that casts the primal problem~\eqref{eq.cmdp.p} as a linear program in the occupancy measure~\citep{altman1999constrained} and apply the convex optimization analysis. This method is not well-suited for large-scale problems, and the projection onto a high-dimensional probability simplex is not desirable in practice. We next introduce a natural policy gradient primal-dual method to overcome these challenges and provide stronger convergence guarantees. 
	
	\subsection{Natural policy gradient primal-dual (NPG-PD) method}
	\label{subsec.npgpd}
	
	The Fisher information matrix induced by a parametrized policy $\pi_\theta$, denoted
	\[
	F_\rho(\theta) 
	\; \DefinedAs  \;  
	\mathbb{E}_{s\,\sim\,d_\rho^{\pi_\theta}} \mathbb{E}_{a\,\sim\,\pi_\theta(\cdot\,\vert\,s)}  
	\left[\,
	\nabla_\theta \log \pi_\theta(a\,\vert\,s) 
	\rbr{ \nabla_\theta \log \pi_\theta(a\,\vert\,s)}^\top
	\,\right],
	\]
	is now used in the update of the primal variable in our primal-dual algorithm. The expectations are taken over the randomness of the state-action trajectory induced by $\pi_\theta$, and the Natural Policy Gradient Primal-Dual (NPG-PD) method for solving Problem~\eqref{eq.Lagrangian.p} is given by
	\begin{equation}\label{eq.NPGAGD}
		\begin{array}{rcl}
			\theta^{(t+1)}  
			& = &  
			\theta^{(t)} 
			\; + \; 
			\eta_1\,
			 F^\dagger_\rho(\theta^{(t)}) \, \nabla_\theta V_L^{\theta^{(t)},\lambda^{(t)}} (\rho)
			\\[0.15cm]
			\lambda^{(t+1)} 
			& = & 
			\calP_\Lambda \left( \lambda^{(t)} \, - \, \eta_2\, \big(\,{V_g^{\theta^{(t)}}(\rho) \,-\,b}\,\big) \right)
		\end{array}
	\end{equation}
	where $\dagger$ denotes the Moore-Penrose inverse of a given matrix, $\calP_\Lambda(\cdot)$ denotes the projection to an  interval $\Lambda$ that will be specified later, and ($\eta_1, \eta_2$) are the constant stepsizes in the updates of the primal and dual variables. The primal update $\theta^{(t+1)}$ is obtained using a preconditioned gradient ascent via the natural policy gradient
	$
	F^\dagger_\rho(\theta^{(t)}) \nabla_\theta V_L^{(t)} (\rho) 
	$
	and it represents the policy gradient of the Lagrangian $V_L^{{(t)}}(\rho)$ in the manifold induced by the Fisher information matrix $F_\rho(\theta^{(t)})$. On the other hand, the dual update $\lambda^{(t+1)}$ is obtained using a projected subgradient descent by collecting the constraint violation
	$
	b  - V_g^{{(t)}}(\rho),
	$
	where, for brevity, we use $V_L^{(t)} (\rho)$ and $V_g^{(t)}(\rho)$ to denote $V_L^{\theta^{(t)},\lambda^{(t)}} (\rho)$ and $V_g^{\theta^{(t)}}(\rho)$, respectively.

    Throughout the paper, we abbreviate the policy notation $\pi_{\theta^{(t)}}$ as $\pi^{(t)}$, or as $\pi_\theta^{(t)}$ when it is necessary to emphasize the dependence on $\theta$.
	
	In Section~\ref{sec.softmax}, we establish the global convergence of  NPG-PD~\eqref{eq.NPGAGD} under the softmax policy parametrization. In Section~\ref{sec.fa}, we examine the general policy parametrization and, in Section~\ref{sec.sample}, we analyze the sample complexity of two sample-based implementations of the NPG-PD method~\eqref{eq.NPGAGD}. 
	
	\begin{remark}\label{re: PDL}
		The performance difference lemma~\citep{kakade2002approximately,agarwal2021optimality}, which quantifies the difference between $V_\diamond^{\pi}(s_0)$ and $V_\diamond^{\pi'}(s_0)$ for any two policies $\pi$ and $\pi'$ and any state $s_0$ as follows
		\[
		V_\diamond^{\pi}(s_0) \, - \, V_\diamond^{\pi'}(s_0) 
		\; = \;
		\frac{1}{1-\gamma} \,\mathbb{E}_{s\,\sim\,d_{s_0}^\pi, a\,\sim\,\pi(\cdot\,\vert\,s)}\left[ \, A_\diamond^{\pi'}(s,a) \,\right], 
		\]
		is utilized in our analysis, where the symbol $\diamond$ denotes either $r$ or $g$.  
	\end{remark}
	
	\section{Tabular softmax parametrization: dimension-free global convergence}
	\label{sec.softmax}
	
	We first examine the NPG-PD method~\eqref{eq.NPGAGD} under the softmax policy parametrization~\eqref{eq.softmax}.  Strong duality in Lemma~\ref{lem.duality} holds on the closure of the softmax policy class, because of the completeness of the softmax policy class. Even though the maximization problem~\eqref{eq.cmdp.p} is not concave, we establish global convergence of our algorithm with dimension-independent convergence rates. 
	
	We first exploit the softmax policy structure to show that the primal update in~\eqref{eq.NPGAGD} can be expressed in a more compact form; see Appendix~\ref{app.npgpd} for the proof.
	
	\begin{lemma}[Primal update as MWU]
		\label{lem.npgpd}
		Let $\Lambda \DefinedAs [\,0, 2/((1-\gamma)\xi)\,]$ and let $A_L^{(t)}(s,a) \DefinedAs A_r^{(t)}(s,a)+\lambda^{(t)} A_g^{(t)}(s,a)$. Under the softmax parametrized policy~\eqref{eq.softmax}, the NPG-PD algorithm~\eqref{eq.NPGAGD} can be brought to the following form
		\begin{subequations}
			\label{eq.NPG-PD-softmax}
			\begin{equation}
				\label{eq.NPGAGD.clean}
				\begin{array}{rcl}
					\theta_{s,a}^{(t+1)} 
					& = & 
					\theta_{s,a}^{(t)} 
					\, + \, 
					\dfrac{\eta_1 }{1-\gamma} \,A_L^{(t)}(s,a)
					\\[0.15cm]
					\lambda^{(t+1)} 
					& = & 
					\calP_\Lambda
					\left( \,\lambda^{(t)} 
					\, - \, 
					\eta_2\, \big(\,V_g^{{(t)}}(\rho)\,-\,b\,\big) \,\right).
				\end{array}
			\end{equation}
			Furthermore, the primal update in~\eqref{eq.NPGAGD.clean} can be equivalently expressed as 
			\begin{equation}
				\pi^{(t+1)}(a\,\vert\,s) 
				\; = \; 
				\pi^{(t)}(a\,\vert\,s) \,\frac{\exp \rbr{ \frac{\eta_1}{1-\gamma}\, A_L^{(t)}(s,a) }}{Z^{(t)}(s)}
				\label{eq.pi-softmax}
			\end{equation}
		\end{subequations}
		where $Z^{(t)}(s) \DefinedAs \sum_{a\,\in\, A} \pi^{(t)}(a\,\vert\,s) \exp \big( \frac{\eta_1}{1-\gamma} A_L^{(t)}(s,a) \big)$.	
	\end{lemma}
	
	The primal update in~\eqref{eq.NPGAGD.clean} does not depend on the state distribution $d_\rho^{\pi^{(t)}}$ that appears in the NPG-PD algorithm~\eqref{eq.NPGAGD} through the policy gradient~\eqref{eq.policy gradient}. This is because of the Moore-Penrose inverse of the Fisher information matrix in~\eqref{eq.NPGAGD}. Furthermore, the policy update~\eqref{eq.pi-softmax} is given by the multiplicative weights update (MWU), which is commonly used in online linear optimization~\citep{cesa2006prediction}. In contrast, an advantage function appears in the MWU policy update at each iteration in~\eqref{eq.pi-softmax}. {Such a connection to MWU has been first identified in the unconstrained MDP case by  \cite{agarwal2021optimality}.}
	
	In Theorem~\ref{thm.convergence.softmax}, we establish global convergence of the NPG-PD algorithm~\eqref{eq.NPGAGD.clean} with respect to both the optimality gap $V_r^\star(\rho) - V_r^{(t)}(\rho)$ and the constraint violation $b-V_g^{(t)}(\rho)$. Even though we set $\theta_{s,a}^{(0)}=0$ and $\lambda^{(0)}=0$ in the proof of Theorem~\ref{thm.convergence.softmax} in Section~\ref{subsec.softmax.pf}, the global convergence can be established for arbitrary initial conditions.
	
	\begin{theorem}[Global convergence: softmax policy parametrization]\label{thm.convergence.softmax}
		Let Assumption~\ref{as.slater} hold for $\xi>0$ and let us fix $T>0$ and $\rho\in\Delta_S$. If we choose $\eta_1 = 2 \log|A|$ and $\eta_2=2(1-\gamma)/{\sqrt{T}}$, then the iterates $\{\pi^{(t)}\}_{t=0}^{T-1}$ generated by the algorithm~\eqref{eq.NPG-PD-softmax} satisfy
		\[
		\begin{array}{rcl}
			\text{\normalfont(Optimality gap)}\;\; 
			\displaystyle 
			\frac{1}{T} \sum_{t\,=\,0}^{T-1} \big(\, V_r^\star(\rho) \,-\, V_r^{(t)}(\rho)\, \big)
			& \leq &
			\displaystyle
			\frac{7}{(1-\gamma)^2}\frac{1}{\sqrt{T}}
			\\[0.4cm]
			\text{\normalfont(Constraint violation)}\;\; 
			\displaystyle 
			\sbr{\,
				\frac{1}{T}\sum_{t \, = \,0 }^{T-1} \big(\,b\,-\,V_g^{(t)}(\rho)\,\big)
				\,}_+  
			& \leq & 
			\displaystyle
			\frac{2/\xi \, + \, 4\,\xi}{(1-\gamma)^{2}}\frac{1}{\sqrt{T}}.
		\end{array}
		\]
	\end{theorem}
	Theorem~\ref{thm.convergence.softmax} demonstrates that, on average, the reward value function converges to its globally optimal value and that the constraint violation decays to zero. In other words, for a desired accuracy $\epsilon$, it takes $O({1}/{\epsilon^2})$ iterations to compute the solution which is $\epsilon$ away from a globally optimal one (with respect to both the optimality gap and the constraint violation). We note that the required number of iterations only depends on the desired accuracy $\epsilon$ and is \emph{independent of the sizes of the state and action spaces\/}. Although the maximization problem~\eqref{eq.cmdp.p} is not concave, our rates of  $(1/\sqrt{T},1/\sqrt{T})$ for the optimality gap and constraint violation outperform the classical ones of  $(1/\sqrt{T},1/{T}^{3/4})$~\citep{mahdavi2012trading}, and match the achievable rates for solving online {\em convex\/} minimization problems with {\em convex} constraint sets~\citep{yu2017online}. Moreover, in contrast to the bounds established for the PG-PD algorithm~\eqref{eq.PGAGD} in Theorem~\ref{thm.convergence.direct}, the bounds in Theorem~\ref{thm.convergence.softmax} for the NPG-PD algorithm~\eqref{eq.NPGAGD} under the softmax policy parameterization do not depend on the initial state distribution~$\rho$.
	
	As shown in Lemma~\ref{lem.improvement} in Section~\ref{subsec.softmax.pf}, the reward and utility value functions are coupled, and the natural policy gradient method in the unconstrained setting does not provide monotonic improvement to either of them~\cite[Section~5.3]{agarwal2021optimality}. To address this challenge, we introduce a new line of nonconvex analysis by bridging the online regret analysis in unconstrained MDPs~\citep{even2009online,agarwal2021optimality} and the Lagrangian methods in constrained optimization~\citep{beck2017first}. To bound the optimality gap, via a drift analysis of the dual update, we first establish the bounded average performance in Lemma~\ref{lem.average.performance} in Section~\ref{subsec.softmax.pf}. Furthermore, instead of using methods from constrained convex optimization~\citep{mahdavi2012trading,yu2017online,wei2019online,yuan2018online}, which either require extra assumptions or have a slow convergence rate, under strong duality, we establish that the constraint violation for the nonconvex problem~\eqref{eq.cmdp.p} converges with the same rate as the optimality gap. To the best of our knowledge, this appears to be the first such result for nonconvex constrained optimization problems. 
    
	
	\subsection{Proof of Theorem~\ref{thm.convergence.softmax}}
	\label{subsec.softmax.pf}
	We employ the performance difference lemma in Remark~\ref{re: PDL} to show the joint policy improvement per iteration in the reward and utility value functions. Neither of them is necessarily monotonic. 
	
	\begin{lemma}[Non-monotonic improvement]\label{lem.improvement}
		For any distribution of the initial state $\mu$, the iterates $(\pi^{(t)}, \lambda^{(t)})$ of the algorithm \eqref{eq.NPG-PD-softmax} satisfy
		\begin{equation}\label{eq.improvement}
			V_r^{(t+1)}(\mu) 
			\,-\, 
			V_r^{(t)}(\mu) 
			\,+\,
			\lambda^{(t)}\,
			 \big(\, V_g^{(t+1)}(\mu) \,-\, V_g^{(t)}(\mu)\, \big) 
			\; \geq \;
			\frac{1-\gamma}{\eta_1}\, 
			\mathbb{E}_{s \,\sim\,\mu} \sbr{\,\log Z^{(t)}(s)\,}
			\; \geq \; 
			0.
		\end{equation}
	\end{lemma}
	\begin{proof}
		Let $d_\mu^{(t+1)} \DefinedAs d_\mu^{\pi^{(t+1)}}$. The performance difference lemma in conjunction with the multiplicative weights update in~\eqref{eq.pi-softmax} yield
		\[
		\begin{array}{rcl}
			\displaystyle V_r^{(t+1)}(\mu) 
			\,-\, 
			V_r^{(t)}(\mu) 
			& = &
			\displaystyle \frac{1}{1-\gamma} \,\mathbb{E}_{s\,\sim\,d_{\mu}^{(t+1)}}\sbr{\, \sum_{a\,\in\, A} \pi^{(t+1)}(a\,\vert\,s) A_r^{(t)}(s,a) \,}
			\\[0.2cm]
			&=&\displaystyle \frac{1}{\eta_1} \,\mathbb{E}_{s\,\sim\,d_{\mu}^{(t+1)}}\sbr{\, \sum_{a\,\in\, A} \pi^{(t+1)}(a\,\vert\,s) \log\rbr{ \frac{\pi^{(t+1)}(a\,\vert\,s)}{\pi^{(t)}(a\,\vert\,s)} Z^{(t)}(s) } \,}
			\\[0.2cm]
			&&\displaystyle -\,\frac{\lambda^{(t)}}{1-\gamma} \,\mathbb{E}_{s\,\sim\,d_{\mu}^{(t+1)}}\sbr{\, \sum_{a\,\in\, A} \pi^{(t+1)}(a\,\vert\,s) A_g^{(t)}(s,a) \,}
			\\
			&  = & \displaystyle \frac{1}{\eta_1} \,\mathbb{E}_{s\,\sim\,d_{\mu}^{(t+1)}}\sbr{\, D_{\text{KL}}\rbr{\pi^{(t+1)}(\cdot\,\vert\,s) \,\big\Vert\,\pi^{(t)}(\cdot\,\vert\,s) } \,} 
			\\[0.2cm]
			&&\displaystyle 
			+\,\frac{1}{\eta_1} \,\mathbb{E}_{s\,\sim\,d_{\mu}^{(t+1)}} \sbr{\, \log Z^{(t)}(s)\,}
			\\[0.2cm]
			&&\displaystyle -\,\frac{\lambda^{(t)}}{1-\gamma} \,\mathbb{E}_{s\,\sim\,d_{\mu}^{(t+1)}}\sbr{\, \sum_{a\,\in\, A} \pi^{(t+1)}(a\,\vert\,s) A_g^{(t)}(s,a) \,}
		\end{array}
		\]
		where the last equality follows from the definition of the Kullback-Leibler divergence or relative entropy between distributions $p$ and $q$, $D_{\text{KL}} (p\,\Vert\,q) \DefinedAs \mathbb{E}_{x\,\sim\,p} \log ( {p(x)}/{q(x)} )$. Furthermore,	
		\[
		\begin{array}{rcl}
			&& \!\!\!\! \!\!\!\! \!\! \displaystyle \frac{1}{\eta_1} \,\mathbb{E}_{s\,\sim\,d_{\mu}^{(t+1)}}\sbr{\, D_{\text{KL}}\rbr{\pi^{(t+1)}(\cdot\,\vert\,s) \,\big\Vert\,\pi^{(t)}(\cdot\,\vert\,s) } \,} 
			\,+\,
			\frac{1}{\eta_1} \,\mathbb{E}_{s\,\sim\,d_{\mu}^{(t+1)}}  \sbr{\,\log Z^{(t)}(s)\,}
			\\[0.2cm]
			&& \!\!\!\! \!\!\!\! \!\! \displaystyle -\,\frac{\lambda^{(t)}}{1-\gamma} \,\mathbb{E}_{s\,\sim\,d_{\mu}^{(t+1)}}\sbr{\, \sum_{a\,\in\, A} \pi^{(t+1)}(a\,\vert\,s) A_g^{(t)}(s,a) \,}
			\\[0.2cm]
			&\overset{(a)}{\geq}&\displaystyle \frac{1}{\eta_1} \,\mathbb{E}_{s\,\sim\,d_{\mu}^{(t+1)}} \sbr{\,\log Z^{(t)}(s)\,}
			\,-\,
			\frac{\lambda^{(t)}}{1-\gamma} \,\mathbb{E}_{s\,\sim\,d_{\mu}^{(t+1)}}\sbr{\, \sum_{a\,\in\, A} \pi^{(t+1)}(a\,\vert\,s) A_g^{(t)}(s,a) \,}
			\\[0.2cm]
			&\overset{(b)}{=}& \displaystyle\frac{1}{\eta_1} \,\mathbb{E}_{s\,\sim\,d_{\mu}^{(t+1)}} \sbr{\,\log Z^{(t)}(s)\,}
			\,-\,
			\lambda^{(t)}\big({V_g^{(t+1)}(\mu) \,-\, V_g^{(t)}(\mu)}\big)
		\end{array}
		\]
		is a consequence of the performance difference lemma (see Remark~\ref{re: PDL}), where we drop a nonnegative term in $(a)$ and $(b)$. The first inequality in~\eqref{eq.improvement} follows from a component-wise inequality $d_\mu^{(t+1)} \geq (1-\gamma)\mu$, which is obtained using~\eqref{eq.visitation}.
		
		Now we prove that $\log Z^{(t)}(s)\geq 0$. From the definition of $Z^{(t)}(s)$, we have
		\[
		\begin{array}{rcl}
			\log Z^{(t)}(s) &=& \displaystyle\log \rbr{ \sum_{a\,\in\,A} \pi^{(t)}(a\,\vert\,s) \exp \rbr{ \frac{\eta_1}{1-\gamma} \rbr{A_r^{(t)}(s,a)+\lambda^{(t)} A_g^{(t)}(s,a)} }}
			\\[0.2cm]
			&\overset{(a)}{\geq}& \displaystyle\sum_{a\,\in\,A} \pi^{(t)}(a\,\vert\,s) \log \rbr{\exp \rbr{ \frac{\eta_1}{1-\gamma} \rbr{A_r^{(t)}(s,a)+\lambda^{(t)} A_g^{(t)}(s,a)} }}
			\\[0.2cm]
			&=& \displaystyle\frac{\eta_1}{1-\gamma}\sum_{a\,\in\, A} \pi^{(t)}(a\,\vert\,s) \rbr{A_r^{(t)}(s,a)+\lambda^{(t)} A_g^{(t)}(s,a)} 
			\\[0.2cm]
			&=&\displaystyle \frac{\eta_1}{1-\gamma}\sum_{a\,\in\, A} \pi^{(t)}(a\,\vert\,s) A_r^{(t)}(s,a)  \, + \, \frac{\eta_1}{1-\gamma} \lambda^{(t)}\sum_{a\,\in\,A} \pi^{(t)}(a\,\vert\,s) A_g^{(t)}(s,a)
			\\[0.2cm]
			&\overset{(b)}{=}& 0
		\end{array}
		\]
		where in $(a)$ we apply the Jensen's inequality to the concave function $\log(x)$. On the other hand, the last equality follows from the definitions of $A_r^{(t)}(s,a)$ and $A_g^{(t)}(s,a)$, which yield
		\[
		\begin{array}{rcl}
			\displaystyle{\sum_{a \, \in \, A}} \pi^{(t)}(a\,\vert\,s) A_r^{(t)}(s,a) 
			& = &
			\displaystyle{\sum_{a \, \in \, A}} \pi^{(t)}(a\,\vert\,s) \big( Q_{r}^{(t)}(s,a) - V_{r}^{(t)}(s)\big)
			\; = \; 
			0
			\\[0.35cm]
			\displaystyle{\sum_{a \, \in \, A}} \pi^{(t)}(a\,\vert\,s) A_g^{(t)}(s,a)  
			& = &
			0,
		\end{array}
		\]
		completing the proof. 
	\end{proof}
	
	Lemma~\ref{lem.improvement} states that each primal update~\eqref{eq.pi-softmax} improves the Lagrangian-like term $V_r^{(t)}(\mu) + \lambda^{(t)} V_g^{(t)}(\mu)$ to $V_r^{(t+1)}(\mu) + \lambda^{(t)} V_g^{(t+1)}(\mu)$, with improvement depending on the previous primal-dual update $(\pi^{(t)},\lambda^{(t)})$.  This lemma can be viewed as a constrained version of the policy improvement established for the unconstrained case~\citep{agarwal2021optimality},  resulting from setting $\lambda^{(t)} = 0$. In fact, the dual iterate $\lambda^{(t)}$ captures how the constraint violation of  policy improvement affects the reward value function, which is a unique feature of constrained policy optimization. Because of this superimposed effect, there is no monotonic improvement in the reward or utility value functions. 
	
	In constrained convex optimization, the primal iterate cannot reduce the unconstrained objective function, monotonically, and some averaging scheme has to be imposed~\citep{beck2017first}. In our nonconvex context, we examine the average of value functions, which is similar to the regret analysis in online optimization.
	We next compare the average value functions of the policy iterates generated by the algorithm~\eqref{eq.NPG-PD-softmax}  with the ones that result from the use of an optimal policy. 
	
	\begin{lemma}[Bounded average performance]\label{lem.average.performance}
		Let Assumption~\ref{as.slater} hold and let us fix $T>0$ and $\rho\in\Delta_S$. Then the iterates $\{(\pi^{(t)},\lambda^{(t)})\}_{t=0}^{T-1}$ generated by the algorithm~\eqref{eq.NPG-PD-softmax} satisfy
		\begin{equation}\label{eq.dualkey}
			\frac{1}{T} \sum_{t\,=\,0}^{T-1}  
			\left(
			\big(V_r^\star(\rho) - V_r^{(t)}(\rho) \big) 
			+
			\lambda^{(t)}\big(V_g^{\star}(\rho) - V_g^{(t)}(\rho)\big)
			\right)
			\, \leq \, 
			\frac{\log|A|}{\eta_1 T} 
			+\frac{1}{(1-\gamma)^2 T} 
			+\frac{2\eta_2}{(1-\gamma)^3}.
		\end{equation}
	\end{lemma}
	\begin{proof}
		Let $d^\star \DefinedAs d_\rho^{\pi^\star}$. The performance difference lemma in conjunction with the multiplicative weights update in~\eqref{eq.pi-softmax} yield
		\[
		\begin{array}{rcl}
			V_r^\star(\rho) \,-\, V_r^{(t)}(\rho) 
			&=& \displaystyle\frac{1}{1-\gamma} \,\mathbb{E}_{s\,\sim\,d^\star}\sbr{ \, \sum_{a\,\in\, A} \pi^\star(a\,\vert\,s) A_r^{(t)}(s,a) \, }
			\\[0.2cm]
			&=&\displaystyle \frac{1}{\eta_1} \,\mathbb{E}_{s\,\sim\,d^\star}\sbr{ \, \sum_{a\,\in\, A} \pi^\star(a\,\vert\,s) \log\rbr{ \frac{\pi^{(t+1)}(a\,\vert\,s)}{\pi^{(t)}(a\,\vert\,s)} Z^{(t)}(s) } \, }
			\\[0.2cm]
			&&\displaystyle
			-~\frac{\lambda^{(t)}}{1-\gamma} \,\mathbb{E}_{s\,\sim\,d^\star}\sbr{ \, \sum_{a\,\in\, A} \pi^\star(a\,\vert\,s) A_g^{(t)}(s,a) \, }.
		\end{array}
		\]
		Application of the definition of the Kullback–Leibler divergence or relative entropy between distributions $p$ and $q$, $D_{\text{KL}} (p\,\Vert\,q) \DefinedAs \mathbb{E}_{x\,\sim\,p} \log ( {p(x)}/{q(x)} )$, and the performance difference lemma again yields
		\begin{equation}\label{eq.PDL}
			\begin{array}{rcl}
				V_r^\star(\rho) \,-\, V_r^{(t)}(\rho) 
				&=& \displaystyle\frac{1}{\eta_1} \,\mathbb{E}_{s\,\sim\,d^\star}\sbr{\, D_{\text{KL}}\rbr{\pi^\star(\cdot\,\vert\,s) \,\big\Vert\,\pi^{(t)}(\cdot\,\vert\,s) } - D_{\text{KL}}\rbr{\pi^\star(\cdot\,\vert\,s) \,\big\Vert\,\pi^{(t+1)}(\cdot\,\vert\,s) }\, } 
				\\[0.2cm]
				&&\displaystyle+~\frac{1}{\eta_1} \,\mathbb{E}_{s\,\sim\,d^\star}\sbr{\, \log Z^{(t)}(s)\,}
				\,-\,\frac{\lambda^{(t)}}{1-\gamma}\, \mathbb{E}_{s\,\sim\,d^\star}\sbr{\, \sum_{a\,\in\, A} \pi^\star(a\,\vert\,s) A_g^{(t)}(s,a) \,}
				\\[0.2cm]
				&=&\displaystyle \frac{1}{\eta_1} \,\mathbb{E}_{s\,\sim\,d^\star}\sbr{\, D_{\text{KL}}\rbr{\pi^\star(\cdot\,\vert\,s) \,\big\Vert\,\pi^{(t)}(\cdot\,\big\vert\,s) } - D_{\text{KL}}\rbr{\pi^\star(\cdot\,\vert\,s) \,\big\Vert\,\pi^{(t+1)}(\cdot\,\vert\,s) }\,} 
				\\[0.2cm]
				&&\displaystyle+~\frac{1}{\eta_1} \,\mathbb{E}_{s\,\sim\,d^\star} \sbr{\,\log Z^{(t)}(s)\,}
				\,-\,\lambda^{(t)} \big({V_g^\star(\rho) \,-\, V_g^{(t)}(\rho)}\big).
			\end{array}
		\end{equation}
		On the other hand, the first inequality in~\eqref{eq.improvement} with $\mu = d^\star$ becomes
		\begin{equation}\label{eq.improvement.d}
			V_r^{(t+1)}(d^\star) \,-\, V_r^{(t)}(d^\star) \,+\,\lambda^{(t)} \big( V_g^{(t+1)}(d^\star) \,-\, V_g^{(t)}(d^\star)  \big) 
			\; \geq \;
			\frac{1-\gamma}{\eta_1} \,\mathbb{E}_{s \,\sim\,d^\star} \sbr{\,\log Z^{(t)}(s)\,}.
		\end{equation}
		Hence, application of~\eqref{eq.improvement.d} to the average of~\eqref{eq.PDL} over $t=0,1,\ldots,T-1$ leads to
		\begin{equation}\label{eq.primalkey}
			\begin{array}{rcl}
				&& \!\!\!\! \!\!\!\! \!\! \displaystyle\frac{1}{T} \sum_{t\,=\,0}^{T-1} \big(\, V_r^\star(\rho) \,-\, V_r^{(t)}(\rho) \,\big)
				\\[0.2cm]
				&=& \displaystyle\frac{1}{\eta_1 T} \sum_{t\,=\,0}^{T-1}  \,\mathbb{E}_{s\,\sim\,d^\star}\sbr{\, D_{\text{KL}}\rbr{\pi^\star(\cdot\,\vert\,s) \,\big\Vert\,\pi^{(t)}(\cdot\,\vert\,s) } - D_{\text{KL}}\rbr{\pi^\star(\cdot\,\vert\,s) \,\big\Vert\,\pi^{(t+1)}(\cdot\,\vert\,s) }\,} 
				\\[0.2cm]
				&&\displaystyle+~\frac{1}{\eta_1 T} \sum_{t\,=\,0}^{T-1} \,\mathbb{E}_{s\,\sim\,d^\star} \sbr{\,\log Z^{(t)}(s)\,}
				\,-\,\frac{1}{T} \sum_{t\,=\,0}^{T-1} \lambda^{(t)}\big(\,{V_g^\star(\rho) \,-\, V_g^{(t)}(\rho)}\,\big)
				\\[0.2cm]
				&\leq &\displaystyle \frac{1}{\eta_1 T} \sum_{t\,=\,0}^{T-1}  \,\mathbb{E}_{s\,\sim\,d^\star}\sbr{\, D_{\text{KL}}\rbr{\pi^\star(\cdot\,\vert\,s) \,\big\Vert\,\pi^{(t)}(\cdot\,\vert\,s) } - D_{\text{KL}}\rbr{\pi^\star(\cdot\,\vert\,s) \,\big\Vert\,\pi^{(t+1)}(\cdot\,\vert\,s) }\,} 
				\\[0.2cm]
				&&\displaystyle+~\frac{1}{ (1-\gamma)T} \sum_{t\,=\,0}^{T-1} \,\big(\,{V_r^{(t+1)}(d^\star) \,-\, V_r^{(t)}(d^\star)}\,\big)
				\\[0.2cm]
				&&\displaystyle+~\frac{1}{(1-\gamma)T} \sum_{t\,=\,0}^{T-1}\lambda^{(t)} \big(\,{V_g^{(t+1)}(d^\star) \,-\, V_g^{(t)}(d^\star)  }\,\big)\,-\,\frac{1}{T} \sum_{t\,=\,0}^{T-1} \lambda^{(t)}\big(\,{V_g^\star(\rho) \,-\, V_g^{(t)}(\rho)}\,\big).
			\end{array}
		\end{equation}
		From the dual update in~\eqref{eq.NPGAGD.clean}, we have
		\begin{equation}\label{eq.primalkey.as}
			\begin{array}{rcl}
				&& \!\!\!\! \!\!\!\! \!\! \displaystyle\frac{1}{T} \sum_{t\,=\,0}^{T-1}\lambda^{(t)} \big({V_g^{(t+1)}(\mu) \,-\, V_g^{(t)}(\mu)  } \big)
				\\[0.2cm]
				&=&\displaystyle \frac{1}{T} \sum_{t\,=\,0}^{T-1} \big(\,{\lambda^{(t+1)}  V_g^{(t+1)}(\mu) \,-\, \lambda^{(t)} V_g^{(t)}(\mu)  }\,\big) \,+\, \frac{1}{T} \sum_{t\,=\,0}^{T-1} \big(\,{\lambda^{(t)}\,-\, \lambda^{(t+1)} }\,\big) V_g^{(t+1)}(\mu) 
				\\[0.2cm]
				&\overset{(a)}{\leq} &\displaystyle \frac{1}{T} \lambda^{(T)}  V_g^{(T)}(\mu) \,+\,\frac{1}{T} \sum_{t\,=\,0}^{T-1} \big\vert{\lambda^{(t)}\,-\, \lambda^{(t+1)} }\big\vert V_g^{(t+1)}(\mu) 
				\\[0.2cm]
				&\overset{(b)}{\leq} &\displaystyle \frac{2\eta_2}{(1-\gamma)^2}
			\end{array}
		\end{equation}
		where we take a telescoping sum for the first sum in $(a)$ and drop a non-positive term, and in $(b)$ we utilize $\vert{\lambda^{(T)}}\vert \leq {\eta_2 T}/(1-\gamma)$ and $\vert{\lambda^{(t)}- \lambda^{(t+1)} }\vert\leq {\eta_2}/(1-\gamma)$, which follows from the dual update in~\eqref{eq.NPGAGD.clean}, the non-expansiveness of the projection $\mathcal{P}_\Lambda$, and the boundedness of the value function $V_g^{(t)}(\mu)\leq {1}/(1-\gamma)$. Application of~\eqref{eq.primalkey.as} with $\mu=d^\star$ and the use of telescoping sum to~\eqref{eq.primalkey} yield
		\[
		\begin{array}{rcl}
			&& \!\!\!\! \!\!\!\! \!\! \displaystyle\frac{1}{T} \sum_{t\,=\,0}^{T-1} \big( \,V_r^\star(\rho) \,-\, V_r^{(t)}(\rho) \,\big)
			\\[0.2cm]
			&\leq&\displaystyle \frac{1}{\eta_1 T} \,\mathbb{E}_{s\,\sim\,d^\star}\sbr{\, D_{\text{KL}}\rbr{\pi^\star(\cdot\,\vert\,s) \,\big\Vert\,\pi^{(0)}(\cdot\,\vert\,s) } \,}
			\,+\,\frac{1}{(1-\gamma) T} V_r^{(T)}(d^\star) 
			\,+\,\frac{2\eta_2}{(1-\gamma)^3}
			\\[0.2cm]
			&&\displaystyle-~\frac{1}{T} \sum_{t\,=\,0}^{T-1} \lambda^{(t)} \big(\,{V_g^\star(\rho) \,-\, V_g^{(t)}(\rho)}\,\big).
		\end{array}
		\]
		Finally, we use $D_{\text{KL}}\rbr{p\,\Vert\,q}\leq \log|A|$ for $p\in \Delta_A$ and $q=\text{Unif}_A$, $V_r^{(T)}(d^\star) \leq{1}/(1-\gamma)$, and $V_g^\star(\rho)\geq b$ to complete the proof.
	\end{proof}
	
	Lemma~\ref{lem.average.performance} shows that the average difference between $(V_r^\star(\rho), V_g^\star(\rho))$ and $(V_r^{(t)}(\rho), V_g^{(t)}(\rho))$ can be bounded by a $(T,\eta_1,\eta_2)$-term. As aforementioned, when there is no constraint (e.g., $\eta_2=0$), it is straightforward to strengthen Lemma~\ref{lem.average.performance} as the fast rate result in the unconstrained case~\citep[Theorem~16]{agarwal2021optimality}. We also note that this average performance analysis generalizes to the function approximation setting in Section~\ref{sec.fa}, with an additional characterization of function approximation errors. 
	
	\begin{proof}[Proof of Theorem~\ref{thm.convergence.softmax}]
		
		\noindent\textbf{Bounding the optimality gap}. From the dual update in~\eqref{eq.NPGAGD.clean}, we have
		\begin{subequations}
			\begin{equation}\label{eq.lambda}
				\begin{array}{rcl}
					0\;\,\leq\;\,\rbr{\lambda^{(T)}}^2 &=&\displaystyle \sum_{t \, = \,0 }^{T-1} \big(\,{(\lambda^{(t+1)} )^2\,-\,(\lambda^{(t)})^2}\,\big)
					\\[0.2cm]
					&=& \displaystyle\sum_{t \, = \,0 }^{T-1} \rbr{ \big({\calP_\Lambda\, ( \, \lambda^{(t)} - \eta_2 \, (V_g^{(t)}(\rho)-b)\, ) }\big)^2\,-\,(\lambda^{(t)})^2 }
					\\[0.2cm]
					&\overset{(a)}{\leq}&\displaystyle \sum_{t \, = \,0 }^{T-1} \rbr{ \big(\,{\lambda^{(t)} \,-\, \eta_2\, (V_g^{(t)}(\rho)-b) }\,\big)^2\,-\,(\lambda^{(t)})^2 }
					\\[0.2cm]
					&=&\displaystyle 2\eta_2 \sum_{t\,=\,0}^{T-1} \lambda^{(t)} \big(\,b\,-\,V_g^{(t)}(\rho)\,\big)
					\,+\,
					\eta_2^2 \sum_{t\,=\,0}^{T-1} \big(\,V_g^{(t)}(\rho)\,-\,b\,\big)^2
					\\[0.2cm]
					&\overset{(b)}{\leq}&\displaystyle 2\eta_2 \sum_{t\,=\,0}^{T-1} \lambda^{(t)} \big(\,{V_g^\star(\rho)\,-\,V_g^{(t)}(\rho)}\,\big) 
					\,+\, 
					\frac{\eta_2^2\, T}{(1-\gamma)^2}
				\end{array}
			\end{equation}
			where $(a)$ holds  because of the projection $\mathcal{P}_\Lambda$, $(b)$ is because of the feasibility of an optimal policy $\pi^\star$: $V_g^\star (\rho) \geq b$, and $|{V_g^{(t)}(\rho)-b}|\leq {1}/(1-\gamma)$. Hence,
			\begin{equation}\label{eq.lambda.clean}
				- \,\frac{1}{T}\sum_{t\,=\,0}^{T-1} \lambda^{(t)} \big(\,{V_g^\star(\rho) \,-\, V_g^{(t)}(\rho)}\,\big)
				\; \leq \; 
				\frac{\eta_2 }{2(1-\gamma)^2}.
			\end{equation}
		\end{subequations}
		To obtain the optimality gap bound, we now substitute~\eqref{eq.lambda.clean} into~\eqref{eq.dualkey}, apply $D_{\text{KL}}\rbr{p\,\Vert\,q}\leq \log|A|$ for $p\in \Delta_A$ and $q=\text{Unif}_A$, and take $\eta_1 =2\log|A|$ and $\eta_2={2(1-\gamma)}/{\sqrt{T}}$.
		
		\vspace*{0.15cm}
		\noindent\textbf{Bounding the constraint violation}. For any $\lambda\in \big[\,0,\, {2}/((1-\gamma)\xi)\,\big]$, from the dual update in~\eqref{eq.NPGAGD.clean}, we have
		\[
		\begin{array}{rcl}
			\vert\lambda^{(t+1)} \,-\, \lambda\vert^2 
			&\overset{(a)}{\leq}& \big\vert{{\lambda^{(t)} \,-\, \eta_2\, \big(\,{V_g^{(t)}(\rho)\,-\,b }}\,\big)  \,-\,\lambda}\big\vert^2
			\\[0.2cm]
			&=& \displaystyle\big\vert{{\lambda^{(t)}  -\lambda}}\big\vert^2 
			\,  -  \,  
			2\eta_2 \big(\, {V_g^{(t)}(\rho)\, -\, b }\, \big)\big(\, {\lambda^{(t)}  \, -\, \lambda}\, \big)  
			\,+ \,
			 \eta_2^2\, \big(\, {V_g^{(t)}(\rho)\, -\, b }\, \big)^2
			\\[0.2cm]
			&\overset{(b)}{\leq}&\displaystyle \big\vert{{\lambda^{(t)}  \, -\, \lambda}}\big\vert^2 \, - \, 
			2\eta_2 \big(\, {V_g^{(t)}(\rho)\, -\, b }\, \big)\big(\, {\lambda^{(t)}  \, -\, \lambda}\, \big) 
			\, + \, 
			\frac{\eta_2^2}{(1-\gamma)^2}
		\end{array}
		\]
		where $(a)$ is because of the non-expansiveness of projection $\calP_\Lambda$ and $(b)$ is because of $({V_g^{(t)}(\rho)-b })^2\leq{1}/{(1-\gamma)^2}$. Averaging the above inequality over $t=0,\ldots,T-1$ yields
		\[
		0
		\; \leq \;
		\frac{1}{T}\, \vert\lambda^{(T)} - \lambda\vert^2 
		\; \leq \; 
		\frac{1}{T}\, \vert{{\lambda^{(0)}  -\lambda}}\vert^2 
		\, -\, 
		\frac{2\eta_2}{T}\sum_{t\,=\,0}^{T-1} \big(\,V_g^{(t)}(\rho)\,-\,b\, \big)\big(\,\lambda^{(t)}  \,-\,\lambda\,\big) 
		\, + \, 
		\frac{\eta_2^2}{(1-\gamma)^2},
		\]
		which implies
		\begin{equation}\label{eq.removelamda}
			\frac{1}{T}\sum_{t\,=\,0}^{T-1} \big(\,V_g^{(t)}(\rho) \,-\, b \,\big) \big(\,\lambda^{(t)}  \,-\, \lambda\, \big)
			\; \leq \; 
			\frac{1}{2\eta_2T}\big\vert{{\lambda^{(0)}  -\lambda}}\big\vert^2 
			\, + \, 
			\frac{\eta_2}{2(1-\gamma)^2}.
		\end{equation}	
		We now add~\eqref{eq.removelamda} to~\eqref{eq.dualkey} on both sides of the inequality, and utilize $V_g^{\star}(\rho) \geq b$ to obtain
		\begin{equation}\label{eq.choose_lambda}
			\begin{array}{rcl}
				&& \!\!\!\! \!\!\!\! \!\! \displaystyle \frac{1}{T} \sum_{t\,=\,0}^{T-1} \big( \, V_r^\star(\rho) \,-\, V_r^{(t)}(\rho) \,\big) 
				\,+\, 
				\frac{\lambda}{T} \sum_{t\,=\,0}^{T-1} \big(\, b\,-\, V_g^{(t)}(\rho)\,\big)
				\\[0.2cm]
				& \leq &
				\displaystyle \frac{\log|A|}{\eta_1 T} 
				\,+\,
				\frac{1}{(1-\gamma)^2 T} 
				\,+\,
				\frac{2\eta_2}{(1-\gamma)^3}
				\,+\,
				\frac{1}{2\eta_2T}\big\vert{{\lambda^{(0)}  \, -\, \lambda}}\big\vert^2 
				\,+\,
				\frac{\eta_2}{2(1-\gamma)^2}.
			\end{array}
		\end{equation}
		Taking $\lambda = {2}/((1-\gamma)\xi)$ when ${\sum_{t\,=\,0}^{T-1} \big({b -V_g^{(t)}(\rho)}}\big)\geq 0$ and $\lambda=0$ otherwise, we obtain
		\[
		\begin{array}{rcl}
			&& \!\!\!\! \!\!\!\! \!\!\displaystyle V_r^\star(\rho) 
			\, - \, 
			\frac{1}{T} \sum_{t\,=\,0}^{T-1}{ V_r^{(t)}(\rho) } 
			\, + \, 
			\frac{2}{(1-\gamma)\xi} \sbr{\, b \, -\,  \frac{1}{T} \sum_{t\,=\,0}^{T-1} {V_g^{(t)}(\rho)}\, }_+
			\\[0.2cm]
			& \leq & \displaystyle\frac{\log|A|}{\eta_1 T} 
			\,+\,
			\frac{1}{(1-\gamma)^2 T} 
			\,+\,
			\frac{2\eta_2}{(1-\gamma)^3}
			\,+\,
			\frac{2}{\eta_2(1-\gamma)^2\xi^2 T}
			\,+\,
			\frac{\eta_2}{2(1-\gamma)^2}.
		\end{array}
		\]
		Note that both $V_r^{(t)}(\rho)$ and $V_g^{(t)}(\rho)$ can be expressed as linear functions in the same occupancy measure~\cite[Chapter~10]{altman1999constrained} that is induced by the policy $\pi^{(t)}$ and the transition $P(s'\,\vert\,s,a)$. The convexity of the set of occupancy measures shows that the average of $T$ occupancy measures is an occupancy measure that produces a policy $\pi'$ with values $V_r^{\pi'}$ and $V_g^{\pi'}$. Hence, there exists a policy $\pi'$ such that $V_r^{\pi'}(\rho)=\frac{1}{T} \sum_{t\,=\,0}^{T-1}{ V_r^{(t)}(\rho) }$ and $V_g^{\pi'}(\rho)=\frac{1}{T} \sum_{t\,=\,0}^{T-1}{ V_g^{(t)}(\rho) }$. Thus, 
		\[
		\begin{array}{rcl}
			&& \!\!\!\! \!\!\!\! \!\! \displaystyle V_r^\star(\rho)
			\,-\, 
			V_r^{\pi'}(\rho) 
			\,+\, 
			\frac{2}{(1-\gamma)\xi} \sbr{\, b \, -\,   {V_g^{\pi'}(\rho)}\, }_+
			\\[0.2cm]
			&\leq& \displaystyle\frac{\log|A|}{\eta_1 T} 
			\,+\,
			\frac{1}{(1-\gamma)^2 T} 
			\,+\,
			\frac{2\eta_2}{(1-\gamma)^3}
			\,+\,
			\frac{2}{\eta_2(1-\gamma)^2\xi^2 T}
			\,+\,
			\frac{\eta_2}{2(1-\gamma)^2}.
		\end{array}
		\]
        From Lemma~\ref{lem.duality} (ii), we have $\lambda^\star \leq 1/((1-\gamma)\xi)$. Application of Lemma~\ref{lem.constraint} with $C= {2}/((1-\gamma)\xi)$ yields
		\[
		\sbr{\, b \, -\,   {V_g^{\pi'}(\rho)} \,}_+
		\; \leq \; 
		\frac{\xi \log|A|}{\eta_1 T}
		\,+\,
		\frac{\xi}{(1-\gamma) T} 
		\,+\,
		\frac{2\eta_2\xi}{(1-\gamma)^2}
		\,+\,
		\frac{2}{\eta_2(1-\gamma)\xi T}
		\,+\,
		\frac{\eta_2\xi}{2(1-\gamma)}
		\]
		which leads to our constraint violation bound if we further utilize $\frac{1}{T} \sum_{t\,=\,0}^{T-1} \big({b -V_g^{(t)}(\rho)}\big)=b -  {V_g^{\pi'}(\rho)}$, $\eta_1 =2\log|A|$, and $\eta_2 = {2(1-\gamma)}/{\sqrt{T}}$.
	\end{proof}
	
	\subsection{Zero constraint violation}
	\label{subsec.softmax.zero.violation}
	
	In practice, it is natural to employ a conservative constraint $V_{g}^{\pi}(\rho)\geq b+\delta$ for some $\delta>0$ in Problem~\eqref{eq.cmdp}. When our desired accuracy $\epsilon$ is small enough, there exists some $\delta$ for the algorithm~\eqref{eq.NPG-PD-softmax} to get zero constraint violation on average.  
	
	\begin{corollary}[Zero constraint violation: softmax policy parametrization]\label{thm.convergence.softmax.zeroviolation}
		Let Assumption~\ref{as.slater} hold for $\xi>0$ and let us fix $\rho\in\Delta_S$ and replace the constraint of Problem~\eqref{eq.cmdp} by $V_g^{\pi}(\rho) \geq \bar{b}$, where $\bar{b}\DefinedAs b+\delta$ for some $\delta>0$.
		For $\epsilon<{\xi}/{2}$, there exists a $\delta = \Theta(\epsilon)$ such that 
		if we choose $T = \Omega(1/\epsilon^2)$, $\eta_1 = 2 \log|A|$, and $\eta_2=2(1-\gamma)/{\sqrt{T}}$, then the iterates $\{\pi^{(t)}\}_{t=0}^{T-1}$ generated by the algorithm~\eqref{eq.NPG-PD-softmax} satisfy
		\[
		\begin{array}{rcl}
			\text{\normalfont(Optimality gap)}\;\; 
			\displaystyle 
			\frac{1}{T} \sum_{t\,=\,0}^{T-1} \big(\,V_r^\star(\rho) \,-\, V_r^{(t)}(\rho) \,\big)
			& = &
			\displaystyle
			O(\epsilon)
			\\[0.4cm]
			\text{\normalfont(Constraint violation)}\;\; 
			\displaystyle 
			\sbr{\,
				\frac{1}{T}\sum_{t \, = \,0 }^{T-1} \big(\,b\,-\,V_g^{(t)}(\rho)\,\big)
				\,}_+  
			& \leq & 
			0.
		\end{array}
		\]
	\end{corollary}
	
	\begin{proof}
		The proof idea is similar to the one used in the proof of Theorem~\ref{thm.convergence.softmax}. Using the new constraint $V_g^{\pi}(\rho) \geq \bar{b}$, Problem~\eqref{eq.cmdp} satisfies Assumption~\ref{as.slater} for $\bar\xi \DefinedAs \xi - \delta$ where $\delta<\xi$, and there exists an optimal policy $\bar\pi^{\star}$. Without loss of generality, by restricting $\delta<{\xi}/{2}$, we can replace $\Lambda$ by $\bar\Lambda \DefinedAs [\,0, 4/((1-\gamma)\xi)\,]$, which contains $[\,0, 2/((1-\gamma)\bar\xi)\,]$ for any such a $\bar{\xi}$. Thus, we can apply the NPG-PD algorithm~\eqref{eq.NPGAGD} to this conservative problem using the projection set $\bar\Lambda$. It is straightforward to check that Lemma~\ref{lem.average.performance} holds for $V_r^{\bar\pi^\star}(\rho)$ and $V_g^{\bar\pi^\star}(\rho)$. Thus, bounding of the optimality gap in the proof of Theorem~\ref{thm.convergence.softmax} proves that after $T= \Omega({1}/{\epsilon^2})$ iterations, 
		\begin{equation}\label{eq.optimality gap relaxed}
			\displaystyle 
			\frac{1}{T} \sum_{t\,=\,0}^{T-1} \big(\,V_r^{\bar\pi^\star}(\rho) \,-\, V_r^{(t)}(\rho) \,\big)
			\; = \;
			\displaystyle
			O(\epsilon).
		\end{equation}
		Let $q^\star$ and $\bar{q}^\star$ be the occupancy measures induced by the policies $\pi^\star$ and $\bar{\pi}^\star$, respectively. In the occupancy measure space, Problem~\eqref{eq.cmdp} becomes a linear program and, thus, $V_r^{\pi^\star}(\rho) = \langle r, q^\star \rangle$ and $V_r^{\bar\pi^\star}(\rho) = \langle r, \bar q^\star \rangle$. By the continuity of optimal objective function in convex optimization~\citep{terazono2015continuity}, $|V_r^{\pi^\star}(\rho) - V_r^{\bar\pi^\star}(\rho)| \leq {2\epsilon}/((1-\gamma)\xi)$ for $\delta = \epsilon$. Therefore,  we can replace $V_r^{\bar\pi^\star}(\rho)$ in~\eqref{eq.optimality gap relaxed} by $V_r^{\star}(\rho)$ to bound the optimality gap by the same desired accuracy $\epsilon$ up to some problem-dependent constant. 
		
		To establish the bound on constraint violation, the key change begins with~\eqref{eq.choose_lambda}. Since we use $\bar b = b+\delta$ and $V_r^{\bar\pi^\star}(\rho)$, the right-hand side of~\eqref{eq.choose_lambda} contains an extra term $ 2\epsilon/((1-\gamma)\xi)-\lambda \delta$. Similarly, there are two options for selecting $\lambda$:  $\lambda = {4}/((1-\gamma)\xi)$ when ${\sum_{t\,=\,0}^{T-1} \big({b -V_g^{(t)}(\rho)}}\big)\geq 0$ and $\lambda=0$ otherwise. In the first case, if we set $\delta=\epsilon$, then the extra term $- 2\epsilon/((1-\gamma)\xi)$ cancels the error $O(1/\sqrt{T})$ for $T = \Omega(1/\epsilon^2)$, concluding zero constraint violation according to Lemma~\ref{lem.constraint}. On the other hand, the second case is exactly the zero constraint violation.
	\end{proof}
	
	\section{Function approximation: convergence rate and optimality}
	\label{sec.fa}
	
	Let us consider a general form of the NPG-PD algorithm~\eqref{eq.NPGAGD}:
	\begin{equation}\label{eq.NPGAGD.primalapproximation}
		\begin{array}{rcl} 
			\theta^{(t+1)} 
			& = & 
			\theta^{(t)} 
			\,+\,
			\dfrac{\eta_1}{1-\gamma}\, w^{(t)} 
			\\[0.15cm]
			\lambda^{(t+1)} 
			& = & 
			\calP_\Lambda  \left(\lambda^{(t)} 
			\, - \, 
			\eta_2\, \big(\, V_g^{(t)}(\rho)\,-\,b\,\big)\right)
		\end{array}
	\end{equation}
	where $w^{(t)}/(1-\gamma)$ denotes either the exact natural policy gradient or its sample-based approximation. For a general policy class, $ \{\pi_\theta\,\vert\,\theta\in\Theta\}$, with the parameter space $\Theta\subset\mathbb{R}^d$, the strong duality in Lemma~\ref{lem.duality} does not necessarily hold and our analysis of Section~\ref{sec.softmax} does not apply directly. Let the parametric dual objective function $V_D^{\lambda_\theta} (\rho) \DefinedAs \maximize_{\theta \,\in\, \Theta}  V_L^{\pi_\theta,\lambda}(\rho)$ be minimized at the optimal dual variable $\lambda_\theta^\star$. Under the Slater condition of  Assumption~\ref{as.slater}, the parametrization gap~\cite[Theorem~2]{paternain2019constrained} is determined by
	\[
	V_r^{\pi^\star} (\rho)
	\; = \;
	V_D^{\lambda^\star} (\rho)
	\; \geq \;
	V_D^{\lambda_\theta^\star} (\rho)
	\; \geq \;
	V_r^{\pi^\star} (\rho) \, - \, M \epsilon_{\pi}
	\]
	where $\epsilon_\pi \DefinedAs \max_s \Vert{ \pi(\cdot\,\vert\,s)-\pi_\theta(\cdot\,\vert\,s)}\Vert_1$ is the policy approximation error and $M>0$ is a problem-dependent constant. Application of the item (ii) in Lemma~\ref{lem.duality} to the set of all optimal dual variables $\lambda_\theta^\star$ yields $\lambda_\theta^\star\in[\,0, 2/((1-\gamma)\xi)\,]$ and,  thus, $\Lambda= [\,0, 2/((1-\gamma)\xi)\,]$.
	
	To quantify the error caused by the restricted policy parametrization, let us first generalize NPG. For a distribution over state-action pair $\nu \in \Delta_{S\times A}$, we introduce the \emph{compatible function approximation error}~\citep{kakade2002natural} as the following regression objective:
	\begin{equation}\nonumber
		E^\nu (w;\theta,\lambda) 
		\; \DefinedAs \;
		\mathbb{E}_{(s,a)\,\sim\,\nu} \left[\,  \left( A_L^{\theta,\lambda}(s,a) - w^\top \nabla_\theta \log \pi_\theta(a\,\vert\,s)  \right)^2 \,\right]
	\end{equation}
	where $A_L^{\theta,\lambda}(s,a) \DefinedAs A_r^{\theta}(s,a)+\lambda A_g^{\theta}(s,a)$. 
	We can view NPG in~\eqref{eq.NPGAGD} as a minimizer of $E^\nu(w;\theta,\lambda)$ for $\nu(s,a) = d_\rho^{\pi_\theta}(s) \pi_\theta(a\,\vert\,s)$:
	\begin{equation}\label{eq.npg}
		(1-\gamma) F^\dagger_\rho(\theta) \nabla_\theta V_L^{\theta,\lambda} (\rho) 
		\; \in \;
		\argmin_{w}\; E^\nu (w;\theta,\lambda). 
	\end{equation}
	Expression~\eqref{eq.npg} follows from the first-order optimality condition and the use of $\nabla_\theta V_L^{\theta,\lambda} (\rho) \DefinedAs \nabla_\theta V_r^{\theta} (\rho)+\lambda \nabla_\theta V_g^{\theta} (\rho)$ allows us to rewrite~\eqref{eq.npg} as a linear combination of 
	\begin{equation}\label{eq.npgs}
		(1-\gamma)
		F^\dagger_\rho(\theta) \nabla_\theta V_\diamond^{\theta} (\rho) 
		\; \in \;
		\argmin_{w_\diamond}\; E_\diamond^\nu (w_\diamond;\theta)
	\end{equation}
	where $\diamond$ denotes $r$ or $g$, and the compatible function approximation error $E_\diamond^\nu (w_\diamond;\theta)$ reads
	\begin{equation}\label{eq.compatible}
		E_\diamond^\nu (w_\diamond;\theta) 
		\; \DefinedAs \;
		\mathbb{E}_{(s,a)\,\sim\,\nu} \left[\,  \left( A_\diamond^{\theta}(s,a) - w_\diamond^\top \nabla_\theta \log \pi_\theta(a\,\vert\,s)  \right)^2 \,\right].
	\end{equation} 
    Let the minimal error be $E_{\diamond,\star}^\nu \DefinedAs \minimize_{w_\diamond}   E_\diamond^\nu (w_\diamond;\theta)$.
    
	When the compatible function approximation error is zero, the global convergence follows from Theorem~\ref{thm.convergence.softmax}. However, this is not the case for a general policy class because it may not include all possible policies (e.g., if we take $d\ll |S||A|$ for tabular constrained MDPs). The intuition behind {\em compatibility\/} is that any minimizer of $E_\diamond^\nu (w_\diamond;\theta)$ can be used as the NPG direction without affecting the global convergence property; also see more discussions in~\cite{kakade2002natural,sutton2000policy,agarwal2021optimality}.
	
	Since the state-action measure $\nu$ of some feasible comparison policy $\pi$ is not known, we introduce an exploratory initial distribution $\nu_0$ over state-action pairs and define a state-action visitation distribution $\nu_{\nu_0}^\pi$ of a policy $\pi$ as 
	\[
	\nu_{\nu_0}^\pi (s,a) 
	\; = \;
	(1-\gamma) \mathbb{E}_{(s_0,a_0)\,\sim\,\nu_0}  \left[ \, \sum_{t\,=\,0}^{\infty} \gamma^t P^\pi \rbr{s_t=s, a_t=a\,\vert\,s_0, a_0} \, \right]
	\]
	where $P^\pi \rbr{s_t=s, a_t=a\,\vert\,s_0, a_0}$ is the probability of visiting a state-action pair $(s, a)$ under policy $\pi$ for an initial state-action pair $(s_0, a_0)$. Whenever clear from context, we use $\nu^{(t)}$ to denote $\nu_{\nu_0}^{\pi^{(t)}}$ for notational convenience. When the minimizer is computed exactly, we can update $w^{(t)}$ in~\eqref{eq.NPGAGD.primalapproximation} using $w^{(t)} = w_r^{(t)} + \lambda^{(t)} w_g^{(t)}$, where $w_r^{(t)}$ and $w_g^{(t)}$ are given by
	\begin{equation}\label{eq.NPGAGD.primal.general}
		w_\diamond^{(t)}
		\; \in \;
		\argmin_{w_\diamond} \; E_\diamond^{\nu^{(t)}} \big(w_\diamond;\theta^{(t)}\big).
	\end{equation} 
	Even though the exact computation of a minimizer in~\eqref{eq.NPGAGD.primal.general} may not be feasible, we can use a sample-based algorithm to approximately solve its  empirical version. By characterizing the errors that result from the sample-based solutions and from the function approximation, we next prove the convergence of the algorithm~\eqref{eq.NPGAGD.primalapproximation} for the log-linear and the general smooth policy classes.
	
	\subsection{Log-linear policy class}
	\label{sec.fa.loglinear}
	
	We first consider the policies $\pi_\theta$ in the log-linear class~\eqref{eq.loglinear}, with the feature maps $\phi_{s,a}\in \mathbb{R}^d$. In this case, the gradient $\nabla_\theta \log \pi_\theta(a\,\vert\,s)$ becomes a shifted version of the feature $\phi_{s,a}$:
	\begin{equation}\label{eq.policyshift}
		\nabla_\theta \log \pi_\theta(a\,\vert\,s)
		\; = \;
		\phi_{s,a} 
		\, - \, \mathbb{E}_{a'\,\sim\,\pi_\theta(\cdot\,\vert\,s)}  [\, \phi_{s,a'} \, ]
		\; \AsDefined \;
		\bar\phi_{s,a}.
	\end{equation}
	Thus, the compatible function approximation error~\eqref{eq.compatible}
	captures how well the linear function $\theta^\top \bar \phi_{s,a}$ approximates the advantage function $A_r^{\theta}(s,a)$ or $A_g^{\theta}(s,a)$ under the state-action distribution $\nu$. We also introduce the compatible function approximation error with respect to the state-action value function $Q_\diamond^{\theta}(s,a)$:
	\[
	\mathcal{E}_\diamond^\nu (w_\diamond ; \theta)
	\; \DefinedAs \;	
	\mathbb{E}_{(s,a)\,\sim\,\nu} 
	\left[\,  \big( \,Q_\diamond^{\theta}(s,a) \, - \, w_\diamond^\top \phi_{s,a}  \,\big)^2 \,\right].
	\]
	When there are no compatible function approximation errors, the log-linear policy update in~\eqref{eq.NPGAGD.primalapproximation} for $w^{(t)}$ that is determined by~\eqref{eq.NPGAGD.primal.general}  is given by $w^{(t)} = w_r^{(t)}+\lambda^{(t)} w_g^{(t)}$, $w_\diamond^{(t)} \in \argmin_{w_\diamond} \mathcal{E}_\diamond^{\nu^{(t)}} \big(w_\diamond ; \theta^{(t)}\big)$ for $\diamond = r$ or $g$, where $\nu^{(t)}(s,a) = d_\rho^{(t)}(s) \pi_\theta^{(t)}(a\,\vert\,s)$ is an on-policy state-action visitation distribution. This is because the softmax function is invariant to any terms that are independent of the action.
	
	Let us consider an approximate solution
	\begin{equation}\label{eq.NPGAGD.primal.app}
		w_\diamond^{(t)}  
		\; \approx \;
		\argmin_{\norm{w_\diamond}_2\,\leq\, W} \; \mathcal{E}_\diamond^{\nu^{(t)}} \big(w_\diamond; \theta^{(t)}\big)
	\end{equation}
	where the constraint with a norm bound  $W>0$ can be viewed as an $L_2$-regularization. We restrict the domain to make the approximate solution well-defined even when it is not well-posed, which is similar to imposing an $L_2$-regularization in practice. Let an exact minimizer be $w_{\diamond,\star}^{(t)} \in \argmin_{\norm{w_\diamond}_2\,\leq\, W} \mathcal{E}_\diamond^{\nu^{(t)}} (w_\diamond; \theta^{(t)})$. Fixing a state-action distribution $\nu^{(t)}$, the estimation error in $w_\diamond^{(t)}$ arises from the discrepancy between $w_\diamond^{(t)}$ and $w_{\diamond,\star}^{(t)}$, which comes from the randomness in a sample-based optimization algorithm and the mismatch between the linear function and the true state-action value function. We represent the estimation error as
	\[
	\mathcal{E}_{\diamond,\text{\normalfont est}}^{(t)}
	\;
	\DefinedAs
	\;
	\mathbb{E} 
	\left[\,
	\mathcal{E}_\diamond^{\nu^{(t)}} \big(w_\diamond^{(t)}; \theta^{(t)}\big)
	\, - \,
	\mathcal{E}_\diamond^{\nu^{(t)}} \big(w_{\diamond,\star}^{(t)}; \theta^{(t)}\big)
	\,\right]
	\]
	where the expectation $\mathbb{E}$ is taken over the randomness of the approximate algorithm used to solve~\eqref{eq.NPGAGD.primal.app}. The estimation error simplifies when the state-action value function is linear~\citep{ding2022policy}.
	
	Note that the state-action distribution $\nu^{(t)}$ is on-policy. To characterize the effect of distribution shift on $w_{\diamond,\star}^{(t)}$, we first introduce some notation. 
	We represent a fixed distribution over state-action pairs $(s,a)$ by 
	\begin{equation}
		\nu^\star(s,a)  
		\; \DefinedAs \; 
		d_\rho^{\pi^\star}(s) \circ \text{Unif}_A(a).
		\label{eq.fixed}
	\end{equation}
	The fixed distribution $\nu^\star$ samples a state from $d_\rho^{\pi^\star}(s)$ and an action uniformly from $\text{Unif}_A(a)$. We characterize the error in $w_{\diamond,\star}^{(t)}$ that arises from the distribution shift via the transfer error
	\[
	\mathcal{E}_{\diamond,\text{\normalfont bias}}^{(t)}
	\;
	\DefinedAs
	\;
	\mathbb{E} 
	\left[\,
	\mathcal{E}_\diamond^{\nu^{\star}} \big(w_{\diamond,\star}^{(t)}; \theta^{(t)}\big)
	\,\right].
	\]
	The transfer error characterizes the expressiveness of function approximation that is affected by the feature maps $\phi_{s,a}\in \mathbb{R}^d$ and the quality of the exact minimizer $w_{\diamond,\star}^{(t)}$.

	\begin{assumption}[Estimation error and transfer error]\label{as.errors}
		Both the estimation error and the transfer error are bounded, i.e., $\mathcal{E}_{\diamond,\text{\normalfont est}}^{(t)}\leq \epsilon_{\text{\normalfont  est}}$ and $\mathcal{E}_{\diamond,\text{\normalfont bias}}^{(t)} \leq \epsilon_{\text{\normalfont  bias}}$ for all $t\geq 0$, where $\diamond$ denotes either $r$ or $g$.
	\end{assumption}
	
	We also point out that it is possible to remove the domain restriction in~\eqref{eq.NPGAGD.primal.app} when some regularity assumptions on the feature maps are made in the sample-based algorithm~\citep{bach2013non}. Let $\bar{w}_{\diamond,\star}^{(t)}
	\in
	\argmin_{w_\diamond} \mathcal{E}_{\diamond}^{\nu^{(t)}} \big(w_\diamond; \theta^{(t)}\big)$. Since the expressiveness of function approximation is captured by the transfer error, the gap between the exact minimizers $w_{\diamond,\star}^{(t)}$ and $\bar w_{\diamond,\star}^{(t)}$ is contained in the transfer error.
	
	When we apply a sample-based algorithm to~\eqref{eq.NPGAGD.primal.app}, it is standard to have $\epsilon_{\text{\normalfont  est}} = O(1/\sqrt{K})$, where $K$ is the number of samples; e.g., see~\citet[Theorem~14.8]{shalev2014understanding}. A special case is the exact tabular softmax policy parametrization for which $\epsilon_{\text{\normalfont  bias}} = \epsilon_{\text{\normalfont  est}} = 0$, since the features  $\phi_{s,a}\in\mathbb{R}^d$ now reduce to indicator functions of the state/action spaces. 
	
	For any state-action distribution $\nu$, we define $\Sigma_{\nu} \DefinedAs \mathbb{E}_{(s,a)\,\sim\,\nu} \left[\, \phi_{s,a} \phi_{s,a}^\top \,\right]$, and following~\cite[Assumption 6.2]{agarwal2021optimality}, to compare $\nu$ with $\nu^\star$, we introduce the notion of \emph{relative condition number}:
	\[
	\kappa  
	\; \DefinedAs \;
	\sup_{w\,\in\,\mathbb{R}^d} \;
	\frac{w^\top \Sigma_{\nu^\star} w}{w^\top \Sigma_{\nu_0}w}. 
	\]
	
	\begin{assumption}[Bounded relative condition number]\label{as.condition}
		For an initial state-action distribution $\nu_0$ and $\nu^\star$ determined by~\eqref{eq.fixed}, the relative condition number $\kappa$ is finite.
	\end{assumption}
	
	With the estimation error $\epsilon_{\text{\normalfont  est}}$, the transfer error $\epsilon_{\text{\normalfont  bias}}$, and the relative condition number $\kappa$ in place, in Theorem~\ref{thm.convergence.loglinear} we establish convergence guarantees for the algorithm~\eqref{eq.NPGAGD.primalapproximation} using the approximate update~\eqref{eq.NPGAGD.primal.app}. Even though we set $\theta^{(0)}=0$ and $\lambda^{(0)}=0$ in the proof of Theorem~\ref{thm.convergence.loglinear}, global convergence can be established for arbitrary initial conditions. 
	
	\begin{theorem}[Convergence and optimality: log-linear policy parametrization]
		\label{thm.convergence.loglinear}
		Let Assumption~\ref{as.slater} hold for $\xi>0$ and let us fix a state distribution $\rho$ and a state-action distribution $\nu_0$. If the iterates $\{(\theta^{(t)},\lambda^{(t)})\}_{t=0}^{T-1}$ generated by the algorithm~\eqref{eq.NPGAGD.primalapproximation} using~\eqref{eq.NPGAGD.primal.app} with $\norm{\phi_{s,a}}\leq B$ and $\eta_1=\eta_2={1}/{\sqrt{T}}$ satisfy Assumptions~\ref{as.errors} and~\ref{as.condition}, then
		\[
		\begin{array}{rcl}
			\displaystyle
			\mathbb{E} 
			\sbr{\,
				\frac{1}{T} \sum_{t\,=\,0}^{T-1} \big(\,V_r^\star(\rho) \,-\,V_r^{(t)}(\rho)\,\big)
				\,}
			& \leq &  \displaystyle
			\frac{C_3}{(1-\gamma)^5}\frac{1}{\sqrt{T}} \,+\, \frac{2+4/\xi}{(1-\gamma)^2} \left( \sqrt{ |A|\, \epsilon_{\text{\normalfont  bias}}} + \sqrt{\! \frac{\kappa\,|A|\, \epsilon_{\text{\normalfont  est}}}{1-\gamma} } \right)
			\\[0.4cm]
			\displaystyle
			\mathbb{E}
			\sbr{\,
				\frac{1}{T} \sum_{t\,=\,0}^{T-1} \big(\,b\,-\,V_g^{(t)}(\rho)\,\big)
				\,}_+
			& \leq & \displaystyle
			\frac{C_4}{(1-\gamma)^4}\frac{1}{\sqrt{T}} \,+\, \left( \frac{4+2\xi}{1-\gamma}\right) \left( \!\sqrt{ |A|\, \epsilon_{\text{\normalfont  bias}}} + \sqrt{ \frac{\kappa\,|A|\, \epsilon_{\text{\normalfont  est}}}{1-\gamma} } \right)
		\end{array}
		\]
		where $C_3 \DefinedAs 1+\log |A|+5B^2 W^2/\xi^2$ and $C_4 \DefinedAs (1+\log |A|+B^2 W^2)\xi + (2+4B^2 W^2)/\xi$.
	\end{theorem}
	
	Theorem~\ref{thm.convergence.loglinear} shows that, on average, the reward value function converges to its globally optimal value and that the constraint violation decays to zero (up to an estimation error $\epsilon_{\text{\normalfont  est}}$ and a transfer error $\epsilon_{\text{\normalfont  bias}}$). When $\epsilon_{\text{\normalfont  bias}} = \epsilon_{\text{\normalfont  est}} = 0$, the rate $(1/\sqrt{T},1/\sqrt{T})$ matches the result in Theorem~\ref{thm.convergence.softmax} for the exact tabular softmax case. Compared to~\cite[Theorem~2]{ding2020natural}, the improved rate $1/\sqrt{T}$ in the constraint violation benefits from a new regret-type primal-dual analysis in Section~\ref{subsec.loglinear.pf}, which departs from the previous drift analysis of constraint violation. In contrast to the optimality gap, the lower order of effective horizon $1/(1-\gamma)$ in the constraint violation yields a tighter error bound.
	
	\begin{remark}
		By a natural  error decomposition (as also used in \cite{agarwal2021optimality})
		\[
		\mathcal{E}_\diamond^{\nu^{(t)}} \big(w_{\diamond}^{(t)}; \theta^{(t)}\big) 
		\; = \;
		\mathcal{E}_\diamond^{\nu^{(t)}} \big(w_{\diamond}^{(t)}; \theta^{(t)}\big) 
		\, - \, 
		\mathcal{E}_\diamond^{\nu^{(t)}} \big(w_{\diamond,\star}^{(t)}; \theta^{(t)}\big) 
		\, + \,
		\mathcal{E}_\diamond^{\nu^{(t)}} \big(w_{\diamond,\star}^{(t)}; \theta^{(t)}\big),
		\] 
		the difference term is the standard estimation error that results from the discrepancy between $w_\diamond^{(t)}$ and $w_{\diamond,\star}^{(t)}$, and the last term characterizes the approximation error in $w_{\diamond,\star}^{(t)}$. In Corollary~\ref{cor.convergence.loglinear}, we repeat Theorem~\ref{thm.convergence.loglinear} in terms of an upper bound $\epsilon_{\text{\normalfont  approx}}$ on the approximation error
		\[
		\mathcal{E}_{\diamond,\text{\normalfont approx}}^{(t)}
		\;
		\DefinedAs
		\;
		\mathbb{E}
		\sbr{
			\,
			\mathcal{E}_\diamond^{\nu^{(t)}} \big(w_{\diamond,\star}^{(t)}; \theta^{(t)}\big)
			\,
		}.
		\]
		Since $\mathcal{E}_{\diamond,\text{\normalfont approx}}^{(t)}$ utilizes an on-policy state-action distribution $\nu^{(t)}$, the error bounds in Corollary~\ref{cor.convergence.loglinear} depend on the worst-case distribution mismatch coefficeint $\norm{\nu^\star/\nu_0}_\infty$. In contrast, application of estimation and transfer errors in Theorem~\ref{thm.convergence.loglinear} does not involve the distribution mismatch coefficient. Therefore, the error bounds in Theorem~\ref{thm.convergence.loglinear} are tighter than the ones in Corollary~\ref{cor.convergence.loglinear} that utilizes this natural error decomposition.
	\end{remark}
	
	\begin{corollary}[Convergence and optimality: log-linear policy parametrization]
		\label{cor.convergence.loglinear}
		Let Assumption~\ref{as.slater} hold for $\xi>0$ and let us fix a state distribution $\rho$ and a state-action distribution $\nu_0$. If the iterates $\{(\theta^{(t)},\lambda^{(t)})\}_{t=0}^{T-1}$ generated by the algorithm~\eqref{eq.NPGAGD.primalapproximation} using~\eqref{eq.NPGAGD.primal.app} with $\norm{\phi_{s,a}}\leq B$ and $\eta_1=\eta_2={1}/{\sqrt{T}}$ satisfy Assumption~\ref{as.errors} except for  $\mathcal{E}_{\diamond,\text{\normalfont bias}}^{(t)}$, Assumption~\ref{as.condition}, and additionally $\mathcal{E}_{\diamond,\text{\normalfont approx}}^{(t)} \leq \epsilon_{\text{\normalfont  approx}}$ (for both $\diamond = r$ and $g$), 
		then
		\[
		\begin{array}{rcl}
			\displaystyle
			\mathbb{E} 
			\sbr{\, 
				\frac{1}{T} \sum_{t\,=\,0}^{T-1} \big(\,V_r^\star(\rho) \,-\,V_r^{(t)}(\rho)\,\big)
				\,}
			& \!\!\leq\!\! &  \displaystyle
			\frac{C_3}{(1-\gamma)^5}\frac{1}{\sqrt{T}} 
			\,+\, 
			C_3' \left( \sqrt{ \frac{|A| \, \epsilon_{\text{\normalfont  approx}}}{1-\gamma} \norm{\frac{\nu^{\star}}{\nu_0}}_\infty } + \sqrt{ \frac{\kappa\, |A|\, \epsilon_{\text{\normalfont  est}}}{1-\gamma} } \right)
			\\[0.4cm]
			\displaystyle
			\mathbb{E}
			\sbr{\,
				\frac{1}{T} \sum_{t\,=\,0}^{T-1} \big(\,b\,-\,V_g^{(t)}(\rho)\,\big)
				\,}_+
			& \!\!\leq\!\! & \displaystyle
			\frac{C_4}{(1-\gamma)^4}\frac{1}{\sqrt{T}} 
			\,+\, 
			C_4' \left( \sqrt{ \frac{|A|\,\epsilon_{\text{\normalfont  approx}}}{1-\gamma} \norm{\frac{\nu^{\star}}{\nu_0}}_\infty } + \sqrt{ \frac{\kappa\, |A|\, \epsilon_{\text{\normalfont  est}}}{1-\gamma} } \right)
		\end{array}
		\]
		where $C_3 \DefinedAs 1+\log |A|+5B^2 W^2/\xi^2$, $C_4 \DefinedAs (1+\log |A|+B^2 W^2)\xi + (2+4B^2 W^2)/\xi$, $C_3' \DefinedAs ({2+4/\xi})/{(1-\gamma)^2}$, and $C_4' \DefinedAs (4+2\xi)/(1-\gamma)$.
	\end{corollary}
	\begin{proof}
		From the definitions of $\mathcal{E}_\diamond^{\nu^{\star}}$ and $\mathcal{E}_\diamond^{\nu^{(t)}}$, we have 
		\[
		\mathcal{E}_\diamond^{\nu^{\star}} \big(w_{\diamond,\star}^{(t)}; \theta^{(t)}\big)
		\; \leq \;
		\norm{\frac{\nu^{\star}}{\nu^{(t)}}}_\infty\mathcal{E}_\diamond^{\nu^{(t)}} \big(w_{\diamond,\star}^{(t)}; \theta^{(t)}\big)
		\; \leq \;
		\frac{1}{1-\gamma}
		\norm{\frac{\nu^{\star}}{\nu_0}}_\infty
		\mathcal{E}_\diamond^{\nu^{(t)}} \big(w_{\diamond,\star}^{(t)}; \theta^{(t)}\big)
		\]
		where the second inequality is because of $(1-\gamma)\nu_0 \leq \nu^{(t)}$. Thus, 
		\[
		\mathcal{E}_{\diamond,\text{\normalfont bias}}^{(t)} 
		\; \leq  \;
		\frac{1}{1-\gamma}
		\norm{\frac{\nu^{\star}}{\nu_0}}_\infty \mathcal{E}_{\diamond,\text{\normalfont approx}}^{(t)}
		\]
		which allows us to replace $\mathcal{E}_{\diamond,\text{\normalfont bias}}^{(t)}$ in the proof of Theorem~\ref{thm.convergence.loglinear} by $\mathcal{E}_{\diamond,\text{\normalfont approx}}^{(t)}$.
	\end{proof}
	
	\subsection{Proof of Theorem~\ref{thm.convergence.loglinear}} 
	\label{subsec.loglinear.pf}
	
	We provide a regret-type analysis for a general class of smooth policies that subsumes the log-linear policy class as a special case, in Lemma~\ref{lem.gap.violation}. Using the property of policy smoothness, we first generalize Lemma~\ref{lem.average.performance} to the function approximation setting. Then, we can utilize the function approximation error to contain the duality gap and characterize the regret and the constraint violation performance.
	
	\begin{lemma}[Regret/Violation lemma]\label{lem.gap.violation}
		Let Assumption~\ref{as.slater} hold for $\xi>0$, let us fix a state distribution $\rho$ and $T>0$, and let $\log \pi_\theta(a\,\vert\,s)$ be $\beta$-smooth in $\theta$  for any $(s,a)$. If the iterates $\{(\theta^{(t)},\lambda^{(t)})\}_{t=0}^{T-1}$ are generated by the algorithm~\eqref{eq.NPGAGD.primalapproximation} with $\theta^{(0)} = 0$, $\lambda^{(0)}=0$, $\eta_1=\eta_2=1/{\sqrt{T}}$, and $\Vert w_\diamond^{(t)}\Vert \leq W$, then
		\[
		\begin{array}{rcl}
			\displaystyle
			\frac{1}{T} \sum_{t\,=\,0}^{T-1} \big(\,V_r^\star(\rho) \,-\,V_r^{(t)}(\rho)\,\big)
			& \leq &
			\displaystyle
			\frac{C_3}{(1-\gamma)^5} \frac{1}{\sqrt{T}}
			\, + \,
			\sum_{t\,=\,0}^{T-1} \frac{\text{\normalfont err}_r^{(t)} (\pi^\star)}{(1-\gamma)T}
			\, + \,
			\sum_{t\,=\,0}^{T-1} \frac{2\times \text{\normalfont err}_g^{(t)} (\pi^\star)}{(1-\gamma)^2\xi T}
			\\[0.4cm]
			\displaystyle
			\sbr{\,
				\frac{1}{T} \sum_{t\,=\,0}^{T-1} \big(\,{b\,-\,V_g^{(t)}(\rho)}\,\big)
				\,}_+
			& \leq &
			\displaystyle
			\frac{C_4}{(1-\gamma)^4}\frac{1}{\sqrt T}
			\, + \,
			\sum_{t\,=\,0}^{T-1} \frac{\xi \times \text{\normalfont err}_r^{(t)} (\pi^\star) }{T}
			\, + \,
			\sum_{t\,=\,0}^{T-1} \frac{2 \times \text{\normalfont err}_g^{(t)}(\pi^\star)}{(1-\gamma) T}
		\end{array}
		\]
		where $C_3 \DefinedAs 1+\log |A|+5\beta W^2/\xi^2$, $C_4 \DefinedAs (1+\log |A|+\beta W^2)\xi + (2+4\beta W^2)/\xi$, and
		\[
		\text{\normalfont err}_\diamond^{(t)} (\pi)
		\; \DefinedAs \; 
		\left\vert\,
		\mathbb{E}_{s\,\sim\,d_\rho^\pi} \mathbb{E}_{a\,\sim\,\pi(\cdot\,\vert\,s)} \sbr{A_\diamond^{(t)}(s,a) - (w_\diamond^{(t)})^\top \nabla_\theta \log \pi_\theta^{(t)}(a\,\vert\,s)}
		\,\right\vert
		\]
		where $\diamond = r$ or $g$.	
	\end{lemma}
	\begin{proof}
		The smoothness of the log-linear policy in conjunction with an application of Taylor expansion to $\log \pi_\theta^{(t)}(a\,\vert\,s)$ yields
		\begin{equation}\label{eq.smoothLL}
			\log \frac{\pi_\theta^{(t)}(a\,\vert\,s) }{\pi_\theta^{(t+1)}(a\,\vert\,s) } 
			\, + \, \rbr{\theta^{(t+1)} -\theta^{(t)} }^\top \nabla_\theta \log \pi_\theta^{(t)}(a\,\vert\,s)   
			\; \leq \;
			\frac{\beta}{2} \left\Vert \theta^{(t+1)} -\theta^{(t)}  \right\Vert^2
		\end{equation}
		where $\theta^{(t+1)} - \theta^{(t)} = \eta_1 w^{(t)} / (1-\gamma)$. 
		Fixing $\pi$ and $\rho$, we use $d$ to denote $d_\rho^\pi$ to obtain
		\[
		\begin{array}{rcl}
			&& \!\!\!\! \!\!\!\! \!\! \mathbb{E}_{s\,\sim\,d} \sbr{\, D_\text{KL}\left(\pi(\cdot\,\vert\,s)\,\big\Vert\,\pi_\theta^{(t)}(\cdot\,\vert\,s)\right) \,-\, D_\text{KL}\left(\pi(\cdot\,\vert\,s)\,\big\Vert\,\pi_\theta^{(t+1)}(\cdot\,\vert\,s)\right)\, } 
			\\[0.2cm]
			& = & \displaystyle-~\mathbb{E}_{s\,\sim\,d} \mathbb{E}_{a\,\sim\,\pi(\cdot\,\vert\,s)} \sbr{\,\log \frac{\pi_\theta^{(t)}(a\,\vert\,s) }{\pi_\theta^{(t+1)}(a\,\vert\,s) }\,}
			\\[0.2cm]
			& \overset{(a)}{\geq} &\displaystyle \eta_1\mathbb{E}_{s\,\sim\,d} \mathbb{E}_{a\,\sim\,\pi(\cdot\,\vert\,s)} \sbr{\,\nabla_\theta \log \pi_\theta^{(t)}(a\,\vert\,s)\, w^{(t)}\,} 
			\,-\,
			\beta\frac{\eta_1^2}{2(1-\gamma)^2} \left\Vert w^{(t)}\right\Vert^2
			\\[0.2cm]
			& \overset{(b)}{=} & \displaystyle\eta_1\mathbb{E}_{s\,\sim\,d} \mathbb{E}_{a\,\sim\,\pi(\cdot\,\vert\,s)} \sbr{\,\nabla_\theta \log \pi_\theta^{(t)}(a\,\vert\,s)\, w_r^{(t)}\,}
			\\[0.2cm]
			&&\displaystyle +~\eta_1 \lambda^{(t)}\mathbb{E}_{s\,\sim\,d} \mathbb{E}_{a\,\sim\,\pi(\cdot\,\vert\,s)} \sbr{\,\nabla_\theta \log \pi_\theta^{(t)}(a\,\vert\,s)\, w_g^{(t)}\,} 
			\,-\,
			\beta\frac{\eta_1^2}{2(1-\gamma)^2} \left\Vert w^{(t)}\right\Vert^2
			\\[0.4cm]
			& = &\displaystyle \eta_1\mathbb{E}_{s\,\sim\,d} \mathbb{E}_{a\,\sim\,\pi(\cdot\,\vert\,s)} \sbr{\,A_r^{(t)}(s,a)\,}
			\,+\,
			\eta_1 \lambda^{(t)}\mathbb{E}_{s\,\sim\,d} \mathbb{E}_{a\,\sim\,\pi(\cdot\,\vert\,s)} \sbr{\,A_g^{(t)}(s,a) \,}
			\\[0.2cm]
			&&\displaystyle+~\eta_1\mathbb{E}_{s\,\sim\,d} \mathbb{E}_{a\,\sim\,\pi(\cdot\,\vert\,s)} \sbr{\,\nabla_\theta \log \pi_\theta^{(t)}(a\,\vert\,s)\, \big({w_r^{(t)}+\lambda^{(t)} w_g^{(t)}}\big) - \big({A_r^{(t)}(s,a)+\lambda^{(t)} A_g^{(t)}(s,a)}\big)\,}
			\\[0.2cm]
			&&\displaystyle-~\beta\frac{\eta_1^2}{(1-\gamma)^2} \Big(\left\Vert w_r^{(t)}\right\Vert^2 + \big(\lambda^{(t)}\big)^2 \left\Vert w_g^{(t)}\right\Vert^2\Big)
			\\[0.4cm]
			&\overset{(c)}{\geq}&\displaystyle \eta_1(1-\gamma) \big(\,{V_r^\pi(\rho) - V_r^{(t)}(\rho)}\, \big)
			\,+\,\eta_1(1-\gamma) \lambda^{(t)} \big(\,{ V_g^\pi(\rho) - V_g^{(t)}(\rho)}\,\big)
			\\[0.2cm]
			&&\displaystyle-~\eta_1  \text{\normalfont err}_r^{(t)} (\pi)
			\,-\,
			\eta_1 \lambda^{(t)}  \text{\normalfont err}_g^{(t)} (\pi)
			\,-\,
			\beta\frac{\eta_1^2 \, W^2}{(1-\gamma)^2} 
			\,-\,
			\beta\frac{\eta_1^2\, W^2}{(1-\gamma)^2}  \big(\lambda^{(t)}\big)^2
		\end{array}
		\]
		where $(a)$ is because of~\eqref{eq.smoothLL}, we use the update $w^{(t)} = w_r^{(t)}+\lambda^{(t)}w_g^{(t)}$ for a given $\lambda^{(t)}$ in $(b)$, and in $(c)$ we apply the performance difference lemma (see Remark~\ref{re: PDL}), the definitions of $\text{\normalfont err}_r^{(t)} (\pi)$ and $\text{\normalfont err}_g^{(t)} (\pi)$, and $ \Vert w_\diamond^{(t)}\Vert\leq W$.
		Rearrangement of the above inequality yields
		\[
		\begin{array}{rcl}
			&& \!\!\!\! \!\!\!\! \!\! V_r^\pi(\rho) \,-\, V_r^{(t)}(\rho) 
			\\[0.2cm]
			& \leq &\displaystyle \frac{1}{1-\gamma} \frac{1}{\eta_1} \mathbb{E}_{s\,\sim\,d} \sbr{\, D_\text{KL}\left(\pi(\cdot\,\vert\,s)\,\big\Vert\,\pi_\theta^{(t)}(\cdot\,\vert\,s)\right) \,-\, D_\text{KL}\left(\pi(\cdot\,\vert\,s)\,\big\Vert\,\pi_\theta^{(t+1)}(\cdot\,\vert\,s)\right) \,}  
			\\[0.4cm]
			&& \displaystyle +~\frac{1}{1-\gamma} \,\text{\normalfont err}_r^{(t)} (\pi)  
			\,+\, 
			\frac{2}{(1-\gamma)^2 \xi }\, \text{\normalfont err}_g^{(t)} (\pi)
			\,+\, 
			\beta\frac{\eta_1 W^2}{(1-\gamma)^3}  
			\,+\, 
			\beta\frac{4\eta_1 W^2}{(1-\gamma)^5\xi^2} 
			\\[0.4cm]
			&& \displaystyle
			-~ \lambda^{(t)} \big(\,{ V_g^\pi(\rho) - V_g^{(t)}(\rho)}\,\big)
		\end{array}
		\]
		where we utilize $0\leq \lambda^{(t)} \leq 2/((1-\gamma)\xi)$ from the dual update in~\eqref{eq.NPGAGD.primalapproximation}. 
		
		Averaging the inequality above over $t=0,1,\ldots,T-1$ yields
		\[
		\begin{array}{rcl}
			&& \!\!\!\! \!\!\!\! \!\!\displaystyle \frac{1}{T} \sum_{t\,=\,0}^{T-1} \big( \, V_r^\pi(\rho) \,-\, V_r^{(t)}(\rho) \,\big)
			\\[0.2cm]
			&\leq&\displaystyle \frac{1}{(1-\gamma)\eta_1 T}\sum_{t \, = \,0 }^{T-1}\mathbb{E}_{s\,\sim\,d} \sbr{\, D_\text{KL}\big(\pi(\cdot\,\vert\,s)\,\big\Vert\,\pi_\theta^{(t)}(\cdot\,\vert\,s)\big) \,-\, D_\text{KL}\left(\pi(\cdot\,\vert\,s)\,\big\Vert\,\pi_\theta^{(t+1)}(\cdot\,\vert\,s)\right) \,}  
			\\[0.2cm]
			&&\displaystyle +~\frac{1}{(1-\gamma)T}\sum_{t\,=\,0}^{T-1} \text{\normalfont err}_r^{(t)} (\pi)  
			\,+\, 
			\frac{2}{(1-\gamma)^2\xi T}\sum_{t\,=\,0}^{T-1}  \text{\normalfont err}_g^{(t)} (\pi)
			\,+\,
			\beta\frac{\eta_1W^2 }{(1-\gamma)^3} 
			\,+\, 
			\beta\frac{4\eta_1W^2}{(1-\gamma)^5\xi^2}  
			\\[0.2cm]
			&& \displaystyle 
			-~\frac{1}{T}\sum_{t\,=\,0}^{T-1}\lambda^{(t)} \big(\,{ V_g^\pi(\rho) - V_g^{(t)}(\rho)}\,\big)
		\end{array}
		\]
		which implies
		\[
		\begin{array}{rcl}
			&& \!\!\!\! \!\!\!\! \!\!\displaystyle \frac{1}{T} \sum_{t\,=\,0}^{T-1} \big( \, V_r^\pi(\rho) \,-\, V_r^{(t)}(\rho) \,\big)
			\\[0.2cm]
			&\leq&\displaystyle \frac{\log|A|}{(1-\gamma)\eta_1 T}
			\,+\,
			\frac{1}{(1-\gamma)T}\sum_{t\,=\,0}^{T-1} \text{\normalfont err}_r^{(t)} (\pi)
			\,+\,
			\frac{2}{(1-\gamma)^2\xi T}\sum_{t\,=\,0}^{T-1}  \text{\normalfont err}_g^{(t)} (\pi)
			\\[0.2cm]
			&&  \displaystyle+~\beta\frac{\eta_1W^2}{(1-\gamma)^3} 
			\,+\, 
			\beta\frac{4\eta_1W^2}{(1-\gamma)^5\xi^2} 
			\,+\, 
			\frac{1}{T}\sum_{t\,=\,0}^{T-1}\lambda^{(t)} \big({ V_g^\pi(\rho) - V_g^{(t)}(\rho)}\big).
		\end{array}
		\]
		If we choose the comparison policy $\pi  = \pi^\star$, then we have
		\begin{equation}\label{eq.keyone}
			\begin{array}{rcl}
				&& \!\!\!\! \!\!\!\! \!\! 
				\displaystyle
				\frac{1}{T} \sum_{t\,=\,0}^{T-1} \big(\,V_r^\star(\rho) \,-\, V_r^{(t)}(\rho) \,\big) \,+\, 
				\frac{1}{T}\sum_{t\,=\,0}^{T-1}\lambda^{(t)} \big(\,{ V_g^\star(\rho) \,-\,  V_g^{(t)}(\rho)}\,\big)
				\\[0.2cm]
				& \leq & \displaystyle
				\frac{\log|A|}{(1-\gamma)\eta_1 T}
				\,+\,
				\frac{1}{(1-\gamma)T}\sum_{t\,=\,0}^{T-1} \text{\normalfont err}_r^{(t)} (\pi^\star)
				\,+\,
				\frac{2}{(1-\gamma)^2\xi T}\sum_{t\,=\,0}^{T-1} \text{\normalfont err}_g^{(t)} (\pi^\star)
				\\[0.2cm]
				&& \displaystyle+~\beta\frac{\eta_1W^2}{(1-\gamma)^3} 
				\,+\, \beta\frac{4\eta_1W^2}{(1-\gamma)^5\xi^2}.
			\end{array}
		\end{equation}
		
		\noindent\textbf{Proving the first inequality}. By the same reasoning as in~\eqref{eq.lambda},
		\begin{subequations}
			\begin{equation}\label{eq.lambda.ap}
				\begin{array}{rcl}
					0
					\;\,\leq\;\,
					\rbr{\lambda^{(T)}}^2 &=& \displaystyle\sum_{t \, = \,0 }^{T-1} \big(\,{(\lambda^{(t+1)} )^2\,-\,(\lambda^{(t)})^2}\,\big)
					\\[0.2cm]		
					&\leq&\displaystyle 2\eta_2 \sum_{t\,=\,0}^{T-1} \lambda^{(t)} \big(\,b\, -\, V_g^{(t)}(\rho)\,\big)\,+\,\eta_2^2 \sum_{t\,=\,0}^{T-1} \big(\,V_g^{(t)} (\rho)\, -\, b\,\big)^2
					\\[0.2cm]
					&\overset{(a)}{\leq}&\displaystyle 2\eta_2 \sum_{t\,=\,0}^{T-1} \lambda^{(t)} \big(\,{V_g^\star(\rho)\, -\, V_g^{(t)}(\rho)}\,\big)\,+\, \frac{\eta_2^2\, T}{(1-\gamma)^2}
				\end{array}
			\end{equation}
			where $(a)$ is because of the feasibility of $\pi^\star$: $V_g^\star (\rho) \geq b$, and $|{V_g^{(t)}(\rho)-b}|\leq {1}/({1-\gamma})$. Hence,
			\begin{equation}\label{eq.lambda.ap.clean}
				- \,\frac{1}{T}\sum_{t\,=\,0}^{T-1} \lambda^{(t)} \big(\,{V_g^\star(\rho)\, -\, V_g^{(t)}(\rho)}\,\big)
				\; \leq \; 
				\frac{\eta_2 }{2(1-\gamma)^2}.
			\end{equation}
		\end{subequations}
		By adding the inequality~\eqref{eq.lambda.ap.clean} to~\eqref{eq.keyone} on both sides and taking $\eta_1=\eta_2={1}/{\sqrt{T}}$, we obtain the first inequality.
		
		\noindent\textbf{Proving the second inequality}.
		Since the dual update in~\eqref{eq.NPGAGD.primalapproximation} is the same as the one in~\eqref{eq.NPGAGD.clean}, we can use the same reasoning to conclude~\eqref{eq.removelamda}. Adding the inequality~\eqref{eq.removelamda} to~\eqref{eq.keyone} on both sides and using $V_g^{\star}(\rho) \geq b$ yield
		\begin{equation}\label{eq.dualkey_fa}
			\begin{array}{rcl}
				&& \!\!\!\! \!\!\!\! \!\! \displaystyle \frac{1}{T} \sum_{t\,=\,0}^{T-1} \big(\,{V_r^\star(\rho) \,-\, V_r^{(t)}(\rho) }\,\big) 
				\,+\, 
				\frac{\lambda}{T} \sum_{t\,=\,0}^{T-1} \big(\,{b\,-\, V_g^{(t)}(\rho)}\,\big)
				\\[0.2cm]
				&\leq&\displaystyle \frac{\log|A|}{(1-\gamma)\eta_1 T} 
				\,+\,
				\frac{1}{(1-\gamma)T}\sum_{t\,=\,0}^{T-1} \text{\normalfont err}_r^{(t)} (\pi^\star)
				\,+\,
				\frac{2}{(1-\gamma)^2\xi T}\sum_{t\,=\,0}^{T-1} \text{\normalfont err}_g^{(t)} (\pi^\star)
				\\[0.2cm]
				&& \displaystyle
				+~\beta\frac{\eta_1W^2}{(1-\gamma)^3} 
				\,+\, 
				\beta\frac{4\eta_1W^2}{(1-\gamma)^5\xi^2}
				\,+\,
				\frac{1}{2\eta_2T}\big\vert{{\lambda^{(0)} -\lambda}}\big\vert^2 
				\,+\,
				\frac{\eta_2}{2(1-\gamma)^2}.
			\end{array}
		\end{equation}
		Taking $\lambda = \frac{2}{(1-\gamma)\xi}$ when ${\sum_{t\,=\,0}^{T-1} \big({b -V_g^{(t)}(\rho)}}\big)\geq 0$ and $\lambda=0$ otherwise, we obtain
		\[
		\begin{array}{rcl}
			&& \!\!\!\! \!\!\!\! \!\!\displaystyle V_r^\star(\rho)\,-\, \frac{1}{T} \sum_{t\,=\,0}^{T-1}{ V_r^{(t)}(\rho) } \,+\, \frac{2}{(1-\gamma)\xi} \sbr{\,b \,-\, \frac{1}{T} \sum_{t\,=\,0}^{T-1} {V_g^{(t)}(\rho)}\,}_+
			\\[0.2cm]
			&\leq&\displaystyle \frac{\log|A|}{(1-\gamma)\eta_1 T} 
			\,+\,
			\frac{1}{(1-\gamma)T}\sum_{t\,=\,0}^{T-1} \text{\normalfont err}_r^{(t)} (\pi^\star)
			\,+\,
			\frac{2}{(1-\gamma)^2\xi T}\sum_{t\,=\,0}^{T-1} \text{\normalfont err}_g^{(t)} (\pi^\star)
			\\[0.2cm]
			&& \displaystyle
			+~\beta\frac{\eta_1W^2}{(1-\gamma)^3} 
			\,+\, 
			\beta\frac{4\eta_1W^2}{(1-\gamma)^5\xi^2T}
			\,+\,
			\frac{2}{\eta_2 (1-\gamma)^2\xi^2}
			\,+\,
			\frac{\eta_2}{2(1-\gamma)^2}.
		\end{array}
		\]
		
		Since $V_r^{(t)}(\rho)$ and $V_g^{(t)}(\rho)$ are linear functions in the occupancy measure~\cite[Chapter~10]{altman1999constrained}, there exists a policy $\pi'$ such that $V_r^{\pi'}(\rho)=\frac{1}{T} \sum_{t\,=\,0}^{T-1}{ V_r^{(t)}(\rho) }$ and $V_g^{\pi'}(\rho)=\frac{1}{T} \sum_{t\,=\,0}^{T-1}{ V_g^{(t)}(\rho) }$. Hence,
		\[
		\begin{array}{rcl}
			&& \!\!\!\! \!\!\!\! \!\! \displaystyle V_r^\star(\rho)
			\,-\, 
			V_r^{\pi'}(\rho) 
			\,+\, 
			\frac{2}{(1-\gamma)\xi} \sbr{\, b \,-\,  {V_g^{\pi'}(\rho)} \,}_+
			\\[0.2cm]
			&\leq&\displaystyle \frac{\log|A|}{(1-\gamma)\eta_1 T} 
			\,+\,
			\frac{1}{(1-\gamma)T}\sum_{t\,=\,0}^{T-1} \text{\normalfont err}_r^{(t)} (\pi^\star)
			\,+\,
			\frac{2}{(1-\gamma)^2\xi T}\sum_{t\,=\,0}^{T-1} \text{\normalfont err}_g^{(t)} (\pi^\star)
			\\[0.2cm]
			&& \displaystyle
			+~\beta\frac{\eta_1W^2}{(1-\gamma)^3} 
			\,+\, 
			\beta\frac{4\eta_1W^2}{(1-\gamma)^5\xi^2}
			\,+\,
			\frac{2}{\eta_2 (1-\gamma)^2\xi^2T}
			\,+\,
			\frac{\eta_2}{2(1-\gamma)^2}.
		\end{array}
		\]
		From Lemma~\ref{lem.duality} (ii), we have $\lambda^\star \leq 1/((1-\gamma)\xi)$. Application of Lemma~\ref{lem.constraint} with $C= {2}/((1-\gamma)\xi)$ yields
		\[
		\begin{array}{rcl}
			\sbr{\,b \,-\,  {V_g^{\pi'}(\rho)}\,}_+
			&\leq&\displaystyle \frac{\xi\log|A|}{\eta_1 T} 
			\,+\,\frac{\xi}{T}\sum_{t\,=\,0}^{T-1} \text{\normalfont err}_r^{(t)} (\pi^\star)
			\,+\,\frac{2}{(1-\gamma) T}\sum_{t\,=\,0}^{T-1} \text{\normalfont err}_g^{(t)} (\pi^\star)
			\\[0.2cm]
			&& \displaystyle
			+~\beta\frac{\eta_1\xi \,W^2}{(1-\gamma)^2} 
			\,+\, 
			\beta\frac{4\eta_1 W^2}{(1-\gamma)^4\xi }
			\,+\,
			\frac{2}{\eta_2 (1-\gamma)\xi T}
			\,+\,
			\frac{\eta_2\xi}{2(1-\gamma)}.
		\end{array}
		\]
		which leads to our constraint violation bound if we further utilize $\frac{1}{T} \sum_{t\,=\,0}^{T-1} \big({b -V_g^{(t)}(\rho)}\big)=b -  {V_g^{\pi'}(\rho)}$ and $\eta_1 =\eta_2={1}/{\sqrt{T}}$.
	\end{proof}

	The analysis of Lemma~\ref{lem.gap.violation} is based on the generalization of Lemma~\ref{lem.average.performance} to the function approximation setting using the property of policy smoothness. A crucial step is to use the original optimal policy as our comparison policy in hindsight, instead of a sub-optimal policy within policy class~\cite[Theorem~2]{ding2020natural}. Although the strong duality may not hold because of the insufficient expressiveness of the parametrized policy class, we can characterize the regret and constraint violation bounds, up to some function approximation errors. 
	
	\begin{proof}[Proof of Theorem~\ref{thm.convergence.loglinear}]
		
		When $\norm{\phi_{s,a}}\leq B$, for the log-linear policy class, $\log \pi_\theta(a\,\vert\,s)$ is $\beta$-smooth with $\beta = B^2$.
		By Lemma~\ref{lem.gap.violation}, it remains to consider the randomness in the sequences of $\{w_r^{(t)}, w_g^{(t)}\}$ and the error bounds for $\text{\normalfont err}_r^{(t)}(\pi^\star)$ and $\text{\normalfont err}_g^{(t)}(\pi^\star)$. Application of the triangle inequality yields 
		\begin{equation}\label{eq.err}
			\begin{array}{rcl}
				\text{\normalfont err}_r^{(t)} (\pi^\star)
				& \leq &
				\left\vert
				\mathbb{E}_{s\,\sim\,d_\rho^\star} \mathbb{E}_{a\,\sim\,\pi^\star(\cdot\,\vert\,s)} 
				\sbr{\,
					A_r^{(t)}(s,a) \, -\,  \left(w_{r,\star}^{(t)}\right)^\top \nabla_\theta \log \pi_\theta^{(t)}(a\,\vert\,s)
					\,}
				\right\vert
				\\[0.2cm]
				&& +~\left\vert
				\mathbb{E}_{s\,\sim\,d_\rho^\star} \mathbb{E}_{a\,\sim\,\pi^\star(\cdot\,\vert\,s)} 
				\sbr{\,
					\left(w_{r,\star}^{(t)}\, -\, w_{r}^{(t)}\right)^\top \nabla_\theta \log \pi_\theta^{(t)}(a\,\vert\,s) 
					\,}
				\right\vert.
			\end{array}
		\end{equation}
		Application of~\eqref{eq.policyshift} and $A_r^{(t)}(s,a)  = Q_r^{(t)}(s,a) -  \mathbb{E}_{a'\,\sim\,\pi_\theta^{(t)}(\cdot\,\vert\,s)} \big[\,Q_r^{(t)}(s,a')\,\big]$ yields
		\begin{equation}\label{eq.err0}
			\begin{array}{rcl}
				&& \!\!\!\! \!\!\!\! \!\!	
				\mathbb{E}_{s\,\sim\,d_\rho^\star} \mathbb{E}_{a\,\sim\,\pi^\star(\cdot\,\vert\,s)} 
				\sbr{\,
					A_r^{(t)}(s,a) \, - \, \left( w_{r,\star}^{(t)}\right)^\top \nabla_\theta \log \pi_\theta^{(t)}(a\,\vert\,s)
					\,}
				\\[0.2cm]
				& = & 
				\mathbb{E}_{s\,\sim\,d_\rho^\star} \mathbb{E}_{a\,\sim\,\pi^\star(\cdot\,\vert\,s)} 
				\sbr{ \, 
					Q_r^{(t)}(s,a) \,-\, \phi_{s,a}^\top w_{r,\star}^{(t)} 
					\, }
				\\[0.2cm]
				&& - \,
				\mathbb{E}_{s\,\sim\,d_\rho^\star} \mathbb{E}_{a'\,\sim\,\pi_\theta^{(t)}(\cdot\,\vert\,s)} 
				\sbr{ \, 
					Q_r^{(t)}(s,a') \, - \, \phi_{s,a'}^\top w_{r,\star}^{(t)} 
					\,}
				\\[0.2cm]
				& \leq &
				\sqrt{
					\mathbb{E}_{s\,\sim\,d_\rho^\star} \mathbb{E}_{a\,\sim\,\pi^\star(\cdot\,\vert\,s)} 
					\rbr{\,
						Q_r^{(t)}(s,a) \,-\, \phi_{s,a}^\top w_{r,\star}^{(t)} 
						\,}^2
				}
				\\[0.2cm]
				&& + ~
				\sqrt{
					\mathbb{E}_{s\,\sim\,d_\rho^\star} \mathbb{E}_{a'\,\sim\,\pi_\theta^{(t)}(\cdot\,\vert\,s)} 
					\rbr{\,
						Q_r^{(t)}(s,a') \,-\, \phi_{s,a'}^\top w_{r,\star}^{(t)} 
						\,}^2
				}
				\\[0.2cm]
				& \leq & 2
				\sqrt{ |A| \,
					\mathbb{E}_{s\,\sim\,d_\rho^\star} \mathbb{E}_{a\,\sim\,\text{Unif}_A} 
					\sbr{ \,
						\rbr{Q_r^{(t)}(s,a) \,-\, \phi_{s,a}^\top w_{r,\star}^{(t)}}^2 
						\,}
				}
				\\[0.2cm]
				& = &
				2\sqrt{ |A| \, \mathcal{E}_r^{\nu^\star} \left(w_{r,\star}^{(t)} ; \theta^{(t)}\right) }.
			\end{array}
		\end{equation}
		Similarly,
		\begin{equation}\label{eq.err1}
			\begin{array}{rcl}
				&& \!\!\!\! \!\!\!\!  \!\!
				\mathbb{E}_{s\,\sim\,d_\rho^\star} \mathbb{E}_{a\,\sim\,\pi^\star(\cdot\,\vert\,s)} 
				\sbr{\, \left(w_{r,\star}^{(t)} 
					\,-\,
					w_{r}^{(t)}\right)^\top \nabla_\theta \log \pi_\theta^{(t)}(a\,\vert\,s) 
					\,}
				\\[0.2cm]
				&=&
				\mathbb{E}_{s\,\sim\,d_\rho^\star} \mathbb{E}_{a\,\sim\,\pi^\star(\cdot\,\vert\,s)} 
				\sbr{ \,
					\left(w_{r,\star}^{(t)}
					\,-\, 
					w_{r}^{(t)}\right)^\top \phi_{s,a} 
					\, }
				\\[0.2cm]
				&& - ~\mathbb{E}_{s\,\sim\,d_\rho^\star} \mathbb{E}_{a'\,\sim\,\pi_\theta^{(t)}(\cdot\,\vert\,s)} 
				\sbr{ \,
					\left(w_{r,\star}^{(t)}
					\, - \,
					w_{r}^{(t)}\right)^\top \phi_{s,a'} 
					\,}
				\\[0.2cm]
				&\leq&
				2 \sqrt{ |A|\,
					\mathbb{E}_{s\,\sim\,d_\rho^\star} \mathbb{E}_{a\,\sim\,\text{Unif}_A} \sbr{ \rbr{ \left(w_{r,\star}^{(t)} 
							\, - \,
							w_{r}^{(t)}\right)^\top \phi_{s,a} }^2 }
				}
				\\[0.2cm]
				&=&
				2 \sqrt{ 
					|A|  
					\norm{ w_{r,\star}^{(t)}
						\, - \,
						w_{r}^{(t)} }_{\Sigma_{\nu^\star}}^2
				}
			\end{array}
		\end{equation}
		where $\Sigma_{\nu^\star} \DefinedAs \mathbb{E}_{(s,a)\,\sim\,\nu^\star} \left[\, \phi_{s,a} \phi_{s,a}^\top \,\right]$. From the definition of $\kappa$, we have
		\begin{equation}\label{eq.err2}
			\norm{ w_{r,\star}^{(t)} \, - \, w_{r}^{(t)} }_{\Sigma_{\nu^\star}}^2
			\; \leq  \;
			\kappa
			\norm{ w_{r,\star}^{(t)} \, - \, w_{r}^{(t)} }_{\Sigma_{\nu_0}}^2
			\; \leq  \;
			\frac{\kappa}{1-\gamma}
			\norm{ w_{r,\star}^{(t)} \, - \, w_{r}^{(t)} }_{\Sigma_{\nu^{(t)}}}^2
		\end{equation}
		where we use 
		$(1-\gamma)\nu_0 \leq \nu_{\nu_0}^{\pi^{(t)}} \DefinedAs \nu^{(t)}$ in the second inequality. Evaluation of the first-order optimality condition of $w_{r,\star}^{(t)} \in \argmin_{\norm{w_r}_2\,\leq\, W} \, \mathcal{E}_r^{\nu^{(t)}} (w_r; \theta^{(t)})$ yields
		\[
		\left( w_r \, - \, w_{r,\star}^{(t)} 
		\right)^\top
		\nabla_\theta \mathcal{E}_r^{\nu^{(t)}} (w_{r,\star}^{(t)}; \theta^{(t)}) 
		\; \geq  \;
		0, 
		\; \text{ for any } w_r \text{ satisfying } \norm{w_r} \,\leq\, W. 
		\]
		Thus,
		\[
		\begin{array}{rcl}
			&& \!\!\!\! \!\!\!\! \!\!
			\mathcal{E}_r^{\nu^{(t)}} (w_{r}; \theta^{(t)}) - \mathcal{E}_r^{\nu^{(t)}} (w_{r,\star}^{(t)}; \theta^{(t)}) 
			\\[0.2cm]
			& = & 	
			\mathbb{E}_{s,a\,\sim\,\nu^{(t)}} \left[\,  \rbr{ Q_r^{(t)}(s,a) \, -\,   \phi_{s,a}^\top w_{r,\star}^{(t)} \, +\,  \phi_{s,a}^\top w_{r,\star}^{(t)} \, -\,  \phi_{s,a}^\top  w_r }^2 \,\right]
			\, -\,  \mathcal{E}_r^{\nu^{(t)}} (w_{r,\star}^{(t)}; \theta^{(t)}) 
			\\[0.4cm]
			& = & 
			2 \left(w_{r,\star}^{(t)} \, -\,  w_r\right)^\top	
			\mathbb{E}_{s,a\,\sim\,\nu^{(t)}} \left[\,  \rbr{ Q_r^{(t)}(s,a) \, -\,   \phi_{s,a}^\top w_{r,\star}^{(t)} } \phi_{s,a}  \,\right]
			\\[0.2cm]
			&& +~
			\mathbb{E}_{s,a\,\sim\,\nu^{(t)}} \left[\,  \rbr{ \phi_{s,a}^\top w_{r,\star}^{(t)} \, -\,  \phi_{s,a}^\top w_r  }^2 \,\right]
			\\[0.4cm]
			& = &  \left( w_r  \, -\,  w_{r,\star}^{(t)} \right)^\top	\nabla_\theta \mathcal{E}_r^{\nu^{(t)}} (w_{r,\star}^{(t)}; \theta^{(t)})  \, +\,  \norm{ w_r \, -\,  w_{r,\star}^{(t)} }_{\Sigma_{\nu^{(t)}}}^2
			\\[0.4cm]
			& \geq &  \norm{ w_r \, -\,  w_{r,\star}^{(t)} }_{\Sigma_{\nu^{(t)}}}^2.
		\end{array}
		\]
		Taking $w_r =  w_r^{(t)}$ in the above inequality and combining it with~\eqref{eq.err1} and~\eqref{eq.err2} yield
		\begin{equation}\label{eq.err3}
			\begin{array}{rcl}
				&& \!\!\!\! \!\!\!\! \!\!
				\mathbb{E}_{s\,\sim\,d_\rho^\star} \mathbb{E}_{a\,\sim\,\pi^\star(\cdot\,\vert\,s)} 
				\sbr{ \,
					\left(w_{r,\star}^{(t)} \,-\,w_{r}^{(t)}\right)^\top \nabla_\theta \log \pi_\theta^{(t)}(a\,\vert\,s) 
					\, }
				\\[0.2cm]
				& \leq & \displaystyle
				2\sqrt{\frac{\kappa\,|A|}{1-\gamma} \left( \mathcal{E}_r^{\nu^{(t)}} (w_{r}^{(t)}; \theta^{(t)}) 
					\, - \,
					\mathcal{E}_r^{\nu^{(t)}} (w_{r,\star}^{(t)}; \theta^{(t)})  \right)}.
			\end{array}
		\end{equation}
		
		Substitution of~\eqref{eq.err0} and~\eqref{eq.err3} into the right-hand side of~\eqref{eq.err} yields
		\[
		\mathbb{E} \sbr{
			\text{\normalfont err}_r^{(t)} (\pi^\star)}
		\; \leq \;
		2\sqrt{ |A| \,  \mathbb{E} \sbr{ \mathcal{E}_r^{d^\star} (w_{r,\star}^{(t)} ; \theta^{(t)}) }}
		\, +\, 
		2\sqrt{\frac{\kappa\,|A|}{1-\gamma} \mathbb{E} \sbr{ \mathcal{E}_r^{\nu^{(t)}} (w_{r}^{(t)}; \theta^{(t)}) - \mathcal{E}_r^{\nu^{(t)}} (w_{r,\star}^{(t)}; \theta^{(t)}) } }.
		\]
		By the same reasoning, we can establish a similar bound on $\mathbb{E} \big[{
			\text{\normalfont err}_g^{(t)} (\pi^\star)}\big]$. 
		Finally, our desired results follow by applying Assumption~\ref{as.errors} and Lemma~\ref{lem.gap.violation}.
	\end{proof}
	
	To obtain zero constraint violation, we apply the algorithm~\eqref{eq.NPGAGD.primalapproximation} to Problem~\eqref{eq.cmdp} with a conservative constraint: $V_{g}^{\pi}(\rho)\geq b+\delta$ for some $\delta>0$, , as done in Corollary~\ref{thm.convergence.softmax.zeroviolation}. In addition to the parameters $(\epsilon, \delta)$, the errors of function approximation (e.g.,  $\epsilon_{\text{\normalfont  est}}$ and $\epsilon_{\text{\normalfont  bias}}$) are required to be small. 
	
	\begin{corollary}[Zero constraint violation: log-linear policy parametrization]\label{cor.convergence.loglinear.zeroviolation}
		Let Assumption~\ref{as.slater} hold for $\xi>0$ and let us fix a state distribution $\rho$ and replace the constraint of Problem~\eqref{eq.cmdp} by $V_g^{\pi}(\rho) \geq \bar{b}$, where $\bar{b}\DefinedAs b+\delta$ for some $\delta>0$.
		For $\epsilon<{\xi}/{2}$, there exists $\delta = \Theta(\epsilon)$ such that 
		if we choose $T = \Omega(1/\epsilon^2)$, $\eta_1=\eta_2={1}/{\sqrt{T}}$, and Assumption~\ref{as.errors} holds for $\epsilon_{\text{\normalfont  est}} = \epsilon_{\text{\normalfont  bias}} = O(\epsilon^2)$, then the iterates $\{\theta^{(t)}\}_{t=0}^{T-1}$ generated by the algorithm~\eqref{eq.NPGAGD.primalapproximation} satisfy 
		\[
		\begin{array}{rcl}
			\text{\normalfont(Optimality gap)}\;\; 
			\displaystyle 
			\mathbb{E}
			\left[\,
			\frac{1}{T} \sum_{t\,=\,0}^{T-1} \big(\,V_r^\star(\rho) \,-\, V_r^{(t)}(\rho)\, \big)
			\,\right]
			& = &
			\displaystyle
			O(\epsilon)
			\\[0.4cm]
			\text{\normalfont(Constraint violation)}\;\; 
			\displaystyle 
			\mathbb{E}
			\sbr{\,
				\frac{1}{T}\sum_{t \, = \,0 }^{T-1} \big(\,b\,-\,V_g^{(t)}(\rho)\,\big)
				\,}_+  
			& \leq & 
			0.
		\end{array}
		\]
	\end{corollary}
	
	\begin{proof}
		The proof idea is similar to the one used in the proof of Theorem~\ref{thm.convergence.loglinear}. Using the new constraint $V_g^{\pi}(\rho) \geq \bar{b}$, Problem~\eqref{eq.cmdp} satisfies Assumption~\ref{as.slater} for $\bar\xi \DefinedAs \xi - \delta$ where $\delta<\xi$, and there exists an optimal policy $\bar\pi^{\star}$. Without loss of generality, by restricting $\delta<{\xi}/{2}$, we can replace $\Lambda$ by $\bar\Lambda \DefinedAs [\,0, 4/((1-\gamma)\xi)\,]$, which contains $[\,0, 2/((1-\gamma)\bar\xi)\,]$ for any such $\bar{\xi}$. Thus, we can apply the NPG-PD algorithm~\eqref{eq.NPGAGD.primalapproximation} to this conservative problem using the projection set $\bar\Lambda$. It is straightforward to check that~\eqref{eq.dualkey_fa} holds for $V_r^{\bar\pi^\star}(\rho)$ and $V_g^{\bar\pi^\star}(\rho)$. Thus, the second inequality proof in Lemma~\ref{lem.gap.violation} in conjunction with $\epsilon_{\text{\normalfont  est}} = \epsilon_{\text{\normalfont  bias}} = O(\epsilon^2)$ proves that after $T= \Omega({1}/{\epsilon^2})$ iterations, 
		\begin{equation}\label{eq.optimality gap relaxed fa}
			\displaystyle 
			\mathbb{E} 
			\left[\,
			\frac{1}{T} \sum_{t\,=\,0}^{T-1} \big(\,V_r^{\bar\pi^\star}(\rho) \,-\, V_r^{(t)}(\rho) \,\big)
			\,\right]
			\; = \;
			\displaystyle
			O(\epsilon)
		\end{equation}
		where  the expectation $\mathbb{E}$ is taken over the randomness of approximate algorithm that is used to solve~\eqref{eq.NPGAGD.primal.app}.
		Let $q^\star$ and $\bar{q}^\star$ be the occupancy measures induced by policies $\pi^\star$ and $\bar{\pi}^\star$, respectively. In the occupancy measure space, Problem~\eqref{eq.cmdp} becomes a linear program, and thus, $V_r^{\pi^\star}(\rho) = \langle r, q^\star \rangle$ and $V_r^{\bar\pi^\star}(\rho) = \langle r, \bar q^\star \rangle$. By the continuity of optimal objective function in convex optimization~\citep{terazono2015continuity}, $|V_r^{\pi^\star}(\rho) - V_r^{\bar\pi^\star}(\rho)| \leq {2\epsilon}/((1-\gamma)\xi)$ for $\delta = \epsilon$. Therefore,  we can replace $V_r^{\bar\pi^\star}(\rho)$ in~\eqref{eq.optimality gap relaxed fa} by $V_r^{\star}(\rho)$ to bound the optimality gap by the same desired accuracy $\epsilon$ up to some problem-dependent constant. 
		
		To establish the bound on the constraint violation, the key change begins with~\eqref{eq.dualkey_fa}. Since we use $\bar b = b+\delta$ and $V_r^{\bar\pi^\star}(\rho)$, the right-hand side of~\eqref{eq.dualkey_fa} contains an extra term $ 2\epsilon/((1-\gamma)\xi)-\lambda \delta$. Similarly, there are two options for selecting $\lambda$:  $\lambda = {4}/((1-\gamma)\xi)$ when ${\sum_{t\,=\,0}^{T-1} \big({b -V_g^{(t)}(\rho)}}\big)\geq 0$ and $\lambda=0$ otherwise. In the first case, if we set $\delta = \epsilon$ and $\epsilon_{\text{\normalfont  est}} = \epsilon_{\text{\normalfont  bias}} = O(\epsilon^2)$, then the extra term $- 2\epsilon/((1-\gamma)\xi)$ cancels the rate $O(1/\sqrt{T})$ for $T = \Omega(1/\epsilon^2)$ and the function approximation errors $\epsilon_{\text{\normalfont  est}}$ and  $\epsilon_{\text{\normalfont  bias}}$, concluding zero constraint violation according to Lemma~\ref{lem.constraint}. On the other hand,  the second case is exactly the zero constraint violation.
	\end{proof}
	
	\begin{remark}[Zero constraint violation: log-linear policy parametrization]
		\label{re.zero.violation.log-linear}
		As done in Corollary~\ref{cor.convergence.loglinear.zeroviolation}, we can refine the constraint violation in Corollary~\ref{cor.convergence.loglinear} to be zero. Given a small desired accuracy $\epsilon>0$, there exists $\delta = \Theta(\epsilon)$ such that 
		if $T = \Omega(1/\epsilon^2)$, and $\epsilon_{\text{\normalfont  est}} = \epsilon_{\text{\normalfont  approx}} = O(\epsilon^2)$, then
		\[
		\begin{array}{rcl}
			\displaystyle 
			\mathbb{E}
			\sbr{\,
				\frac{1}{T}\sum_{t \, = \,0 }^{T-1} \big(\,b\,-\,V_g^{(t)}(\rho)\,\big)
				\,}_+  
			& \leq & 
			0.
		\end{array}
		\]
		We note that $\epsilon_{\text{\normalfont  approx}} = O(\epsilon^2)$ is more difficult to achieve than $\epsilon_{\text{\normalfont  bias}} = O(\epsilon^2)$ in practice. 
	\end{remark}
	
	\subsection{General smooth policy class}
	\label{sec.fa.general}

	For a general class of smooth policies~\citep{zhang2020global,agarwal2021optimality}, we now establish the convergence of algorithm~\eqref{eq.NPGAGD.primalapproximation} with an approximate gradient update:
	\begin{equation}\label{eq.NPGAGD.primal.generalapp}
		\begin{array}{rcl}
			w^{(t)} 
			& = &
			w_r^{(t)} 
			\, + \,
			\lambda^{(t)}
			\, 
			w_g^{(t)}
			\\[0.15cm]
			w_\diamond^{(t)}  
			& \approx &
			\argmin\limits_{\norm{w_\diamond}_2\,\leq\, W}
			\;{E}_\diamond^{\nu^{(t)}} \big(w_\diamond; \theta^{(t)}\big)
		\end{array}
	\end{equation} 
	where $\diamond$ denotes $r$ or $g$ and an exact minimizer is given by $w_{\diamond,\star}^{(t)} \in \argmin_{\norm{w_\diamond}_2 \leq W} {E}_\diamond^{\nu^{(t)}} (w_\diamond; \theta^{(t)})$. We treat $w_\diamond^{(t)}$ as random since it is typically obtained from a sample-based algorithm.
	\begin{assumption}[Policy smoothness]
		\label{as.smoothness}
		For all $s\in S$ and $a\in A$, $\log \pi_\theta(a\,\vert\,s)$ is a $\beta$-smooth function of $\theta$
		\[
		\norm{\nabla_\theta \log \pi_\theta(a\,\vert\,s) \,-\,\nabla_{\theta} \log \pi_{\theta'}(a\,\vert\,s)}
		\;\leq\;
		\beta\norm{ \theta \,-\, \theta'} \; \text{ for all } ~\theta, \, \theta' \, \in \, \mathbb{R}^d.
		\]
	\end{assumption}
	Since both the tabular softmax and log-linear policies are smooth~\citep[Remark 28 and Appendix~D]{agarwal2021optimality}, Assumption~\ref{as.smoothness} covers a broader function class relative to the softmax policy parametrization~\eqref{eq.softmax}. 
	
	Given a state-action distribution $\nu^{(t)}$, we introduce the estimation error as 
	\[
	{E}_{\diamond,\text{\normalfont est}}^{(t)}
	\;
	\DefinedAs
	\;
	\mathbb{E} 
	\left[\,
	{E}_\diamond^{\nu^{(t)}} \big(w_\diamond^{(t)}; \theta^{(t)}\big)
	\, - \,
	{E}_\diamond^{\nu^{(t)}} \big(w_{\diamond,\star}^{(t)}; \theta^{(t)}\big)
	\,\big\vert\,\theta^{(t)}
	\,\right].
	\]
	Furthermore, given a state distribution $\rho$ and an optimal policy $\pi^\star$, we define a state-action distribution $\nu^\star(s,a) \DefinedAs d_\rho^{\pi^\star}(s) \pi^\star(a\,\vert\,s)$ as a comparator and introduce the transfer error
	\[
	{E}_{\diamond,\text{\normalfont bias}}^{(t)}
	\;
	\DefinedAs
	\;
	\mathbb{E} 
	\left[\,
	{E}_\diamond^{\nu^{\star}} \big(w_{\diamond,\star}^{(t)}; \theta^{(t)}\big)
	\,\right].
	\]
	For any state-action distribution $\nu$, we define a Fisher information-like matrix induced by $\pi_\theta$ as
	\[
	\Sigma_{\nu}^\theta 
	\; \DefinedAs \; 
	\mathbb{E}_{(s,a)\,\sim\,\nu} 
	\left[\, 
	\nabla_\theta \log \pi_\theta (a\,\vert\,s) \left( \nabla_\theta \log \pi_\theta (a\,\vert\,s) \right)^\top 
	\,\right]
	\] 
	and use $\Sigma_{\nu}^{(t)}$ to denote $\Sigma_{\nu}^{\theta^{(t)}}$.
	
	\begin{assumption}[Estimation/transfer errors and relative condition number]\label{as.errors+condtion}
		The estimation and transfer errors as well as the expected relative condition number are bounded, i.e., ${E}_{\diamond,\text{\normalfont est}}^{(t)}\leq \epsilon_{\text{\normalfont  est}}$ and ${E}_{\diamond,\text{\normalfont bias}}^{(t)} \leq \epsilon_{\text{\normalfont  bias}}$, for $\diamond = r$ or $g$, and  
		\[
		\mathbb{E}
		\sbr{\,
			\sup_{w\,\in\,\mathbb{R}^d}\, \frac{w^\top \Sigma_{\nu^\star}^{(t)} w}{w^\top \Sigma_{\nu_0}^{(t)}w}
			\,}
		\; \leq \;
		\kappa.
		\]
	\end{assumption}
	
	We next provide convergence guarantees for the algorithm~\eqref{eq.NPGAGD.primalapproximation} in Theorem~\ref{thm.convergence.general} using the approximate update~\eqref{eq.NPGAGD.primal.generalapp}. Even though we set $\theta^{(0)}=0$ and $\lambda^{(0)}=0$ in the proof of Theorem~\ref{thm.convergence.general}, convergence can be established for arbitrary initial conditions.
	
	\begin{theorem}[Convergence and optimality: general policy parametrization]
		\label{thm.convergence.general}
		Let Assumptions~\ref{as.slater}  and~\ref{as.smoothness} hold and let us fix a state distribution $\rho$, a state-action distribution $\nu_0$, and $T>0$.
		If the iterates $\{(\theta^{(t)},\lambda^{(t)})\}_{t=0}^{T-1}$  generated by the algorithm~\eqref{eq.NPGAGD.primalapproximation} using~\eqref{eq.NPGAGD.primal.generalapp} with $\eta_1=\eta_2={1}/{\sqrt{T}}$ satisfy Assumption~\ref{as.errors+condtion}
		and $ \Vert w_\diamond^{(t)}\Vert \leq W$, then
		\[
		\begin{array}{rcl}
			\displaystyle
			\mathbb{E}
			\sbr{\,
				\frac{1}{T} \sum_{t\,=\,0}^{T-1} \big(\,V_r^\star(\rho) \,-\,V_r^{(t)}(\rho)\,\big)
				\,}
			& \leq & \displaystyle
			\frac{C_3}{(1-\gamma)^5}\frac{1}{\sqrt{T}} 
			\,+\, 
			\frac{1+2/\xi}{(1-\gamma)^2} \left( \sqrt{\epsilon_{\text{\normalfont  bias}}} + \sqrt{\frac{\kappa\,\epsilon_{\text{\normalfont  est}}}{1-\gamma}}\right)
			\\[0.4cm]
			\displaystyle
			\mathbb{E}
			\sbr{\,
				\frac{1}{T} \sum_{t\,=\,0}^{T-1} \big(\,b\,-\,V_g^{(t)}(\rho)\,\big)
				\,}_+
			& \leq & \displaystyle
			\frac{C_4}{(1-\gamma)^4}\frac{1}{\sqrt{T}} 
			\,+\,
			\frac{2+\xi}{1-\gamma} \left( \sqrt{\epsilon_{\text{\normalfont  bias}}} + \sqrt{\frac{\kappa\,\epsilon_{\text{\normalfont  est}}}{1-\gamma}}\right)
		\end{array}
		\]
		where $C_3 \DefinedAs 1+\log |A|+5\beta W^2/\xi^2$ and $C_4 \DefinedAs (1+\log |A|+\beta W^2)\xi + (2+4\beta W^2)/\xi$.
	\end{theorem}
	
	\begin{proof}
		Since Lemma~\ref{lem.gap.violation} holds for any smooth policy class that satisfies Assumption~\ref{as.smoothness}, it remains to bound $\text{\normalfont err}_\diamond^{(t)} (\pi^\star)$ for $\diamond = r$ or $g$. We next separately bound each term on the right-hand side of~\eqref{eq.err}.
		For the first term,
		\begin{equation}\label{eq.err0g}
			\begin{array}{rcl}
				&& \!\!\!\! \!\!\!\! \!\!	
				\mathbb{E}_{s\,\sim\,d_\rho^\star} \mathbb{E}_{a\,\sim\,\pi^\star(\cdot\,\vert\,s)} 
				\sbr{\,
					A_r^{(t)}(s,a) 
					\, - \, \left(w_{r,\star}^{(t)}\right)^\top \nabla_\theta \log \pi_\theta^{(t)}(a\,\vert\,s)
					\,}
				\\[0.2cm]
				& \leq &
				\sqrt{
					\mathbb{E}_{s\,\sim\,d_\rho^\star} \mathbb{E}_{a\,\sim\,\pi^\star(\cdot\,\vert\,s)} 
					\rbr{\,
						A_r^{(t)}(s,a) 
						\, - \,
						(w_{r,\star}^{(t)})^\top \nabla_\theta \log \pi_\theta^{(t)}(a\,\vert\,s) \,}^2
				}
				\\[0.2cm]
				& = &
				\sqrt{{E}_r^{\nu^\star} \left(w_{r,\star}^{(t)} ; \theta^{(t)}\right) }.
			\end{array}
		\end{equation}
		Similarly,
		\begin{subequations}
			\label{eq.err0g.2nd}
			\begin{equation}\label{eq.err1g}
				\begin{array}{rcl}
					&& \!\!\!\! \!\!\!\!  \!\!
					\mathbb{E}_{s\,\sim\,d_\rho^\star} \mathbb{E}_{a\,\sim\,\pi^\star(\cdot\,\vert\,s)} 
					\sbr{ \,
						\left(w_{r,\star}^{(t)} \, - \, w_{r}^{(t)}\right)^\top \nabla_\theta \log \pi_\theta^{(t)}(a\,\vert\,s) \,}
					\\[0.2cm]
					&\leq&
					\sqrt{
						\mathbb{E}_{s\,\sim\,d_\rho^\star} \mathbb{E}_{a\,\sim\,\pi^\star(\cdot\,\vert\,s)} 
						\sbr{\,
							\rbr{ \left(w_{r,\star}^{(t)}
								\, - \,
								w_{r}^{(t)}\right)^\top  \nabla_\theta \log \pi_\theta^{(t)}(a\,\vert\,s) }^2 
							\,}
					}
					\\[0.2cm]
					&=&
					\sqrt{ 
						\norm{ w_{r,\star}^{(t)}
							\, - \,
							w_{r}^{(t)} }_{\Sigma_{\nu^\star}^{(t)}}^2.
					}
				\end{array}
			\end{equation}
			Let $\kappa^{(t)} \DefinedAs \left\Vert{ \left(\Sigma_{\nu_0}^{(t)}\right)^{-1/2} \Sigma_{\nu^\star}^{(t)} \left(\Sigma_{\nu_0}^{(t)}\right)^{-1/2} }\right\Vert_2$ be the relative condition number at time $t$. Thus,
			\[
			\begin{array}{rcl}
				\norm{ w_{r,\star}^{(t)} 
					\,-\,
					w_{r}^{(t)} }_{\Sigma_{\nu^\star}^{(t)}}^2
				& \leq  &
				\displaystyle
				\norm{\left(\Sigma_{\nu_0}^{(t)}\right)^{-1/2} \Sigma_{\nu^\star}^{(t)} \left(\Sigma_{\nu_0}^{(t)}\right)^{-1/2}}
				\norm{ w_{r,\star}^{(t)} \,-\, w_{r}^{(t)} }_{\Sigma_{\nu_0}^{(t)}}^2
				\\[0.2cm]
				& \overset{(a)}{\leq} & 
				\displaystyle
				\frac{\kappa^{(t)}}{1-\gamma}
				\norm{ w_{r,\star}^{(t)} \,-\, w_{r}^{(t)} }_{\Sigma_{\nu^{(t)}}}^2
				\\[0.2cm]
				& \overset{(b)}{\leq}  & 
				\displaystyle
				\frac{\kappa^{(t)}}{1-\gamma}
				\left( {E}_r^{\nu^{(t)}} \left(w_{r}^{(t)} ; \theta^{(t)}\right) \, - \, 
				{E}_r^{\nu^{(t)}} \left(w_{r,\star}^{(t)} ; \theta^{(t)}\right) \right)
			\end{array}
			\]
			where we use 
			$(1-\gamma)\nu_0 \leq \nu_{\nu_0}^{\pi^{(t)}} \DefinedAs \nu^{(t)}$ in (a), and we get (b) by the same reasoning as bounding~\eqref{eq.err2}. Taking an expectation over the randomness in $w_{r}^{(t)}$ and $w_{r,\star}^{(t)}$ for the inequality above from both sides yields
			\begin{equation}\label{eq.err2g}
				\begin{array}{rcl}
					\mathbb{E}\sbr{ 
						\norm{ w_{r,\star}^{(t)}
							\, - \,
							w_{r}^{(t)} }_{\Sigma_{\nu^\star}^{(t)}}^2
					}
					& \leq & \displaystyle
					\mathbb{E}\sbr{  \frac{\kappa^{(t)}}{1-\gamma}
						\mathbb{E}
						\sbr{\, 
							{E}_r^{\nu^{(t)}} (w_{r}^{(t)} ; \theta^{(t)}) \, - \,
							{E}_r^{\nu^{(t)}} (w_{r,\star}^{(t)} ; \theta^{(t)})\,\vert\,\theta^{(t)} 
							\,}  }
					\\[0.4cm]
					& \leq & \displaystyle
					\mathbb{E} \sbr{\frac{\kappa^{(t)}}{1-\gamma}} \epsilon_{\text{\normalfont  est}}
					\\[0.4cm] 
					& \leq & \displaystyle
					\frac{\kappa\, \epsilon_{\text{\normalfont  est}}}{1-\gamma}
				\end{array}
			\end{equation}
		\end{subequations}
		where the last two inequalities are due to Assumption~\ref{as.errors+condtion}.
		
		Substitution of~\eqref{eq.err0g} and~\eqref{eq.err0g.2nd} to the right-hand side of~\eqref{eq.err} yields an upper bound on $\mathbb{E} \big[\,
		\text{\normalfont err}_r^{(t)} (\pi^\star) \,\big]$.
		By the same reasoning, we can establish a similar bound on $\mathbb{E} \big[\,
		\text{\normalfont err}_g^{(t)} (\pi^\star) \,\big]$. 
		Finally, application of these upper bounds to Lemma~\ref{lem.gap.violation} yields the desired result.
	\end{proof}
	
	
	We refine the constraint violation in Theorem~\ref{thm.convergence.general} to be zero by employing the same reasoning as in the proof of Corollary~\ref{cor.convergence.loglinear.zeroviolation}. We state it below as  Corollary~\ref{cor.convergence.general.zeroviolation} and leave out the proof to avoid repetition. 
    
	\begin{corollary}[Zero constraint violation: general policy parametrization]\label{cor.convergence.general.zeroviolation}
		Let Assumptions~\ref{as.slater}  and~\ref{as.smoothness} hold, let us fix a state distribution $\rho$ and a state-action distribution $\nu_0$, and replace the constraint of Problem~\eqref{eq.cmdp} by $V_g^{\pi}(\rho) \geq \bar{b}$, where $\bar{b}\DefinedAs b+\delta$ for some $\delta>0$.
		For $\epsilon<{\xi}/{2}$, there exists $\delta = \Theta(\epsilon)$ such that 
		if we choose $T = \Omega(1/\epsilon^2)$, $\eta_1=\eta_2={1}/{\sqrt{T}}$, and Assumption~\ref{as.errors+condtion} holds for $\epsilon_{\text{\normalfont  est}} = \epsilon_{\text{\normalfont  bias}} = O(\epsilon^2)$, then the iterates $\{\theta^{(t)}\}_{t=0}^{T-1}$ generated by the algorithm~\eqref{eq.NPGAGD.primalapproximation} using~\eqref{eq.NPGAGD.primal.generalapp} satisfy
		\[
		\begin{array}{rcl}
			\text{\normalfont(Optimality gap)}\;\; 
			\displaystyle 
			\mathbb{E}
			\left[\,
			\frac{1}{T} \sum_{t\,=\,0}^{T-1} \big(\,V_r^\star(\rho) \,-\, V_r^{(t)}(\rho)\, \big)
			\,\right]
			& = &
			\displaystyle
			O(\epsilon)
			\\[0.4cm]
			\text{\normalfont(Constraint violation)}\;\; 
			\displaystyle 
			\mathbb{E}
			\sbr{\,
				\frac{1}{T}\sum_{t \, = \,0 }^{T-1} \big(\,b\,-\,V_g^{(t)}(\rho)\,\big)
				\,}_+  
			& \leq & 
			0.
		\end{array}
		\]
	\end{corollary}

	\section{Sample-based NPG-PD algorithms}
	\label{sec.sample}
	
	We now leverage the convergence results established in Theorems~\ref{thm.convergence.loglinear} and~\ref{thm.convergence.general} to design two model-free algorithms that utilize sample-based estimates. In particular, we propose a sample-based extension of the NPG-PD algorithm~\eqref{eq.NPGAGD.primalapproximation} with function approximation and $\Lambda= [\,0, 2/((1-\gamma)\xi)\,]$ as follows
	\begin{equation}\label{eq.NPGAGD.sample-based}
		\begin{array}{rcl}
			\theta^{(t+1)} 
			& = &
			\theta^{(t)} 
			\,+\, 
			\dfrac{\eta_1}{1-\gamma}\, \hat{w}^{(t)} 
			\\[0.15cm]
			\lambda^{(t+1)} 
			& = &
			\calP_\Lambda \left( \lambda^{(t)} 
			\,- \, 
			\eta_2\, \big(\,\hat{V}_g^{(t)}(\rho) \,-\, b\,\big) \right)
		\end{array}
	\end{equation}
	where $\hat{w}^{(t)}$ and $\hat{V}_g^{(t)}(\rho)$ are the sample-based estimates of the gradient and the value function, respectively. At each time $t$, we can access a constrained MDP environment by executing a policy $\pi$ with terminating probability $1-\gamma$. For the minimization problem in~\eqref{eq.NPGAGD.primal.generalapp}, we can run the stochastic gradient descent (SGD) for $K$ rounds, $w_{\diamond,k+1} =\mathcal{P}_{\norm{w_{\diamond, k} }\,\leq\,W}\left( w_{\diamond, k} - \alpha_k \, G_{\diamond,k}\right)$, where $\alpha_k$ is the stepsize. Here, $G_{\diamond,k}$ is a sample-based estimate of the population gradient $\nabla_\theta E_\diamond^{\nu^{(t)}} (w_\diamond; \theta^{(t)})$:
	\[
	G_{\diamond,k} 
	\; = \; 
	2\, \Big((w_{\diamond,k})^\top \nabla_\theta\log \pi_\theta^{(t)}(a\,\vert\,s) \, - \, \hat{A}_\diamond^{(t)}(s,a) \Big) \,
	\nabla_\theta\log \pi_\theta^{(t)}(a\,\vert\,s)
	\] 
	where $\hat{A}_\diamond^{(t)}(s,a) \DefinedAs \hat{Q}_\diamond^{(t)}(s,a)-\hat{V}_\diamond^{(t)}(s)$, $\hat{Q}_\diamond^{(t)}(s,a)$ and $\hat{V}_\diamond^{(t)}(s)$ are undiscounted sums that are collected in Algorithm~\ref{alg.estimate.A}. In addition, we estimate $\hat{V}_g^{(t)}(\rho)$ using an undiscounted sum in Algorithm~\ref{alg.estimate.V}. As shown in Appendix~\ref{alg.general}, $G_{\diamond,k} $, $\hat{A}_\diamond^{(t)}(s,a)$, and $\hat{V}_g^{(t)}(\rho)$ are unbiased estimates and we approximate the gradient using the average of the SGD iterates $\hat{w}^{(t)} = 2{(K(K+1))}^{-1} \sum_{k\,=\,1}^{K} (k+1)(w_{r,k}+\lambda^{(t)}w_{g,k})$, which is an approximate solution to least-squares regression~\citep{lacoste2012simpler}.
	
	\begin{algorithm}[]
		\caption{ Sample-based NPG-PD algorithm with general policy parametrization }
		\label{alg.sample-based.general}
		\begin{algorithmic}[1]
			\STATE \textbf{Initialization}: Learning rates $\eta_1$ and $\eta_2$, number of SGD iterations $K$, SGD learning rate $\alpha_k = \frac{2}{\sigma_F (k+1)}$ for $k\geq0$.
			\STATE Initialize $\theta^{(0)}=0$, $\lambda^{(0)} =0$.
			\FOR{$t=0,\ldots,T-1$} 
			\STATE Initialize $w_{r,0}=w_{g,0} =0$.
			\FOR{$k=0,1,\ldots,K-1$ } 
			\STATE Estimate $\hat A_r(s,a)$ and  $\hat A_g(s,a)$ for some $(s,a) \sim \nu^{(t)}$, using Algorithm~\ref{alg.estimate.A} with policy $\pi_\theta^{(t)}$.
			\STATE Take a step of SGD
			\[
			\begin{array}{rcl}
				\!\!\!\! w_{r,k+1} & \!\! = \!\! & \mathcal{P}_{\norm{w_r} \,\leq\, W}\! \left( w_{r,k} - 2\alpha_k \big( (w_{r,k} )^\top \nabla_\theta \log \pi_\theta^{(t)}(s,a) - \hat{A}_r^{(t)}(s,a)\big) \nabla_\theta \log\pi_\theta^{(t)}(s,a) \right)
				\\[0.15cm]
				\!\!\!\! w_{g,k+1} & \!\! = \!\! & \mathcal{P}_{\norm{w_g} \,\leq\, W}\! \left( w_{g,k} - 2\alpha_k \big( (w_{g,k})^\top \nabla_\theta \log \pi_\theta^{(t)}(s,a) - \hat{A}_g^{(t)}(s,a)\big) \nabla_\theta \log\pi_\theta^{(t)}(s,a) \right).
			\end{array}
			\]
			\ENDFOR
			\STATE Set $\hat{w}^{(t)} =\hat{w}_r^{(t)} +\lambda^{(t)}\hat{w}_g^{(t)}$, where 
			\[
			\displaystyle
			\hat{w}_r^{(t)}
			\; = \;
			\frac{2}{K(K+1)} \sum_{k\,=\,0}^{K-1} (k+1)w_{r,k} 
			\; \text{ and } \;
			\hat{w}_g^{(t)} 
			\; = \;
			\frac{2}{K(K+1)} \sum_{k\,=\,0}^{K-1} (k+1)w_{g,k}.
			\]
			\STATE Estimate $\hat{V}_g^{(t)}(\rho)$ using Algorithm~\ref{alg.estimate.V} with policy $\pi_\theta^{(t)}$.
			\STATE Natural policy gradient primal-dual update 
			\[
			\begin{array}{rcl}
				\theta^{(t+1)} & = & \theta^{(t)} 
				\,+\, 
				\eta_1\,
				\hat{w}^{(t)} 
				\\[0.15cm]
				\lambda^{(t+1)} & = &  \calP_{[\,0,\,2/((1-\gamma)\xi)\,]}  \rbr{\,  \lambda^{(t)} 
					\, - \, 
					\eta_2\, \big(\,\hat{V}_g^{(t)}(\rho)\,-\,b\,\big)
					\,}.
			\end{array}
			\]
			\ENDFOR
		\end{algorithmic}
	\end{algorithm}
	
	\begin{algorithm}[]
		\caption{ $A$-Unbiased estimate ($\mathcal A_\diamond^\text{\normalfont est}$, $\diamond = r$ or $g$) }
		\label{alg.estimate.A}
		\begin{algorithmic}[1]
			\STATE \textbf{Input}: Initial state-action distribution $\nu_0$, policy $\pi$, discount factor $\gamma$. 
			\STATE Sample $(s_0,a_0) \sim \nu_0$, execute the policy $\pi$ with probability $\gamma$ at each step $h$; otherwise, accept $(s_h,a_h)$ as the sample. 
			\STATE Start with $(s_h,a_h)$, execute the policy $\pi$ with the termination probability $1-\gamma$. Once terminated, add all rewards/utilities from step $h$ onward as $\hat{Q}_\diamond^\pi(s_h,a_h)$ for $\diamond = r$ or $g$, respectively.
			\STATE Start with $s_h$, sample $a_h'\sim\pi(\cdot\,\vert\,s_h)$, and execute the policy $\pi$ with the termination probability $1-\gamma$. Once terminated, add all rewards/utilities from step $h$ onward as $\hat{V}_\diamond^\pi(s_h)$ for $\diamond = r$ or $g$, respectively.
			\STATE \textbf{Output}: $(s_h,a_h)$ and $\hat{A}_\diamond^\pi(s_h,a_h) \DefinedAs \hat{Q}_\diamond^\pi(s_h,a_h) - \hat{V}_\diamond^\pi(s_h)$, $\diamond = r$ or $g$.
		\end{algorithmic}
	\end{algorithm}
	
	\begin{algorithm}[] 
		\caption{ $V$-Unbiased estimate ($\mathcal V_g^\text{\normalfont est}$) }
		\label{alg.estimate.V}
		\begin{algorithmic}[1]
			\STATE \textbf{Input}: Initial state distribution $\rho$, policy $\pi$, discount factor $\gamma$. 
			\STATE Sample $s_0 \sim \rho$, execute the policy $\pi$ with the termination probability $1-\gamma$. Once terminated, add all utilities up as $\hat{V}_g^\pi(\rho)$.
			\STATE \textbf{Output}: $\hat{V}_g^\pi(\rho)$.
		\end{algorithmic}
	\end{algorithm}

	To establish the sample complexity of Algorithm~\ref{alg.sample-based.general}, we assume that the score function $\nabla_\theta \log \pi_\theta(a\,\vert\,s)$ has a bounded norm and the policy parametrization $\pi_\theta$ has a non-degenerate Fisher information matrix~\citep{zhang2020global,agarwal2021optimality,liu2020improved}.
	
	\begin{assumption}[Lipschitz policy]
		\label{as.Lipschitzpolicy}
		For $0\leq t<T$, the policy $\pi^{(t)}$ satisfies 
		\[
		\norm{\nabla_\theta \log \pi_\theta^{(t)}(a\,\vert\,s) }  
		\; \leq \; 
		L_\pi, \; \text{ where }  L_\pi>0.
		\]
	\end{assumption}
	
	\begin{assumption}[Fisher-non-degenerate policy]
		\label{as.Fishernonpolicy}
		There exists a $\sigma_F>0$ such that
		\[
		\Sigma_{\nu}^\theta   
		\; \succcurlyeq \; 
		\sigma_F  I
		\]
		for all $\nu \text{ and } \theta \in \mathbb{R}^d$, where $I$ is the identity matrix in $\mathbb{R}^{d\times d}$.
	\end{assumption}
	
	Assumption~\ref{as.Fishernonpolicy} holds for the Gaussian policy class when the parameterized mean has a full row-rank Jacobian and the covariance matrix is fixed~\citep{fatkhullin2023stochastic}, a policy class in the full rank exponential family~\citep{ding2022global}, and certain neural policies~\citep{liu2020improved}. We introduce it to tighten the sample complexity analysis, although this assumption does not necessarily hold for the tabular softmax policy~\citep{fatkhullin2023stochastic}. 
	
	In Theorem~\ref{thm.samplecomplexity.general}, we establish the sample complexity of Algorithm~\ref{alg.sample-based.general}.
	
	\begin{theorem}[Sample complexity: general policy parametrization]\label{thm.samplecomplexity.general}
		Let Assumptions \ref{as.slater}, \ref{as.smoothness}, \ref{as.Lipschitzpolicy}, and~\ref{as.Fishernonpolicy} hold and let us fix a state distribution $\rho$, a state-action distribution $\nu_0$, and $T>0$. If the iterates $\{(\theta^{(t)},\lambda^{(t)})\}_{t=0}^{T-1}$ are generated by the sample-based NPG-PD method described in Algorithm~\ref{alg.sample-based.general} with $\eta_1=\eta_2={1}/{\sqrt{T}}$ and $\alpha_k= 2/(\sigma_F (k+1))$, in which $K$ rounds of trajectory samples are used at each time $t$, then
		\[
		\begin{array}{rcl}
			\displaystyle
			\mathbb{E}
			\sbr{\,
				\frac{1}{T} \sum_{t\,=\,0}^{T-1} \big(\,V_r^\star(\rho) \,-\,V_r^{(t)}(\rho)\,\big)
				\,}
			& \leq &\displaystyle
			\frac{C_5}{(1-\gamma)^5}\frac{1}{\sqrt{T}} \,+\, \frac{1+2/\xi}{(1-\gamma)^3} \left( \sqrt{\epsilon_{\text{\normalfont  bias}}} + \sqrt{\frac{2\,\kappa\, G^2}{\sigma_F (K+1)}}\right)
			\\[0.4cm]
			\displaystyle
			\mathbb{E}
			\sbr{\,
				\frac{1}{T} \sum_{t\,=\,0}^{T-1} \big(\,b\,-\,V_g^{(t)}(\rho)\,\big)
				\,}_+
			& \leq  &\displaystyle
			\frac{C_6}{(1-\gamma)^4}\frac{1}{\sqrt T}  \,+\,
			\frac{2+\xi}{(1-\gamma)^2} \left( \sqrt{\epsilon_{\text{\normalfont  bias}}} + \sqrt{\frac{2\,\kappa \, G^2}{\sigma_F(K+1)}}\right)
		\end{array}
		\]
		where $C_5 \DefinedAs 2+\log |A|+5\beta W^2/\xi^2$, $C_6 \DefinedAs (2+\log |A|+\beta W^2)\xi + (2+4\beta W^2)/\xi$, and $G^2 \DefinedAs 4 ( W^2 L_\pi^2 + {2}/{(1-\gamma)^2} ) L_\pi^2$.
	\end{theorem}

	In Theorem~\ref{thm.samplecomplexity.general}, the sampling effect appears as an error of rate $1/\sqrt{K}$, where $K$ is the number of sampled trajectories. This rate follows the standard SGD result~\citep{lacoste2012simpler} and it can be increased to $1/{K}^{1/4}$ under less restrictive assumptions on the policy class~\citep{shamir2013stochastic}. 
	When $\epsilon_{\text{\normalfont  bias}} = 0$, it takes $O({1}/{\epsilon^4})$ sampled trajectories for Algorithm~\ref{alg.sample-based.general} to output an $\epsilon$-optimal policy.
	The proof of Theorem~\ref{thm.samplecomplexity.general} in Appendix~\ref{pf.samplecomplexity.general} follows the proof of Theorem~\ref{thm.convergence.general} except that we use sample-based estimates of gradients in the primal update and sample-based value functions in the dual update. Compared to \cite[Theorem~3]{ding2020natural}, the improved sample complexity from $O(1/\epsilon^8)$ to $O(1/\epsilon^4)$ is owed to a new regret-type primal-dual analysis in Section~\ref{subsec.loglinear.pf}.

	\begin{algorithm}[]
		\caption{ Sample-based NPG-PD algorithm with log-linear policy parametrization }
		\label{alg.sample-based.loglinear}
		\begin{algorithmic}[1]
			\STATE \textbf{Input}: Learning rates $\eta_1$ and $\eta_2$, number of SGD iterations $K$, SGD learning rate  $\alpha_k = \frac{2}{\sigma_F (k+1)}$ for $k\geq0$.
			\STATE Initialize $\theta^{(0)}=0$, $\lambda^{(0)} =0$.
			\FOR{$t=0,\ldots,T-1$} 
			\STATE Initialize $w_{r,0}=w_{g,0}=0$.
			\FOR{$k=0,1,\ldots,K-1$ } 
			\STATE Estimate $\hat Q_r^{(t)}(s,a)$ and $\hat Q_g^{(t)}(s,a)$ for some $(s,a)\sim \nu^{(t)}$, using Algorithm~\ref{alg.estimate.Q} with log-linear policy $\pi_\theta^{(t)}$. 
			\STATE Take a step of SGD
			\[
			\begin{array}{rcl}
				w_{r,k+1} & = & \mathcal{P}_{\norm{w_r} \,\leq\, W} \left( w_{r,k} \, - \, 2\alpha_k \, \big( \phi_{s,a}^\top  w_{r,k} - \hat{Q}_r^{(t)}(s,a)\big) \phi_{s,a} \right)
				\\[0.15cm]
				w_{g,k+1} & = & \mathcal{P}_{\norm{w_g} \,\leq\, W} \left( w_{g,k} \, - \, 2\alpha_k \, \big( \phi_{s,a}^\top w_{g,k}- \hat{Q}_g^{(t)}(s,a)\big) \phi_{s,a} \right).
			\end{array}
			\]
			\ENDFOR
			\STATE Set $\hat{w}^{(t)} =\hat{w}_r^{(t)} +\lambda^{(t)}\hat{w}_g^{(t)}$, where 
			\[
			\displaystyle
			\hat{w}_r^{(t)}
			\; = \;
			\frac{2}{K(K+1)} \sum_{k\,=\,0}^{K-1} (k+1)w_{r,k} 
			\; \text{ and } \;
			\hat{w}_g^{(t)} 
			\; = \;
			\frac{2}{K(K+1)} \sum_{k\,=\,0}^{K-1} (k+1)w_{g,k}.
			\]
			\STATE Estimate $\hat{V}_g^{(t)}(\rho)$ using Algorithm~\ref{alg.estimate.V} with log-linear policy $\pi_\theta^{(t)}$. 
			\STATE Natural policy gradient primal-dual update 
			\begin{equation}\label{eq.NPGAGD.sample-basedsoft}
				\begin{array}{rcl}
					\theta^{(t+1)} & = & \theta^{(t)} 
					\,+\, 
					\dfrac{\eta_1}{1-\gamma} \hat \,
					w^{(t)} 
					\\[0.15cm]
					\lambda^{(t+1)} 
					& = & 
					\calP_{[\,0,\,2/((1-\gamma)\xi)\,]} \rbr{\,  \lambda^{(t)} 
						\, - \, 
						\eta_2 \,  \big(\,\hat{V}_g^{(t)}(\rho)\,-\,b\,\big)\,}.
				\end{array}
			\end{equation}
			\ENDFOR
		\end{algorithmic}
	\end{algorithm}

	\begin{algorithm}[]
		\caption{ $Q$-Unbiased estimate ($\mathcal Q_\diamond^\text{\normalfont est}$, $\diamond = r$ or $g$) }
		\label{alg.estimate.Q}
		\begin{algorithmic}[1]
			\STATE \textbf{Input}: Initial state-action distribution $\nu_0$, policy $\pi$, discount factor $\gamma$. 
			\STATE Sample $(s_0,a_0) \sim \nu_0$, execute the policy $\pi$ with probability $\gamma$ at each step $h$; otherwise, accept $(s_h,a_h)$ as the sample.
			\STATE Start with $(s_h,a_h)$, execute the policy $\pi$ with the termination probability $1-\gamma$. Once terminated, add all rewards/utilities from step $h$ onward as $\hat{Q}_\diamond^\pi(s_h,a_h)$ for $\diamond = r$ or $g$, respectively.
			\STATE \textbf{Output}: $(s_h,a_h)$ and $\hat{Q}_\diamond^\pi(s_h,a_h)$, $\diamond = r$ or $g$.
		\end{algorithmic}
	\end{algorithm}
	
	Algorithm~\ref{alg.sample-based.loglinear} is utilized for the log-linear policy parametrization. For the feature $\phi_{s,a}$ that has bounded norm $\norm{\phi_{s,a}}\leq B$, the sample-based gradient in SGD has the second-order moment bound $G^2 \DefinedAs 4 ( W^2 B^2 + {2}/{(1-\gamma)^2} ) B^2$. In Theorem~\ref{thm.samplecomplexity.loglinear}, we establish the sample complexity for Algorithm~\ref{alg.sample-based.loglinear}; see Appendix~\ref{pf.samplecomplexity.loglinear} for the proof.
	
	\begin{theorem}[Sample complexity: log-linear policy parametrization]\label{thm.samplecomplexity.loglinear}
		Let Assumption~\ref{as.slater} hold and let us fix a state distribution $\rho$ and a state-action distribution $\nu_0$.
		If the iterates $\{(\theta^{(t)}, \lambda^{(t)})\}_{t=0}^{T-1}$ are generated by the sample-based NPG-PD method described in Algorithm~\ref{alg.sample-based.loglinear} with $\norm{\phi_{s,a}}\leq B$, $\eta_1=\eta_2={1}/{\sqrt{T}}$, and $\alpha_k = 2/(\sigma_F (k+1))$, in which $K$ rounds of trajectory samples are used at each time $t$, and there exists $\sigma_F>0$ such that $\mathbb{E}_{(s,a)\,\sim\,\nu^{(t)}}\sbr{ \phi_{s,a} \phi_{s,a}^\top }  
		\succcurlyeq 
		\sigma_F I$,
		then
		\[
		\begin{array}{rcl}
			\displaystyle
			\mathbb{E}
			\sbr{\,
				\frac{1}{T} \sum_{t\,=\,0}^{T-1} \big(\,V_r^\star(\rho) \,-\,V_r^{(t)}(\rho)\,\big)
				\,}
			& \leq &
			\displaystyle
			\frac{C_5}{(1-\gamma)^5}\frac{1}{\sqrt{T}} \,+\, \frac{2+4/\xi}{(1-\gamma)^3} \left( \sqrt{|A|\, \epsilon_{\text{\normalfont  bias}}} + \sqrt{\frac{2\kappa\, |A|\, G^2}{\sigma_F(K+1)}} \right)
			\\[0.4cm]
			\displaystyle
			\mathbb{E}
			\sbr{\,
				\frac{1}{T} \sum_{t\,=\,0}^{T-1} \big(\,b\,-\,V_g^{(t)}(\rho)\,\big)
				\,}_+
			& \leq  &\displaystyle
			\frac{C_6}{(1-\gamma)^4}\frac{1}{\sqrt T}  \,+\,
			\frac{4+2\xi}{(1-\gamma)^2} \left( \sqrt{|A|\, \epsilon_{\text{\normalfont  bias}}} + \sqrt{\frac{2\kappa\, |A| \, G^2}{\sigma_F(K+1)}} \right)
		\end{array}
		\]
		where $C_5 \DefinedAs 2+\log |A|+5\beta W^2/\xi^2$ and $C_6 \DefinedAs (2+\log |A|+\beta W^2)\xi + (2+4\beta W^2)/\xi$.
	\end{theorem}
	
	When we specialize the log-linear policy to be the softmax policy, Algorithm~\ref{alg.sample-based.loglinear} becomes a sample-based implementation of the NPG-PD method~\eqref{eq.NPG-PD-softmax} that utilizes the state-action value functions. In this case, $\epsilon_{\text{\normalfont  bias}} = 0$ and $B = 1$ in Theorem~\ref{thm.samplecomplexity.loglinear}. When there are no sampling effects, i.e., as $K\to\infty$, our rate $(1/\sqrt{T}, 1/\sqrt{T})$ matches the rate in Theorem~\ref{thm.convergence.softmax}. It takes $O({1}/{\epsilon^4})$ sampled trajectories for Algorithm~\ref{alg.sample-based.loglinear} to output an $\epsilon$-optimal policy.

\begin{remark}[Zero constraint violation for sample-based NPG-PD algorithms]
		As done in Corollary~\ref{cor.convergence.loglinear.zeroviolation}, we can refine the constraint violation in Theorem~\ref{thm.samplecomplexity.loglinear} to be zero. Given a small desired accuracy $\epsilon>0$, there exists $\delta = \Theta(\epsilon)$ such that 
		if $T =K= \Omega(1/\epsilon^2)$, and $\epsilon_{\text{\normalfont  bias}}  = O(\epsilon^2)$, then
		\[
		\begin{array}{rcl}
			\displaystyle 
			\mathbb{E}
			\sbr{\,
				\frac{1}{T}\sum_{t \, = \,0 }^{T-1} \big(\,b\,-\,V_g^{(t)}(\rho)\,\big)
				\,}_+  
			& \leq & 
			0.
		\end{array}
		\]
		Similarly, we can strengthen Theorem~\ref{thm.samplecomplexity.general} to achieve zero constraint violation.
\end{remark}
	
	\section{Computational experiments}
	\label{sec.experiments}
	
	\begin{figure}[h!]
		\begin{tabular}{c}
			Ant-v1
			\\[0.2cm]
			\begin{tabular}{cccc}
				{\rotatebox{90}{ \;\;\;\;\;\;\; average reward}}
				\!\!\!\!
				&
				{\includegraphics[scale=0.36]{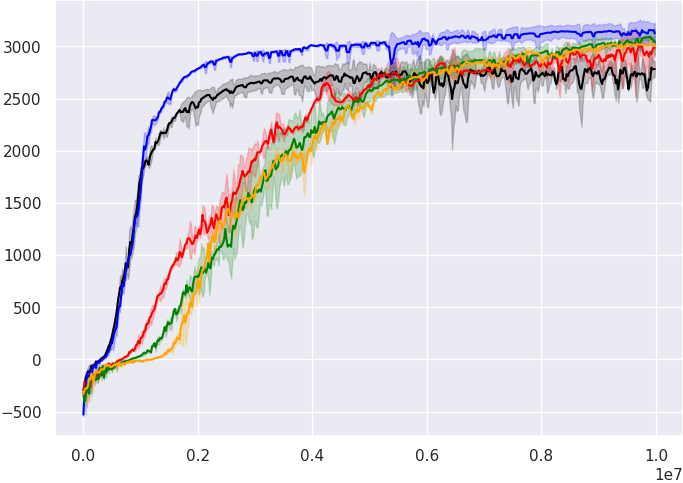}} 
				&  
				\;\;\;\;
				{\rotatebox{90}{ \;\;\;\;\;\;\;\;\; average cost}} 
				\!\!\!\!
				&
				{\includegraphics[scale=0.36]{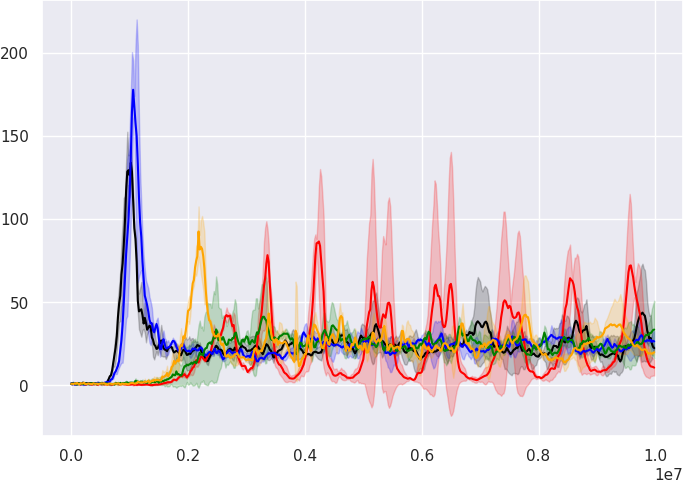}}
				\\[-0.1cm]
				{} & { \;\;\;\; sample count} & {} & { \;\;\;\;\;\; sample count}
			\end{tabular}
			\\[2.8cm]
			Humanoid-v1
			\\[0.2cm]
			\begin{tabular}{cccc}
				{\rotatebox{90}{ \;\;\;\;\;\;\; average reward}}
				\!\!\!\!
				&
				{\includegraphics[scale=0.36]{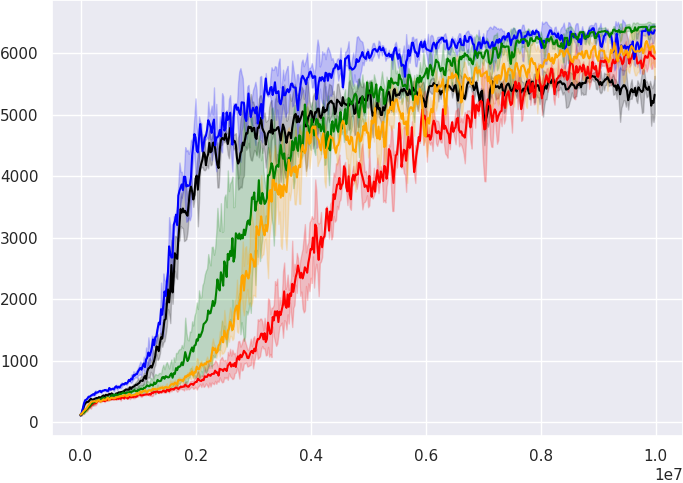}} 
				&  
				\;\;\;\;
				{\rotatebox{90}{ \;\;\;\;\;\;\;\;\; average cost}} 
				\!\!\!\!
				&
				{\includegraphics[scale=0.36]{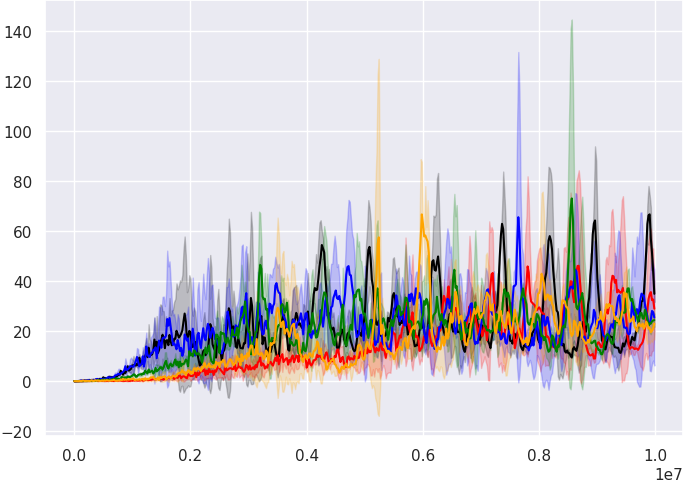}}
				\\[-0.1cm]
				{} & { \;\;\;\; sample count} & {} & { \;\;\;\;\;\; sample count}
			\end{tabular}
		\end{tabular}
		\caption{ Learning curves of NPG-PD method (\textbf{\color{myblue}---}, blue), CUP~\citep{yang2022constrained} (\textbf{\color{myred}---}, red), FOCOPS~\citep{zhang2020first} (\textbf{\color{orange}---}, orange), TRPOLag~\citep{ray2019benchmarking} (\textbf{\color{black}---}, black), and PPOLag~\citep{ray2019benchmarking} (\textbf{\color{mygreen}---}, green)  for Ant-v1 and Humanoid-v1 robotic tasks with the speed limit $25$. The vertical axes represent the average reward and the average cost (i.e., average speed). The solid lines show the means of $1000$ bootstrap samples obtained over $3$ random seeds and the shaded regions display the bootstrap $95\%$ confidence intervals. }
		\label{fig.speed limit last two}
	\end{figure}
	
	\begin{figure}[h!]
		\begin{tabular}{c}
			HalfCheetah-v1
			\\[0.2cm]
			\begin{tabular}{cccc}
				{\rotatebox{90}{ \;\;\;\;\;\;\; average reward}}
				\!\!\!\!
				&
				{\includegraphics[scale=0.36]{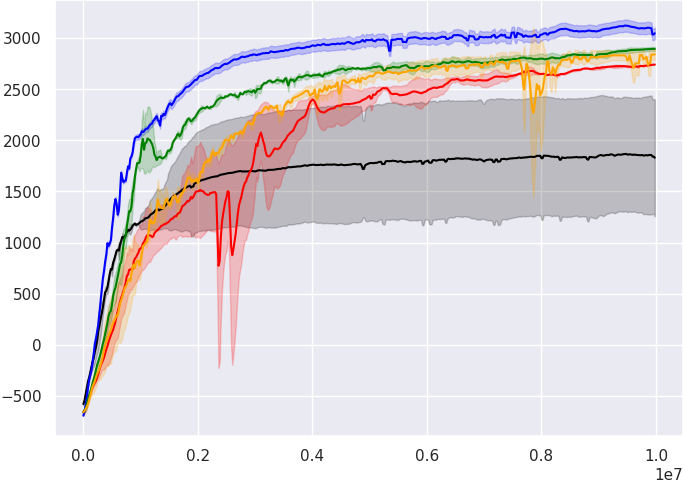}} 
				&  
				\;\;\;\;
				{\rotatebox{90}{ \;\;\;\;\;\;\;\;\;\;\;  average cost}} 
				\!\!\!\!
				&
				{\includegraphics[scale=0.36]{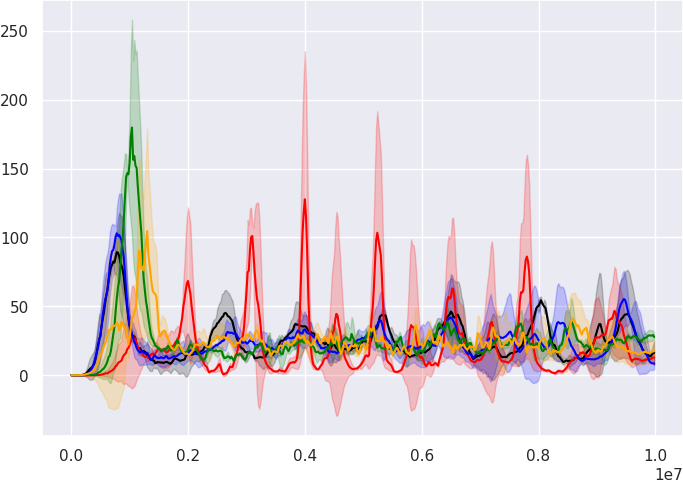}}
				\\[-0.1cm]
				{} & { \;\;\;\; sample count} & {} & { \;\;\;\;\;\; sample count}
			\end{tabular}
			\\[2.8cm]
			Walker2d-v1
			\\[0.2cm]
			\begin{tabular}{cccc}
				{\rotatebox{90}{ \;\;\;\;\;\;\; average reward}}
				\!\!\!\!
				&
				{\includegraphics[scale=0.36]{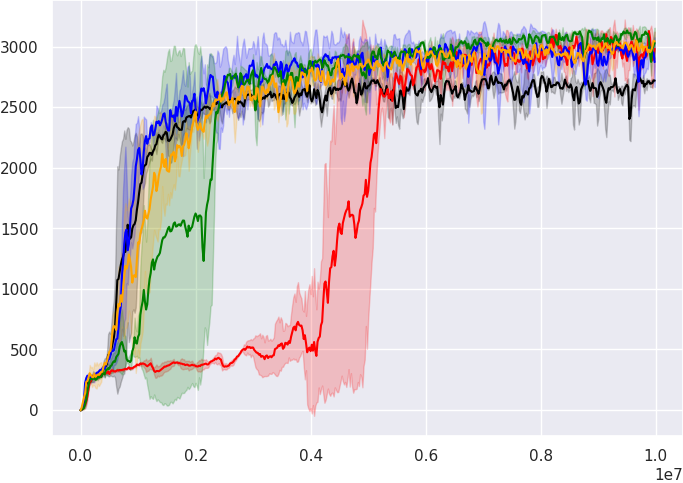}} 
				&  
				\;\;\;\;
				{\rotatebox{90}{ \;\;\;\;\;\;\;\;\;\;\;  average cost}} 
				\!\!\!\!
				&
				{\includegraphics[scale=0.36]{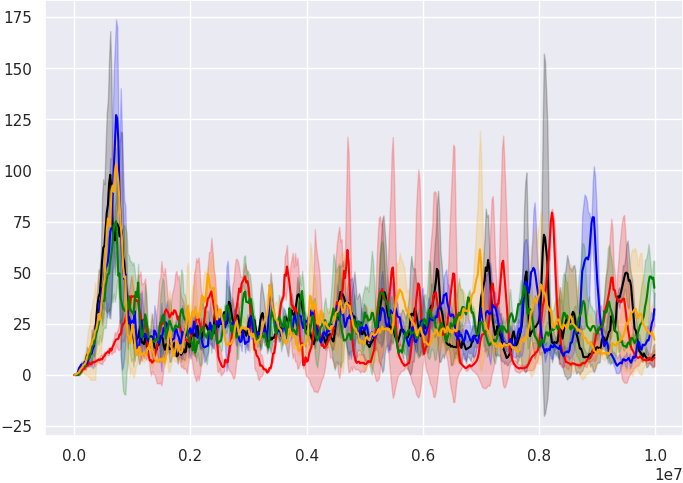}}
				\\[-0.1cm]
				{} & { \;\;\;\; sample count} & {} & { \;\;\;\;\;\; sample count}
			\end{tabular}
		\end{tabular}
		\caption{ Learning curves of NPG-PD method (\textbf{\color{myblue}---}, blue), CUP~\citep{yang2022constrained} (\textbf{\color{myred}---}, red), FOCOPS~\citep{zhang2020first} (\textbf{\color{orange}---}, orange), TRPOLag~\citep{ray2019benchmarking} (\textbf{\color{black}---}, black), and PPOLag~\citep{ray2019benchmarking} (\textbf{\color{mygreen}---}, green) for HalfCheetah-v1 and Walker2d-v1 robotic tasks with the speed limit $25$. The vertical axes represent the average reward and the average cost (i.e., average speed). The solid lines show the means of $1000$ bootstrap samples obtained over $3$ random seeds and the shaded regions display the bootstrap $95\%$ confidence intervals. }
		\label{fig.speed limit second two}
	\end{figure}
	
	\begin{figure}[h!]
		\begin{tabular}{c}
			Hopper-v1
			\\[0.2cm]
			\begin{tabular}{cccc}
				{\rotatebox{90}{ \;\;\;\;\;\;\; average reward}}
				\!\!\!\!
				&
				{\includegraphics[scale=0.36]{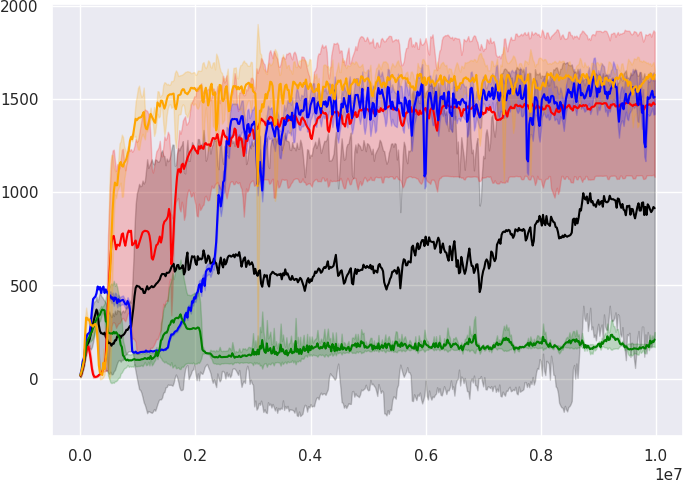}} 
				&  
				\;\;\;\;
				{\rotatebox{90}{ \;\;\;\;\;\;\;\;\;\;\; average cost}} 
				\!\!\!\!
				&
				{\includegraphics[scale=0.36]{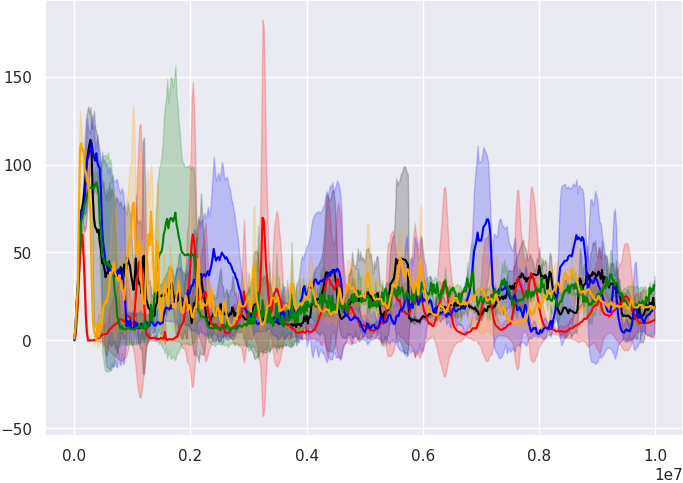}}
				\\[-0.1cm]
				{} & { \;\;\;\; sample count} & {} & { \;\;\;\;\;\; sample count}
			\end{tabular}
			\\[2.8cm]
			Swimmer-v1
			\\[0.2cm]
			\begin{tabular}{cccc}
				{\rotatebox{90}{ \;\;\;\;\;\;\; average reward}}
				\!\!\!\!
				&
				{\includegraphics[scale=0.36]{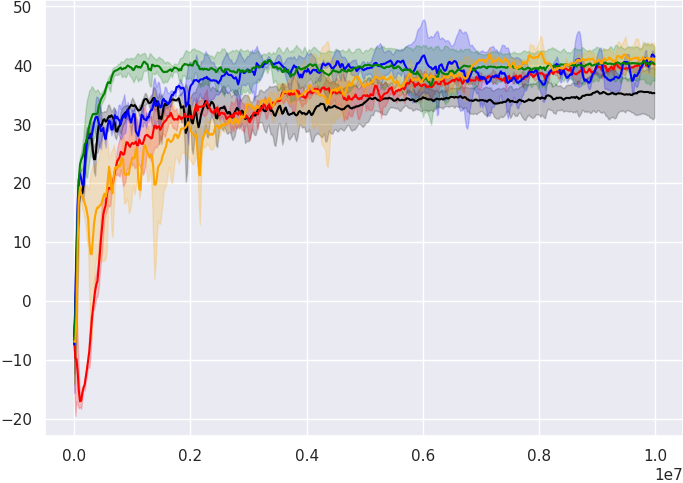}} 
				&  
				\;\;\;\;
				{\rotatebox{90}{ \;\;\;\;\;\;\;\;\;\;\; average cost}} 
				\!\!\!\!
				&
				{\includegraphics[scale=0.36]{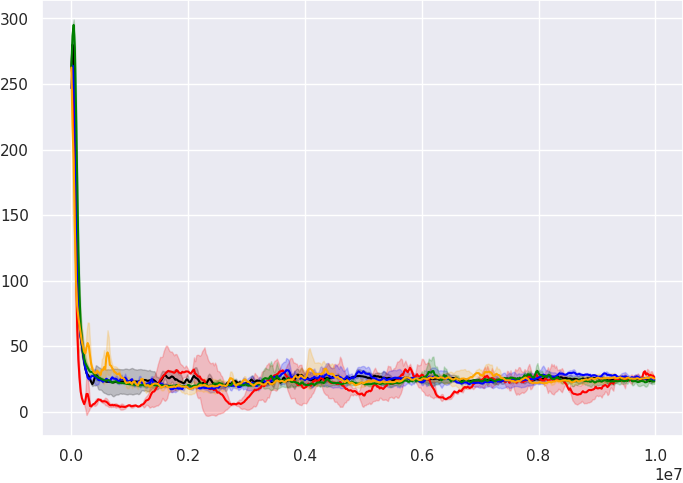}}
				\\[-0.1cm]
				{} & { \;\;\;\; sample count} & {} & { \;\;\;\;\;\; sample count}
			\end{tabular}
		\end{tabular}
		\caption{ Learning curves of NPG-PD method (\textbf{\color{myblue}---}, blue), CUP~\citep{yang2022constrained} (\textbf{\color{myred}---}, red), FOCOPS~\citep{zhang2020first} (\textbf{\color{orange}---}, orange), TRPOLag~\citep{ray2019benchmarking} (\textbf{\color{black}---}, black), and PPOLag~\citep{ray2019benchmarking} (\textbf{\color{mygreen}---}, green) for Hopper-v1 and Swimmer-v1 robotic tasks with the speed limit $25$. The vertical axes represent the average reward and the average cost (i.e., average speed). The solid lines show the means of $1000$ bootstrap samples obtained over $3$ random seeds and the shaded regions display the bootstrap $95\%$ confidence intervals. }
		\label{fig.speed limit first two}
	\end{figure}	
	
	We utilize a set of robotic tasks to demonstrate the effectiveness of our sample-based NPG-PD method described in Algorithm~\ref{alg.sample-based.general}. In our computational experiments, the  robotic agents are trained to move along a straight line or in a plane with speed limits for safety~\citep{zhang2020first}. We compare the performance of our NPG-PD algorithm with two classes of representative state-of-the-art methods: (i) two classical primal-dual policy search methods: Trust Region Policy Optimization based Lagrangian (TRPOLag) method and Proximal Policy Optimization based Lagrangian (PPOLag) method~\citep{ray2019benchmarking}; (ii) two methods that utilize the state-of-the-art policy optimization techniques: Constrained Update Projection (CUP) approach~\citep{yang2022constrained} and First Order Constrained Optimization in Policy Space (FOCOPS) algorithm~\citep{zhang2020first}. We conduct computational experiments in the OmniSafe framework~\citep{ji2024omnisafe} and implement robotic environments using the OpenAI Gym~\citep{brockman2016openai} for the MuJoCo simulators~\citep{todorov2012mujoco}. 
	
	We train six MuJoCo robotic agents to walk: Ant-v1, Humanoid-v1, HalfCheetah-v1, Walker2d-v1, Hopper-v1, and Swimmer-v1, while constraining the moving speed to be under a given threshold. Figure~\ref{fig.speed limit last two} shows that, in the first two tasks, our NPG-PD algorithm uniformly outperforms other four methods by reaching higher rewards while maintaining similar constraint satisfaction costs. This superior  performance of NPG-PD is also demonstrated in HalfCHeetah-v1 and Walker2d-v1 tasks in Figure~\ref{fig.speed limit second two}; in particular, we note that NPG-PD achieves a performance similar to that of PPOLag and that they both outperform the other three methods in Walker2d-v1 task. On the other hand, PPOLag does not perform well in Hopper-v1 in Figure~\ref{fig.speed limit first two}. For the last two tasks, Figure~\ref{fig.speed limit first two} shows a competitive performance of NPG-PD with two state-of-the-art methods: FOCOPS and CUP. Even though early oscillatory behavior slows down convergence of NPG-PD in Hopper-v1, it achieves higher rewards than CUP and FOCOPS. This demonstrates that NPG-PD not only converges faster than classical Lagrangian-based primal-dual methods but also matches the performance of state-of-the-art policy optimization methods.  
	
	\section{Concluding remarks}
	\label{sec.conclusion}
	
	We have proposed a Natural Policy Gradient Primal-Dual (NPG-PD) algorithm for solving the optimal control problems in constrained MDPs. Our algorithm utilized natural policy gradient ascent to update the primal variable and projected subgradient descent to update the dual variable. Although the underlying maximization involves a nonconcave objective function and a nonconvex constraint set, we have established global convergence for  both softmax or general smooth policy parametrizations, and have provided finite-sample complexity guarantees for two model-free extensions of the NPG-PD algorithm. 
	To the best of our knowledge, our work is the first to offer finite-time performance guarantees for policy-based primal-dual methods in the context of discounted infinite-horizon constrained MDPs. 
	
	In future work, we will address the oscillatory behavior commonly observed in primal-dual methods~\citep{stooke2020responsive,moskovitz2023reload}. To ensure policy-iterate convergence, recent works have incorporated optimism and regularization into single-timescale primal-dual algorithms~\citep{ding2023last,muller2024truly}. It is relevant to examine whether similar techniques can achieve policy-iterate convergence in the function approximation setting, e.g., using regularization~\citep{montenegro2024last}. We also aim to extend our policy gradient primal-dual algorithms to other constrained MDP settings, e.g., those with continuous state/action spaces~\citep{rozada2025deterministic}, as well as to those with risk-sensitive and resilient constraints~\citep{chow2017risk,ding2024resilient}. Additional future directions include further improving the sample efficiency of policy gradient primal-dual algorithms~\citep{mondal2024sample}, examining the robustness against model uncertainties~\citep{zhang2024distributionally}, and developing constrained policy optimization methods that work with offline datasets~\citep{hong2024primal,wei2024adversarially}.
	
	\begin{acks}
		Research of  D.~Ding and M.~R.~Jovanovi\'c was supported by the National Science Foundation under Awards ECCS-1708906 and ECCS-1809833. 
		Research of K.~Zhang and T.~Ba\c{s}ar was supported by the US Army Research Laboratory (ARL) Cooperative Agreement W911NF-17-2-0196 and by the Office of Naval Research (ONR) MURI Grant N00014-16-1-2710. K. Zhang also acknowledges the support from the Army Research Laboratory (ARL) Grant W911NF-24-1-0085,  the NSF CAREER Award ECCS-2443704, and the AFOSR YIP Award FA9550-25-1-0258.  T.~Ba\c{s}ar  also acknowledges the support provided by the Army Research Office (ARO) under Grant W911NF-24-1-0085.
	\end{acks}
	

	\newpage	
	\appendix
	
	\section{Proof of Lemma~\ref{lem.nonconvex}}\label{pf.nonconvex}
	
	We prove Lemma~\ref{lem.nonconvex} by providing a constrained MDP example as shown in Figure~\ref{fig.cmdp}. States $s_3$, $s_4$, and $s_5$ are three terminal states with zero reward and utility. We consider a non-trivial state $s_1$ with two actions: $a_1$ moving `up' and $a_2$ going `right', and the associated value functions are given by
	\[
	V_r^{\pi} (s_1) 
	\; = \; 
	\pi(a_2\,\vert\,s_1) \pi(a_1\,\vert\,s_2) 
	\]
	\[
	V_g^{\pi} (s_1) 
	\; = \;
	\pi(a_1\,\vert\,s_1)
	\, + \, 
	\pi(a_2\,\vert\,s_1) \pi(a_1\,\vert\,s_2).
	\]
	
	We consider the following two policies $\pi^{(1)}$ and $\pi^{(2)}$ using the softmax parametrization~\eqref{eq.softmax}:
	\[
	\theta^{(1)} 
	\; = \; 
	\rbr{ \log 1, \log x, \log x, \log 1}
	\]
	\[
	\theta^{(2)} 
	\; = \; 
	\rbr{ -\log 1, -\log x, -\log x, -\log 1}
	\]
	where the parameter is of the form $\rbr{\theta_{s_1,a_1}, \theta_{s_1,a_2}, \theta_{s_2,a_1}, \theta_{s_2,a_2}}$ for $x>0$. 
	
	First, we show that $V_r^{\pi}$ is not concave. We compute that
	\[
	\pi^{(1)} (a_1\,\vert\,s_1)
	\; = \; 
	\frac{1}{1+x},\;
	\pi^{(1)} (a_2\,\vert\,s_1)
	\; = \; \frac{x}{1+x},\;
	\pi^{(1)} (a_1\,\vert\,s_2)
	\; = \; 
	\frac{x}{1+x}
	\]
	\[
	V_r^{(1)}(s_1)  
	\; = \; 
	\rbr{\frac{x}{1+x}}^2,\;
	V_g^{(1)}(s_1)  
	\; = \; 
	\frac{1+x+x^2}{(1+x)^2}
	\]
	\[
	\pi^{(2)} (a_1\,\vert\,s_1)
	\; = \; 
	\frac{x}{1+x},\; 
	\pi^{(2)} (a_2\,\vert\,s_1)
	\; = \; 
	\frac{1}{1+x},\; 
	\pi^{(2)} (a_1\,\vert\,s_2)
	\; = \; 
	\frac{1}{1+x}
	\]
	\[
	V_r^{(2)}(s_1)  
	\; = \; 
	\rbr{\frac{1}{1+x}}^2,\;
	V_g^{(2)}(s_1)  
	\; = \; 
	\frac{1+x+x^2}{(1+x)^2}.
	\]
	
	Now, we consider a policy $\pi^{(\zeta)}$:
	\[
	\zeta\,\theta^{(1)} \,+\, (1-\zeta)\,\theta^{(2)} 
	\; = \; 
	\rbr{\log 1, \log\rbr{x^{2\zeta-1}}, \log\rbr{x^{2\zeta-1}}, \log 1 }
	\]
	for some $\zeta\in[0,1]$, which is defined on the segment between $\theta^{(1)}$ and $\theta^{(2)}$. Therefore, 
	\[
	\pi^{(1)} (a_1\,\vert\,s_1)
	\; = \; 
	\frac{1}{1+x^{2\zeta-1}},\;
	\pi^{(1)} (a_2\,\vert\,s_1)
	\; = \; 
	\frac{x^{2\zeta-1}}{1+x^{2\zeta-1}},\;
	\pi^{(1)} (a_1\,\vert\,s_2)
	\; = \; 
	\frac{x^{2\zeta-1}}{1+x^{2\zeta-1}}
	\]
	\[
	V_r^{(\zeta)}(s_1) 
	\; = \;
	\rbr{\frac{x^{2\zeta-1}}{1+x^{2\zeta-1}}}^2,\; 
	V_g^{(\zeta)}(s_1) 
	\; = \;
	\frac{1+x^{2\zeta-1}+(x^{2\zeta-1})^2}{(1+x^{2\zeta-1})^2}.
	\]
	
	When $x=3$ and $\zeta=\frac{1}{2}$, 
	\[
	\frac{1}{2} V_r^{(1)}(s_1)  \,+\, \frac{1}{2} V_r^{(2)}(s_1)  \;=\; \frac{5}{16} \;>\; V_r^{(\frac{1}{2})}(s_1) \;=\; \frac{4}{16}
	\]
	which implies that $V_r^{\pi}$ is not concave.
	
	When $x=10$ and $\zeta=\frac{1}{2}$, 
	\[
	V_g^{(1)}(s_1) \;=\;
	V_g^{(2)}(s_1) \;\geq\; 0.9 
	\;\text{ and }\;
	V_g^{(\frac{1}{2})}(s_1) \;=\; 0.75
	\]
	which shows that if we take a constraint offset $b=0.9$, then $V_g^{(1)}(s_1) = V_g^{(2)}(s_1) \geq b$, and $ V_g^{(\frac{1}{2})}(s_1) < b$ in which the policy $\pi^{(\frac{1}{2})}$ is infeasible. Therefore, the set $\{\theta \,\vert\,V_g^{\pi_\theta}(s)\geq b\}$ is not convex.
	
	\section{Proof of Theorem~\ref{thm.convergence.direct}}\label{pf.convergence.direct}
	
	Let us first recall the notion of occupancy measure~\citep{altman1999constrained}. An occupancy measure $q^\pi$ of a policy $\pi$ is defined as a set of distributions generated by executing $\pi$:
	\begin{equation}\label{eq.pi2ocm}
		q_{s,a}^\pi 
		\; = \; 
		\sum_{t \, = \,0 }^\infty \gamma^t \, P^\pi(\, s_t = s, a_t = a\,\vert\,s_0\sim\rho\,)
	\end{equation}
	for all $s\in S$ and $a\in A$, where $P^\pi \rbr{s_t=s, a_t=a\,\vert\,s_0 \sim \rho}$ is the probability of visiting a state-action pair $(s, a)$ under the policy $\pi$ for an initial state $s_0$. For brevity, we put all $q_{s,a}^\pi$ together as $q^\pi \in\mathbb{R}^{|S||A|}$ and $q_a^\pi = [\,q_{1,a}^\pi,\cdots,q_{|S|,a}^\pi\,]^\top$. For an action $a$, we collect all transition probabilities $P(s'\,\vert\,s,a)$ for  $s',s\in S$ to have shorthand notation $P_a \in\mathbb{R}^{|S|\times|S|}$. The occupancy measure $q^\pi$ has to satisfy a set of linear constraints given by $\mathcal{Q} \DefinedAs \{ q^\pi\in\mathbb{R}^{|S||A|} \,\vert\,\sum_{a\,\in\,A} (I - \gamma P_a^\top) q_a^\pi =\rho \text{ and } q^\pi\geq 0\}$.
	With a slight abuse of notation, we write $r\in[0,1]^{|S||A|}$ and $g\in[0,1]^{|S||A|}$.
	Thus, the value functions $V_{r}^{\pi}$, $V_{g}^{\pi}$: $S\to\mathbb{R}$ under the initial state distribution $\rho$ are linear in $q^\pi$:
	\[
	V_{r}^{\pi}(\rho) 
	\; = \; 
	\langle q^\pi,r\rangle \;\DefinedAs \; F_r(q^\pi)
	\; \text{ and  } \;
	V_{g}^{\pi}(\rho) 
	\; = \; 
	\langle q^\pi,g\rangle \;\DefinedAs \; F_g(q^\pi).
	\] 
	We are now in a position to consider the primal problem~\eqref{eq.cmdp.p} as a linear program:
	\begin{equation}\label{eq.cmdp.m}
		\begin{array}{c}
			\maximize\limits_{q^\pi\,\in\,\mathcal{Q}}
			\;
			F_r(q^\pi)
			\;\;
			\subject 
			\;\;
			F_g(q^\pi)
			\;\geq\; 
			b
		\end{array}
	\end{equation}
	where the maximization is over all occupancy measures $q^\pi\in\mathcal{Q}$. Once we compute a solution $q^\pi$, the associated policy solution $\pi$ can be recovered via
	\begin{equation}\label{eq.ocm2pi}
		\pi(a\,\vert\,s) 
		\; = \; 
		\frac{q_{s,a}^\pi}{\sum_{a\,\in\, A} q_{s,a}^\pi} \; \text{ for all } s\in S, a\in A.
	\end{equation}
	Abstractly, we let $\pi^q$: $\mathcal{Q}\to \Delta_{A}^{|S|}$ be a mapping from an occupancy measure $q^\pi$ to a policy $\pi$. Similarly, as defined by~\eqref{eq.pi2ocm} we let $q^\pi$: $\Delta_{A}^{|S|} \to \mathcal{Q}$ be a mapping from a policy $\pi$ to an occupancy measure $q^\pi$. Clearly, $q^\pi = (\pi^q)^{-1}$. 
	
	Despite the nonconvexity essence of~\eqref{eq.cmdp.p} in policy space, the reformulation~\eqref{eq.cmdp.m} reveals the underlying convexity in occupancy measure $q^\pi$. In Lemma~\ref{lem.BAP}, we exploit this convexity to show the average policy improvement over $T$ steps. 
	
	\begin{lemma}[Bounded average performance]\label{lem.BAP}
		Let assumptions in Theorem~\ref{thm.convergence.direct} hold.
		Then, the iterates $\{(\theta^{(t)},\lambda^{(t)})\}_{t\,=\,0}^{T-1}$ generated by the PG-PD method~\eqref{eq.PGAGD} satisfy
		\begin{equation}\label{eq.result}
			\frac{1}{T}\sum_{t \, = \,0 }^{T-1}
			\big(\,{ F_r( q^{\theta^\star} ) - F_r(q^{\theta^{(t)}})}\,\big)  
			\, + \, 
			\frac{1}{T}\sum_{t \, = \,0 }^{T-1} \lambda^{(t)}\, \big(\,{F_g ( q^{\theta^\star} ) -F_g(q^{\theta^{(t)}}) }\,\big)
			\; \leq \; 
			\dfrac{D_\theta L_\theta}{{T}^{1/4}}
		\end{equation}
		where 
		$D_\theta = \frac{8 |S|}{(1-\gamma)^2}\|{{d_\rho^{\pi^\star}}/{\rho}}\|_\infty^2$ and $L_\theta = \frac{2 |A| (1+2/\xi)}{(1-\gamma)^4} $.
		
	\end{lemma}
	
	\begin{proof}
		From the dual update in~\eqref{eq.PGAGD}, we have $0\leq \lambda^{(t)}\leq 2/((1-\gamma)\xi)$. 
		From the smooth property of the value functions under the direct policy parametrization~\cite[Lemma~D.3]{agarwal2021optimality}, we have
		\[
		\left| F_r(q^\theta) - F_r(q^{\theta^{(t)}}) 
		\, - \, \big\langle \nabla_\theta F_r(q^{\theta^{(t)}}), \theta-\theta^{(t)}\big\rangle \right|
		\; \leq \;
		\dfrac{\gamma|A|}{(1-\gamma)^3} \norm{\theta-\theta^{(t)}}^2.
		\]
		If we fix $\lambda^{(t)} \geq 0$, then
		\[
		\begin{array}{rcl}
			&& \!\!\!\! \!\!\!\! \!\!
			\left|(F_r+\lambda^{(t)} F_g)(q^\theta) \, - \, (F_r+\lambda^{(t)} F_g)(q^{\theta^{(t)}}) 
			\, - \, \big\langle \nabla_\theta F_r(q^{\theta^{(t)}})+\lambda^{(t)}\nabla_\theta F_g(q^{\theta^{(t)}}) , \theta-\theta^{(t)}\big\rangle \right|
			\\[0.2cm]
			& \leq &
			\dfrac{L_\theta}{2} \norm{\theta-\theta^{(t)}}^2.
		\end{array}
		\]
		Thus,
		\begin{equation}\label{eq.ascent}
			\begin{array}{rcl}
				(F_r+\lambda^{(t)} F_g)(q^\theta) 
				& \geq & \displaystyle
				(F_r+\lambda^{(t)} F_g)(q^{\theta^{(t)}}) 
				\, + \, \big\langle \nabla_\theta F_r(q^{\theta^{(t)}}) + \lambda^{(t)}\nabla_\theta F_g(q^{\theta^{(t)}}) , \theta - \theta^{(t)}\big\rangle 
				\\[0.2cm]
				&& -~ \dfrac{L_\theta}{2} \norm{\theta-\theta^{(t)}}^2
				\\[0.2cm]
				& \geq & (F_r+\lambda^{(t)} F_g)(q^\theta) \, - \, L_\theta\norm{\theta - \theta^{(t)}}^2.
			\end{array}
		\end{equation}
		We note that the primal update in~\eqref{eq.PGAGD} is equivalent to
		\[
		\begin{array}{rcl}
			\theta^{(t+1)}
			& = &\displaystyle
			\argmax_{\theta\,\in\,\Theta} \, \Bigg\{V_r^{\theta^{(t)}}(\rho) \, + \, \lambda^{(t)} V_g^{\theta^{(t)}}(\rho) 
			\\[0.2cm]
			&&  \;\;\;\;  \;\;\;\;   \;\;\;\;  \;\;\;\;
			\displaystyle +~ \big\langle \nabla_\theta V_r^{\theta^{(t)}}(\rho) +\lambda^{(t)} \nabla_\theta V_g^{\theta^{(t)}}(\rho), \theta - \theta^{(t)} \big\rangle 
			\, - \, \dfrac{1}{2\eta_1} \norm{\theta-\theta^{(t)}}^2\Bigg\}.
		\end{array}
		\]
		By taking $\eta_1={1}/{L_\theta}$ and $\theta=\theta^{(t+1)}$ in~\eqref{eq.ascent}, we have
		\begin{equation}\label{eq.Frg}
			\begin{array}{rcl}
				&& \!\!\!\! \!\!\!\! \!\!
				(F_r+\lambda^{(t)} F_g)(q^{\theta^{(t+1)}}) 
				\\[0.2cm]
				& \geq & \displaystyle \maximize_{\theta \,\in\,\Theta} \,\Bigg\{
				(F_r+\lambda^{(t)} F_g)(q^{\theta^{(t)}}) 
				\\[0.2cm]
				&& \;\;\;\;  \;\;\;\;  \;\;\;\;  \;\;\;\;
				+~ \big\langle \nabla_\theta F_r(q^{\theta^{(t)}}) \, + \, \lambda^{(t)}\nabla_\theta F_g(q^{\theta^{(t)}}) , \theta-\theta^{(t)}\big\rangle \, -\, \dfrac{L_\theta}{2} \norm{\theta-\theta^{(t)}}^2\Bigg\}
				\\[0.2cm]
				& \geq & \displaystyle \maximize_{\theta \,\in\,\Theta} \,\cbr{ (F_r+\lambda^{(t)} F_g)(q^\theta) \, - \, L_\theta\norm{\theta-\theta^{(t)}}^2 }
				\\[0.2cm]
				& \geq & \displaystyle \maximize_{\alpha \,\in\,[0,1]} \,\cbr{ (F_r+\lambda^{(t)} F_g)(q^{\theta_\alpha}) \, - \, L_\theta\norm{\theta_\alpha-\theta^{(t)}}^2  }
			\end{array}
		\end{equation}
		where $\theta_\alpha \DefinedAs \pi^q (\alpha q^{\theta^\star} + (1-\alpha) q^{\theta^{(t)}} )$, we apply~\eqref{eq.ascent} for the second inequality, and the last inequality is due to $\pi^q \circ q^{\pi} = \text{id}_{SA}$ and linearity of $q^{\theta}$ in $\theta$. Since $F_r$ and $F_g$ are linear in $q^\theta$,  we have
		\begin{equation}\label{eq.key1}
			(F_r+\lambda^{(t)} F_g)(q^{\theta_\alpha}) 
			\; = \;
			\alpha(F_r+\lambda^{(t)} F_g)( q^{\theta^\star} )
			\, + \,
			(1-\alpha) (F_r+\lambda^{(t)} F_g)(q^{\theta^{(t)}}).
		\end{equation}
		By the definition of $\pi^q$, 
		\[
		\displaystyle
		(\pi^q(q) - \pi^q(q'))_{sa} 
		\; = \; 
		\frac{1}{\sum_{a\,\in\, A}q_{sa}} (q_{sa}-q_{sa}') 
		\, + \,
		\frac{\sum_{a\,\in\, A}q_{sa}'-\sum_{a\,\in\, A}q_{sa}}{\sum_{a\,\in\, A}q_{sa}\sum_{a\,\in\, A}q_{sa}'} q_{sa}'
		\]
		which, together with $\|x+y\|^2 \leq 2\|x\|^2+2\|y\|^2$, gives
		\[
		\begin{array}{rcl}
			&& \!\!\!\! \!\!\!\! \!\!
			\norm{\pi^q(q) - \pi^q(q')}^2 
			\\[0.2cm]
			& \leq & \displaystyle 
			2 \sum_{s\,\in\, S}\sum_{a\,\in\, A}  \dfrac{(q_{sa}-q_{sa}')^2}{(\sum_{a\,\in\, A}q_{sa})^2}  
			\,+\, 2\sum_{s\,\in\,S}\sum_{a\,\in\, A} \rbr{\dfrac{\sum_{a\,\in\, A}q_{sa}'-\sum_{a\,\in\, A}q_{sa}}{\sum_{a\,\in\, A}q_{sa}\sum_{a\,\in\, A}q_{sa}'}}^2 (q_{sa}')^2
			\\[0.4cm]
			& \leq &  \displaystyle
			2 \sum_{s\,\in\, S} \dfrac{1}{(\sum_{a\,\in\, A}q_{sa})^2} \left( \sum_{a\,\in\, A} (q_{sa}-q_{sa}')^2 
			\, + \, \rbr{\sum_{a\,\in\, A}q_{sa}'-\sum_{a\,\in\, A}q_{sa}}^2 \right).
		\end{array}
		\]
		Therefore, 
		\[
		\begin{array}{rcl}
			&& \!\!\!\! \!\!\!\! \!\!
			\norm{\theta_\alpha-\theta^{(t)}}^2 
			\\[0.2cm]
			& = & \Big\Vert{ \pi^q \rbr{\alpha q^{\theta^\star} + (1-\alpha) q^{\theta^{(t)}} } - \pi^q \rbr{q^{\theta^{(t)}}}  }\Big\Vert^2
			\\[0.2cm]
			& \leq & \displaystyle
			\sum_{s\,\in\,S} \frac{2\alpha^2}{\rbr{\sum_{a\,\in\, A}q_{sa}^{\theta^{(t)}}}^2 } \left(\sum_{a\,\in\, A}\rbr{q_{sa}^{\theta^\star} - q_{sa}^{\theta^{(t)}}}^2 
			\, + \, \rbr{ \sum_{a\,\in\, A}q_{sa}^{\theta^{(t)}}-\sum_{a\,\in\, A}q_{sa}^{\theta^\star} }^2\right)
		\end{array}
		\]
		in which the upper bound further can be relaxed into
		\begin{equation}\label{eq.key2}
			\begin{array}{rcl}
				&&  \!\!\!\! \!\!\!\! \!\! 
				\displaystyle
				\sum_{s\,\in\,S} \frac{4\alpha^2}{\rbr{\sum_{a\,\in\, A}q_{sa}^{\theta^{(t)}}}^2 } \left(\rbr{\sum_{a\,\in\, A}q_{sa}^{\theta^\star}}^2 
				+ \rbr{\sum_{a\,\in\, A}q_{sa}^{\theta^{(t)}}}^2 \right)
				\\[0.2cm]
				& = &  \displaystyle
				4\alpha^2 \sum_{s\,\in\,S} \tfrac{ \rbr{d_\rho^{\pi^\star} (s)}^2 \, + \, \rbr{d_\rho^{\pi^{(t)}}(s)}^2}{ \rbr{d_\rho^{\pi^{(t)}}(s)}^2}
				\\[0.2cm]
				& \leq &  \displaystyle
				4\alpha^2 |S| \, + \, 4\alpha^2|S| \left\Vert{\frac{d_\rho^{\pi^\star}}{d_\rho^{\pi^{(t)}}}}\right\Vert_\infty^2
				\\[0.2cm]
				& \leq & \displaystyle
				4\alpha^2 |S|\left({1 \, + \, \frac{1}{(1-\gamma)^2}\left\Vert{\frac{d_\rho^{\pi^\star}}{\rho}}\right\Vert_\infty^2}\right)
				\\[0.2cm]
				& \leq &  \alpha^2 D_\theta
			\end{array}
		\end{equation}
		where we apply $d_\rho^{\pi^{(t)}} \geq (1-\gamma)\rho$ componentwise in the second inequality.
		
		We now apply~\eqref{eq.key1} and~\eqref{eq.key2} to~\eqref{eq.Frg},
		\[
		\begin{array}{rcl}
			&& \!\!\!\! \!\!\!\! \!\!
			(F_r+\lambda^{(t)} F_g)( q^{\theta^\star} ) \, - \, (F_r+\lambda^{(t)} F_g)(q^{\theta^{(t+1)}}) 
			\\[0.2cm]
			& \leq & \displaystyle \minimize_{\alpha \,\in\,[0,1]} \, \left\{ L_\theta\norm{\theta_\alpha-\theta^{(t)}}^2
			\, + \, (F_r+\lambda^{(t)} F_g)( q^{\theta^\star} ) \, - \, (F_r+\lambda^{(t)} F_g)(q^{\theta_\alpha})  \right\}
			\\[0.4cm]
			& \leq & \displaystyle \minimize_{\alpha \,\in\,[0,1]} \, \left\{
			\alpha^2 D_\theta L_\theta
			\, + \, (1-\alpha)\big( (F_r+\lambda^{(t)} F_g)( q^{\theta^\star} ) \, - \, (F_r+\lambda^{(t)} F_g)(q^{\theta^{(t)}})\big) \right\}
		\end{array}
		\]
		which further implies 
		\begin{equation}\label{eq.key3}
			\begin{array}{rcl}
				&&\!\!\!\! \!\!\!\! \!\!
				{(F_r+\lambda^{(t+1)} F_g)( q^{\theta^\star} ) \, - \, (F_r+\lambda^{(t+1)} F_g)(q^{\theta^{(t+1)}})}  
				\\[0.2cm]
				& \leq & \displaystyle \minimize_{\alpha \,\in\,[0,1]} \, \left\{  \alpha^2 D_\theta L_\theta
				\, + \, (1-\alpha) \big( (F_r + \lambda^{(t)} F_g)( q^{\theta^\star} ) \, - \, (F_r + \lambda^{(t)} F_g)(q^{\theta^{(t)}})\big) 
				\right\}
				\\[0.2cm]
				&&  -~(\lambda^{(t)} - \lambda^{(t+1)} ) (F_g( q^{\theta^\star} )  - F_g( q^{\theta^{(t+1)}} )).
			\end{array}
		\end{equation}
		We check the right-hand side of the inequality~\eqref{eq.key3}. By the dual update in~\eqref{eq.PGAGD}, it is easy to see that $-(\lambda^{(t)} - \lambda^{(t+1)} ) (F_g( q^{\theta^\star} )  - F_g( q^{\theta^{(t+1)}} )) \leq \vert\lambda^{(t)} - \lambda^{(t+1)}\vert /(1-\gamma)\leq{\eta_2}/{(1-\gamma)^2}$. We can solve the minimization problem in~\eqref{eq.key3} by taking $\alpha = 0$ if $\alpha^{(t)}<0$; $\alpha = 1$ if $\alpha^{(t)}>0$; $\alpha = \alpha^{(t)}$ if $\alpha^{(t)} \in [0, 1]$, where 
			\[
			\alpha^{(t)}
			\; \DefinedAs \;
			\frac{(F_r+\lambda^{(t)} F_g)( q^{\theta^\star} ) \, - \, (F_r+\lambda^{(t)} F_g)(q^{\theta^{(t)}})}{2D_\theta L_\theta}.
			\]
		
		We next discuss three cases.
		\begin{itemize}
			\item[(i)] When $\alpha^{(t)}<0$, we set $\alpha=0$ for~\eqref{eq.key3},
			\begin{equation}\label{eq.result1}
				(F_r+\lambda^{(t+1)} F_g)( q^{\theta^\star} ) \, - \, (F_r+\lambda^{(t+1)} F_g)(q^{\theta^{(t+1)}}) 
				\; \leq \;
				\dfrac{D_\theta L_\theta}{2\sqrt{T}};
			\end{equation}
			\item[(ii)] When $\alpha^{(t)}>1$, we set $\alpha=1$ that leads to $(F_r+\lambda^{(t+1)} F_g)( q^{\theta^\star} )- (F_r+\lambda^{(t+1)} F_g)(q^{\theta^{(t+1)}}) \leq \frac{3}{2} D_\theta L_\theta $, i.e., $\alpha^{(t+1)}\leq {3}/{4}$. Thus, this case reduces to the next case.
			\item[(iii)] When $0\leq \alpha^{(t)}\leq1$, we can express~\eqref{eq.key3} as
			\[
			\begin{array}{rcl}
				&& \!\!\!\! \!\!\!\! \!\!
				(F_r+\lambda^{(t+1)} F_g)( q^{\theta^\star} )- (F_r+\lambda^{(t+1)} F_g)(q^{\theta^{(t+1)}}) 
				\\[0.2cm]
				& \leq &  \rbr{1-\frac{(F_r+\lambda^{(t)} F_g)( q^{\theta^\star} ) - (F_r+\lambda^{(t)} F_g)(q^{\theta^{(t)}})}{4 D_\theta L_\theta}}
				\times\rbr{ (F_r+\lambda^{(t)} F_g)( q^{\theta^\star} ) - (F_r+\lambda^{(t)} F_g)(q^{\theta^{(t)}}) }  
				\\[0.2cm]
				&&+~ \dfrac{D_\theta L_\theta}{2\sqrt{T}}
			\end{array}
			\]
			or equivalently
			\begin{equation}\label{eq.result2}
				\alpha^{(t+1)} 
				\; \leq \; 
				\rbr{1-\frac{\alpha^{(t)}}{2}}\alpha^{(t)} \,+\, \frac{1}{4\sqrt{T}}.
			\end{equation}
			
			By choosing $\lambda^{(0)} = 0$ and $\theta^{(0)}$ such that $V_r^{\theta^{(0)}}(\rho)\geq V_r^{\theta^\star}(\rho)$, we know that $\alpha^{(0)} \leq 0$. Thus, $\alpha^{(1)} \leq {1}/(4\sqrt{T})$. By~\eqref{eq.result1}, the case $\alpha^{(1)}\leq 0$ is trivial. Without loss of generality, we assume that $0\leq\alpha^{(t)} \leq {1}/{{T}^{1/4}} \leq 1$. By induction over $t$ for~\eqref{eq.result2}, 
			\begin{equation}\label{eq.result3}
				\alpha^{(t+1)}
				\; \leq \; 
				\rbr{1-\frac{\alpha^{(t)}}{2}}\alpha^{(t)} \, + \, \frac{1}{4\sqrt{T}}
				\; \leq \; 
				\frac{1}{{T}^{1/4}}.
			\end{equation}
		\end{itemize}
		
		By combining~\eqref{eq.result1} and~\eqref{eq.result3}, and averaging over $t=0,1,\cdots,T-1$, we get the desired bound.
	\end{proof}
	
	\begin{proof}[Proof of Theorem~\ref{thm.convergence.direct}]
		
		\noindent\textbf{Bounding the optimality gap}. By the dual update~\eqref{eq.PGAGD} and $\lambda^{(0)} = 0$, it is convenient to bound $(\lambda^{(T)})^2 $ by
		\[
		\begin{array}{rcl}
			\rbr{\lambda^{(T)}}^2 
			&=&\displaystyle \sum_{t \, = \,0 }^{T-1} \rbr{(\lambda^{(t+1)} )^2
				\,-\,
				(\lambda^{(t)})^2}
			\\[0.2cm]
			& = &
			\displaystyle 
			2\eta_2 
			\sum_{t\,=\,0}^{T-1} \lambda^{(t)} \big(\,b\, -\, F_g(q^{\theta^{(t)}})\,\big) 
			\,+\, 
			\eta_2^2 \sum_{t\,=\,0}^{T-1} \big(\,F_g(q^{\theta^{(t)}})\, -\, b\,\big)^2
			\\[0.2cm]
			& \leq &
			\displaystyle 
			2\eta_2 \sum_{t\,=\,0}^{T-1} \lambda^{(t)} \big(\,{F_g(q^\star)\, -\, F_g(q^{\theta^{(t)}})}\,\big) 
			\,+\, 
			\frac{\eta_2^2\, T}{(1-\gamma)^2}
		\end{array}
		\]
		where the inequality is due to the feasibility of the optimal policy $\pi^\star$ or the associated occupancy measure $q^\star = q^{\theta^\star}$: $F_g(q^\star) \geq b$, and $|F_g(q^{\theta^{(t)}})-b|\leq {1}/(1-\gamma)$. The above inequality further implies
		\[
		- \frac{1}{T}\sum_{t\,=\,0}^{T-1} \lambda^{(t)} \big(\,{F_g(q^\star)\, -\, F_g(q^{\theta^{(t)}})}\,\big)
		\; \leq \; 
		\frac{\eta_2 }{2(1-\gamma)^2}.
		\]
		By substituting the above inequality into~\eqref{eq.result} in Lemma~\ref{lem.BAP}, we obtain the desired optimality gap bound, where we take $\eta_2={(1-\gamma)^2D_\theta L_\theta}/(2 \sqrt{T})$.
		
		\noindent\textbf{Bounding the constraint violation}. From the dual update in~\eqref{eq.PGAGD}, we have for any $\lambda\in[\,0,{2}/((1-\gamma)\xi)\,]$
		\[
		\begin{array}{rcl}
			&&\!\!\!\! \!\!\!\! \!\!
			\vert\lambda^{(t+1)} \, -\,  \lambda\vert^2 
			\\[0.2cm]
			& \overset{(a)}{\leq} & \big|{\lambda^{(t)} \, -\,  \eta_2\,  \big(\, F_g(q^{\theta^{(t)}})\, -\, b \, \big)  \, -\, \lambda}\big|^2
			\\[0.2cm]
			& \overset{(b)}{\leq} &\displaystyle \big|{{\lambda^{(t)} \, -\,  \lambda}}\big|^2 
			\, - \, 
			2\eta_2 \big(\, {F_g(q^{\theta^{(t)}})\, -\, b }\, \big)\big(\, {\lambda^{(t)} \, -\,  \lambda} \, \big) 
			\, + \, 
			\frac{\eta_2^2}{(1-\gamma)^2}
		\end{array}
		\]
		where $(a)$ is due to the non-expansiveness of projection $\calP_\Lambda$ and $(b)$ is due to $({F_g(q^{\theta^{(t)}})-b })^2\leq{1}/{(1-\gamma)^2}$. Summing it up from $t=0$ to $t=T-1$, and dividing it by $T$, yield
		\[
		\begin{array}{rcl}
			&&\!\!\!\! \!\!\!\! \!\!
			\displaystyle\frac{1}{T}\vert\lambda^{(T)} \, -\,  \lambda\vert^2  
			\, - \, 
			\frac{1}{T}\vert{{\lambda^{(0)}  \, -\,  \lambda}}\vert^2
			\\[0.2cm]
			& \leq &\displaystyle  
			-\,
			\frac{2\eta_2}{T}\sum_{t\,=\,0}^{T-1} \big(\, {F_g(q^{\theta^{(t)}})\, -\, b }\, \big)\big(\, {\lambda^{(t)} \,  -\, \lambda}\, \big)
			\, + \, 
			\frac{\eta_2^2}{(1-\gamma)^2}
		\end{array}
		\]
		which further implies
		\[
		\displaystyle\frac{1}{T}\sum_{t\,=\,0}^{T-1} \big(\, { F_g(q^{\theta^{(t)}}) \, -\,  b }\, \big)\big(\, {\lambda^{(t)}  \, -\,  \lambda}\, \big)
		\; \leq \; 
		\frac{\vert{{\lambda^{(0)} \, -\,  \lambda}}\vert^2}{2\eta_2T} 
		\, + \, 
		\frac{\eta_2}{2(1-\gamma)^2}.
		\]
		We note that $F_g(q^{\theta^\star}) \geq b$.
		By adding the inequality above to~\eqref{eq.result} in Lemma~\ref{lem.BAP} from both sides, we have
		\[
		\begin{array}{rcl}
			&& \!\!\!\! \!\!\!\! \!
			\displaystyle \frac{1}{T} \sum_{t\,=\,0}^{T-1} \big(\,{F_r(q^{\theta^\star})  \, -\, F_r(q^{\theta^{(t)}})}\,\big) \,+\, 
			\frac{\lambda}{T} \sum_{t\,=\,0}^{T-1} \big(\,{b\, -\,  F_g(q^{\theta^{(t)}})}\,\big)
			\\[0.2cm]
			&\!\!\leq\!\!&\displaystyle \dfrac{D_\theta L_\theta}{{T}^{1/4}}
			\,+\,
			\frac{1}{2\eta_2T}\vert{{\lambda^{(0)}  \, -\, \lambda}}\vert^2 
			\,+\,
			\frac{\eta_2}{2(1-\gamma)^2}.
		\end{array}
		\]
		We choose $\lambda = \frac{2}{(1-\gamma)\xi}$ if ${\sum_{t\,=\,0}^{T-1} \big({b -F_g(q^{\theta^{(t)}})}} \big)\geq 0$; otherwise $\lambda=0$. Thus, 
		\[
		\displaystyle F_r(q^{\theta^\star}) \, -\,  F_r(q') \,+\, 
		\frac{2}{(1-\gamma)\xi} \sbr{\,b \, -\,  F_g(q')\,}_+
		\; \leq \;
		\displaystyle
		\dfrac{D_\theta L_\theta}{{T}^{1/4}}
		\,+\,
		\frac{1}{2\eta_2(1-\gamma)^2\xi^2 T}
		\,+\,
		\frac{\eta_2}{2(1-\gamma)^2}
		\]
		where there exists $q'$ such that $F_r(q') \DefinedAs \frac{1}{T} \sum_{t\,=\,0}^{T-1}{ F_r(q^{\theta^{(t)}}) } $ and $F_g(q') \DefinedAs\frac{1}{T} \sum_{t\,=\,0}^{T-1}{ F_g(q^{\theta^{(t)}}) }$ by the definition of occupancy measure.
		
		From Lemma~\ref{lem.duality} (ii), we have $\lambda^\star \leq 1/((1-\gamma)\xi)$. Application of Lemma~\ref{lem.constraint} with $C= {2}/((1-\gamma)\xi)$ yields
		\[
		\sbr{\,b \, -\,   {F_g(q')}\,}_+
		\; \leq \;
		\dfrac{(1-\gamma)\xi D_\theta L_\theta }{{T}^{1/4}}
		\, + \, 
		\frac{1}{2\eta_2(1-\gamma)\xi T}
		\, + \,
		\frac{\eta_2\xi}{2(1-\gamma)}
		\]
		which readily leads to the desired constraint violation bound by noting that $$\frac{1}{T} \sum_{t\,=\,0}^{T-1} \big(\,{b \, -\, {F_g(q^{\theta^{(t)}})}}\,\big) 
		\;  =  \; 
		b \,-\,  {F_g(q')}$$
		and taking $\eta_2={(1-\gamma)^2D_\theta L_\theta}/(2 \sqrt{T})$ and $\|{{d_\rho^{\pi^\star}}/{\rho}}\|_\infty^2\geq (1-\gamma)^2$.  
		
	\end{proof}
	
	\section{Proof of Lemma~\ref{lem.npgpd}}
	\label{app.npgpd}
	
	The dual update follows Lemma~\ref{lem.duality}. Since
	$\lambda^\star\leq \rbr{V_r^{\star}(\rho)-V_r^{  \bar{\pi}} (\rho)}/ {\xi}$ with $0\leq V_r^{\star}$, $V_r^{  \bar{\pi}}\leq{1}/(1-\gamma)$, we take a projection interval $\Lambda=[\,0, {2}/((1-\gamma)\xi)\,]$ such that the upper bound ${2}/((1-\gamma)\xi)$ satisfies ${2}/((1-\gamma)\xi)\geq 2 \lambda^\star$.
	
	We now verify the primal update. 
	We expand the primal update in~\eqref{eq.NPGAGD} into
	\begin{equation}\label{eq.NPGAGD.primal}
		\theta^{(t+1)} 
		\;=\;
		\theta^{(t)} 
		\,+\, 
		\eta_1 F^\dagger_\rho(\theta^{(t)}) \nabla_\theta V_r^{\theta^{(t)}} (\rho)
		\,+\, 
		\eta_1 \lambda^{(t)} F^\dagger_\rho(\theta^{(t)}) \nabla_\theta V_g^{\theta^{(t)}} (\rho).
	\end{equation}
	
	We now deal with $F^\dagger_\rho(\theta^{(t)}) \nabla_\theta V_r^{\theta^{(t)}} (\rho)$ and $F^\dagger_\rho(\theta^{(t)}) \nabla_\theta V_g^{\theta^{(t)}} (\rho)$. For the first one, the proof begins with a solution to the following approximation error minimization problem:
	\[
	\minimize_{w \,\in\,\mathbb{R}^{|S||A|}}\; E_r (w) 
	\;\DefinedAs\;
	\mathbb{E}_{s\,\sim\,d_\rho^{\pi_\theta}, a \,\sim\, \pi_\theta(a\,\vert\,s)}  
	\sbr{ \,
		\rbr{A_r^{\pi_\theta}(s,a) - w^\top \nabla_\theta \log \pi_\theta(a\,\vert\,s)}^2 
		\,}.
	\]
	
	Using the Moore-Penrose inverse, an optimal solution reads
	\[
	w_r^\star 
	\; = \;
	F^\dagger_\rho(\theta) \mathbb{E}_{s\,\sim\, d_{\rho}^{\pi_\theta},a\,\sim\,\pi_\theta(a\,\vert\,s)} 
	\sbr{ \,
		\nabla_\theta \log \pi_\theta(a\,\vert\,s) \,A_r^{\pi_\theta,\lambda} (s,a) 
		\, }
	\; = \; 
	(1-\gamma)F^\dagger_\rho(\theta) \nabla_\theta V_r^{\pi_\theta,\lambda}(\rho)
	\] 
	where $F_\rho(\theta)$ is the Fisher information matrix induced by $\pi_\theta$. One key observation from this solution is that $w_r^\star$ is parallel to the NPG direction 
	$F^\dagger_\rho(\theta) \nabla_\theta V_r^{\pi_\theta,\lambda} (\rho)$.
	
	On the other hand, it is easy to verify that $A_r^{\pi_\theta}$ is a minimizer of $E_r(w)$. The softmax parametrization~\eqref{eq.softmax} implies that 
	\begin{equation}\label{eq.policyderivative}
		\frac{\partial \log\pi_\theta(a\,\vert\,s)}{\partial  \theta_{s',a'}}
		\; = \;
		\Ind{s=s'} \rbr{\Ind{a=a'} - \pi_\theta(a'\,\vert\,s)}
	\end{equation}
	where $\Ind{E}$ is the indicator function of event $E$ being true. Thus, we have
	\[
	w^\top \nabla_\theta \log \pi_\theta (a\,\vert\,s)
	\; = \;
	w_{s,a} 
	\, - \, 
	\sum_{a'\,\in\,A}^{} w_{s,a'} \pi_\theta (a'\,\vert\,s).
	\]
	
	The above equality together with the fact: $\sum_{a\,\in\, A}^{} \pi_\theta(a\,\vert\,s) A_r^{\pi_\theta,\lambda}(s,a) =0$, show that $E_r(A_r^{\pi_\theta}) = 0$. 
	However, $A_r^{\pi_\theta}$ may not be the unique minimizer. 
	We consider the following general form of possible solutions:
	\[
	A_r^{\pi_\theta} 
	\,+\, 
	u, 
	\;\text{ where }\; u \,\in\,\mathbb{R}^{|S||A|}.
	\]
	
	For any state $s$ and action $a$ such that $s$ is reachable under $\rho$, using~\eqref{eq.policyderivative} yields
	\[
	u^\top \nabla_\theta \log \pi_\theta(a\,\vert\,s)
	\;=\;
	u_{s,a} 
	\,-\, 
	\sum_{a'\,\in\,A} u_{s,a'} \pi_\theta(a'\,\vert\,s).
	\]
	
	Here, we make use of the following fact: $\pi_\theta$ is a stochastic policy with $\pi_\theta(a\,\vert\,s)>0$ for all actions $a$ in each state $s$, so that if a state is reachable under $\rho$, then it will also be reachable using $\pi_\theta$. Therefore, we require zero derivative at each reachable state:
	\[
	u^\top \nabla_\theta \log \pi_\theta(a\,\vert\,s) 
	\;=\; 
	0
	\]
	for all $(s,a)$, so that $u_{s,a}$ is independent of the action and becomes a constant $c_s$ for each $s$. Therefore, the minimizer of $E_r(w)$ is given by, up to some state-dependent offset
	\begin{equation}\label{eq.Fisher.r}
		F^\dagger_\rho(\theta) \nabla_\theta V_r^{\pi_\theta} (\rho) 
		\;=\; 
		\frac{A_r^{\pi_\theta}}{1-\gamma} \,+\,
		u
	\end{equation}
	where $u_{s,a} = c_s$ for some $c_s\in\mathbb{R}$ for each $s$ and $a$. 
	
	We can repeat the above procedure for $F^\dagger_\rho(\theta^{(t)}) \nabla_\theta V_g^{\theta^{(t)}} (\rho)$ and show
	\begin{equation}\label{eq.Fisher.g}
		F^\dagger_\rho(\theta) \nabla_\theta V_g^{\pi_\theta} (\rho) 
		\; = \; 
		\frac{A_g^{\pi_\theta}}{1-\gamma} \,+\, 
		v
	\end{equation}
	where $v_{s,a} = d_s$ for some $d_s\in\mathbb{R}$ for each state $s$ and action $a$. 
	
	Substituting~\eqref{eq.Fisher.r} and~\eqref{eq.Fisher.g} into the primal update~\eqref{eq.NPGAGD.primal} yields
	\[
	\theta^{(t+1)} 
	\; = \; 
	\theta^{(t)} 
	\, + \, 
	\frac{\eta_1 }{1-\gamma} \,\rbr{A_r^{(t)} +\lambda^{(t)}A_g^{(t)} }
	\, + \,
	\eta_1 \rbr{u+ \lambda^{(t)}v}
	\]
	\[
	\pi^{(t+1)}(a\,\vert\,s) 
	\; = \;
	\pi^{(t)}(a\,\vert\,s) 
	\,
	\frac{\exp \rbr{ \frac{\eta_1}{1-\gamma} \rbr{A_r^{(t)}(s,a)+\lambda^{(t)}A_g^{(t)}(s,a)} 
			\, + \,
			\eta_1 \rbr{c_s+\lambda^{(t)}d_s} }}{Z^{(t)}(s)}
	\]
	where the second equality also utilizes the normalization term $Z^{(t)}(s)$. Finally, we complete the proof by setting $c_s = d_s = 0$.
	
	\section{Sample-based NPG-PD algorithm with function approximation}\label{alg.general}
	
	We describe a sample-based NPG-PD algorithm with function approximation in Algorithm~\ref{alg.sample-based.general}. 
	We calculate the computational complexity of Algorithm~\ref{alg.sample-based.general} as follows: each round has expected length ${2}/(1-\gamma)$, yielding the expected number of total samples ${4KT}/(1-\gamma)$; the total number of gradient computations $\nabla_\theta\log\pi^{(t)}(a\,\vert\,s)$ is $2KT$; the total number of scalar multiplies, divides, and additions is $O (dKT+{KT}/(1-\gamma))$.
	
	The following unbiased estimates that are useful in our analysis.
	\[
	\begin{array}{rcl}
		\mathbb{E} 
		\sbr{\,
			\hat{V}_g^{(t)}(s) 
			\,} 
		&=&\displaystyle 
		\mathbb{E} \sbr{\,
			\sum_{k\,=\,0}^{K'-1} g(s_k,a_k)\, \bigg\vert\, \theta^{(t)}, s_0=s
			\,}
		\\[0.2cm]
		&=&\displaystyle \mathbb{E} \sbr{\,
			\sum_{k\,=\,0}^{\infty} \Ind{K'-1\geq k\geq 0} g(s_k,a_k)\, \bigg\vert\, \theta^{(t)}, s_0=s
			\,}
		\\[0.2cm]
		&\overset{(a)}{=}& \displaystyle\sum_{k\,=\,0}^{\infty} \mathbb{E} 
		\sbr{\,
			\mathbb{E}_{K'}\sbr{\Ind{K'-1\geq k\geq 0}} g(s_k,a_k)\, \bigg\vert\, \theta^{(t)}, s_0=s
			\, }
		\\[0.2cm]
		&\overset{(b)}{=}& \displaystyle\sum_{k\,=\,0}^{\infty} \mathbb{E} 
		\sbr{ \,
			\gamma^k g(s_k,a_k)\, \bigg\vert\, \theta^{(t)}, s_0=s
			\,}
		\\[0.2cm]
		&\overset{(c)}{=}& \displaystyle \mathbb{E} 
		\sbr{ \, \sum_{k\,=\,0}^{\infty}\gamma^k g(s_k,a_k)\, \bigg\vert\, \theta^{(t)}, s_0=s
			\,}
		\\[0.2cm]
		&=& \displaystyle{V}_g^{(t)}(s)
	\end{array}
	\] 
	where we apply the Monotone Convergence Theorem and the Dominated Convergence Theorem for $(a)$ and swap the expectation and the infinite sum in $(c)$, and in $(b)$ we use $\mathbb{E}_{K'}\sbr{\, \Ind{K'-1\geq k\geq 0} \,} = 1- P\rbr{K'< k} = \gamma^k$ since $K'\sim\text{Geo}(1-\gamma)$, a geometric distribution.
	
	By a similar agument as above,
	\[		
	\begin{array}{rcl}
		\mathbb{E}
		\sbr{\,
			\hat{Q}_r^{(t)}(s,a)
			\,} 
		&=&\displaystyle 
		\mathbb{E}
		\sbr{\, \sum_{k\,=\,0}^{K'-1} r(s_k,a_k)\,\bigg\vert\,\theta^{(t)}, s_0=s, a_0=a
			\,}
		\\[0.2cm]
		&=&\displaystyle \mathbb{E}
		\sbr{\, 
			\sum_{k\,=\,0}^{\infty} \Ind{K'-1\geq k\geq 0}r(s_k,a_k)\,\bigg\vert\,\theta^{(t)}, s_0=s, a_0=a
			\,}
		\\[0.2cm]
		&=&\displaystyle \sum_{k\,=\,0}^{\infty} \mathbb{E} 
		\sbr{\, \mathbb{E}_{K'}\sbr{\Ind{K'-1\geq k\geq 0}} r(s_k,a_k)\, \bigg\vert\, \theta^{(t)}, s_0=s, a_0=a
			\,}
		\\[0.2cm]
		&=&\displaystyle \sum_{k\,=\,0}^{\infty} \mathbb{E} 
		\sbr{\,
			\gamma^kr(s_k,a_k)\, \bigg\vert\, \theta^{(t)}, s_0=s, a_0=a
			\,}
		\\[0.2cm]
		&=&\displaystyle  \mathbb{E} 
		\sbr{\,
			\sum_{k\,=\,0}^{\infty} \gamma^kr(s_k,a_k)\, \bigg\vert\, \theta^{(t)}, s_0=s, a_0=a
			\,}
		\\[0.2cm]
		&=&\displaystyle {Q}_r^{(t)}(s,a).
	\end{array}
	\] 
	
	Therefore,
	\[
	\mathbb{E}\sbr{\,\hat{A}_r^{(t)}(s,a)\,} 
	\;=\; \mathbb{E}\sbr{\,\hat{Q}_r^{(t)}(s,a)\,}\,-\,\mathbb{E}\sbr{\,\hat{V}_r^{(t)}(s)\,}
	\;=\;
	{Q}_r^{(t)}(s,a)\,-\,{V}_r^{(t)}(s) 
	\;=\;
	{A}_r^{(t)}(s,a).
	\]
	
	We also provide a bound on the variance of $\hat{V}_g^{(t)}(s)$ as follows
	\[
	\begin{array}{rcl}
		\text{Var}\sbr{\,\hat{V}_g^{(t)}(s)\,} &=& \mathbb{E}\sbr{\,\rbr{\hat{V}_g^{(t)}(s) \, -\,  {V}_g^{(t)}(s)}^2\, \bigg\vert\, \theta^{(t)}, s_0=s\,} 
		\\[0.2cm]
		&=& \displaystyle
		\mathbb{E}
		\sbr{\,
			\rbr{\sum_{k\,=\,0}^{K'-1} g(s_k,a_k) \, -\,  {V}_g^{(t)}(s)}^2\, \bigg\vert\, \theta^{(t)}, s_0=s
			\,} 
		\\[0.2cm]
		&=& \displaystyle
		\mathbb{E}_{K'}
		\sbr{\, \mathbb{E}\sbr{\,\rbr{\sum_{k\,=\,0}^{K'-1} g(s_k,a_k) \, -\,  {V}_g^{(t)}(s)}^2\,}\, \bigg\vert\, K'
			\,} 
		\\[0.2cm]
		&\overset{(a)}{\leq}& \mathbb{E}_{K'}
		\sbr{\,
			\rbr{K'}^2\, \big\vert\, K'
			\,} 
		\\[0.2cm]
		&\overset{(b)}{=}&\displaystyle\frac{1}{(1-\gamma)^2}
	\end{array}
	\]
	where $(a)$ is due to $0\leq g(x_k,a_k)\leq 1$ and $V_g^{(t)}(s)\geq 0$ and $(b)$ is clear from $K' \sim\text{Geo}(1-\gamma)$. Similarly, we have the variance bound
	$\text{Var}\sbr{\hat{Q}_r^{(t)}(s,a)} \leq \frac{1}{(1-\gamma)^2}$.
	
	By the sampling scheme of Algorithm~\ref{alg.estimate.A}, we can show that $G_{r,k}$ is an unbiased estimate of the population gradient $\nabla_\theta E_r^{\nu^{(t)}} (w_r; \theta^{(t)}) $:
	\[
	\begin{array}{rcl}
		\mathbb{E}_{(s,a)\,\sim\,d^{(t)}}\sbr{\,G_{\diamond,k}\,} 
		& = & 2\mathbb{E}
		\sbr{\,
			\Big(w_{r,k}^\top \nabla_\theta\log \pi_\theta^{(t)}(a\,\vert\,s) \, -\,  \hat{A}_r^{(t)}(s,a) \Big) \, \nabla_\theta\log \pi_\theta^{(t)}(a\,\vert\,s)
			\,}
		\\[0.2cm]
		& = & 2\mathbb{E}
		\sbr{\,
			\Big(w_{r,k}^\top \nabla_\theta\log \pi_\theta^{(t)}(a\,\vert\,s) \, - \, \mathbb{E} \sbr{\hat{A}_r^{(t)}(s,a)\,\big\vert\,s,a} \Big)\, \nabla_\theta\log \pi_\theta^{(t)}(a\,\vert\,s)
			\,}
		\\[0.2cm]
		& = & 2\mathbb{E}
		\sbr{\,
			\Big(w_{r,k}^\top \nabla_\theta\log \pi_\theta^{(t)}(a\,\vert\,s) \,-\, {A}_r^{(t)}(s,a) \Big) \,
			\nabla_\theta\log \pi_\theta^{(t)}(a\,\vert\,s)
			\,}
		\\[0.2cm]
		& = & \nabla_{w_r} E_r^{\nu^{(t)}} (w_r; \theta^{(t)}).
	\end{array}
	\] 
	
	\section{Proof of Theorem~\ref{thm.samplecomplexity.general}}\label{pf.samplecomplexity.general}
	
	We first adapt Lemma~\ref{lem.gap.violation} to the sample-based case as follows. 
	
	\begin{lemma}[Sample-based regret/violation lemma]\label{lem.gap.violation.sample}
		Let Assumption~\ref{as.slater} hold
		and let us fix a state distribution $\rho$ and $T>0$. 
		Assume that $\log \pi_\theta(a\,\vert\,s)$ is $\beta$-smooth in $\theta$  for any $(s,a)$.
		If the iterates $\{(\theta^{(t)},\lambda^{(t)})\}_{t\,=\,0}^{T-1}$ are generated by Algorithm~\ref{alg.sample-based.general} with $\theta^{(0)} = 0$, $\lambda^{(0)}=0$, $\eta_1=\eta_2=1/{\sqrt{T}}$, and $\Vert \hat w_r^{(t)}\Vert$, $\Vert \hat  w_g^{(t)}\Vert\leq W$, then
		\[
		\begin{array}{rcl}
			\displaystyle  \mathbb{E}
			\sbr{\,
				\frac{1}{T} \sum_{t\,=\,0}^{T-1} \big(\,V_r^\star(\rho) \,-\,V_r^{(t)}(\rho)\,\big)
				\,} 
			& \leq &
			\displaystyle
			\frac{C_5}{(1-\gamma)^5} \frac{1}{\sqrt{T}}
			\, + \,
			\sum_{t\,=\,0}^{T-1} \frac{ \mathbb{E}\sbr{\text{\normalfont err}_r^{(t)} (\pi^\star)}}{(1-\gamma)T}
			\, + \,
			\sum_{t\,=\,0}^{T-1} \frac{2 \mathbb{E}\sbr{\text{\normalfont err}_g^{(t)} (\pi^\star)}}{(1-\gamma)^2\xi T}
			\\[0.4cm]
			\displaystyle \mathbb{E}
			\sbr{\,
				\frac{1}{T} \sum_{t\,=\,0}^{T-1} \big(\,{b\,-\,V_g^{(t)}(\rho)}\,\big)
				\,}_+
			& \leq &
			\displaystyle
			\frac{C_6}{(1-\gamma)^4}\frac{1}{\sqrt T}
			\, + \,
			\sum_{t\,=\,0}^{T-1} \frac{\xi \mathbb{E}\sbr{\text{\normalfont err}_r^{(t)} (\pi^\star)} }{T}
			\, + \,
			\sum_{t\,=\,0}^{T-1} \frac{2 \mathbb{E}\sbr{\text{\normalfont err}_g^{(t)}(\pi^\star)}}{(1-\gamma) T}
		\end{array}
		\]
		where $C_5 = 2+\log |A|+5\beta W^2/\xi^2$, $C_6 = (2+\log |A|+\beta W^2)\xi + (2+4\beta W^2)/\xi$, and 	
		\[
		\widehat{\text{\normalfont err}}_\diamond^{(t)} (\pi)
		\; \DefinedAs \; 
		\left\vert\,
		\mathbb{E}_{s\,\sim\,d_\rho^\pi} \mathbb{E}_{a\,\sim\,\pi(\cdot\,\vert\,s)} \sbr{A_\diamond^{(t)}(s,a) \,-\, \left(\hat w_\diamond^{(t)}\right)^\top \nabla_\theta \log \pi_\theta^{(t)}(a\,\vert\,s)}
		\,\right\vert,
		\; \text{ where } \;
		\diamond \, = \, r \text{ or }g.
		\]
	\end{lemma}
	\begin{proof}
		The smoothness of log-linear policy in conjunction with an application of Taylor's theorem to $\log \pi_\theta^{(t)}(a\,\vert\,s)$ yield
		\[
		\log \frac{\pi_\theta^{(t)}(a\,\vert\,s) }{\pi_\theta^{(t+1)}(a\,\vert\,s) } 
		\,+\, 
		\rbr{\theta^{(t+1)} \, -\, \theta^{(t)} }^\top \nabla_\theta \log \pi_\theta^{(t)}(a\,\vert\,s)  
		\; \leq \;
		\frac{\beta}{2} \left\Vert \theta^{(t+1)} \, -\, \theta^{(t)}  \right\Vert^2
		\]
		where $\theta^{(t+1)} - \theta^{(t)} = \frac{\eta_1}{1-\gamma} \hat w^{(t)}$. 
		We unload $d_\rho^{\pi^\star}$ as $d^\star$ since $\pi^\star$ and $\rho$ are fixed. Therefore,
		\[
		\begin{array}{rcl}
			&& \!\!\!\! \!\!\!\! \!\! \mathbb{E}_{s\,\sim\,d^\star} \sbr{\, D_\text{KL}\left(\pi^\star(\cdot\,\vert\,s)\,\big\Vert\,\pi_\theta^{(t)}(\cdot\,\vert\,s)\right) \,-\, D_\text{KL}\left(\pi^\star(\cdot\,\vert\,s)\,\big\Vert\,\pi_\theta^{(t+1)}(\cdot\,\vert\,s)\right) \,} 
			\\[0.2cm]
			& = & \displaystyle-~\mathbb{E}_{s\,\sim\,d^\star} \mathbb{E}_{a\,\sim\,\pi^\star(\cdot\,\vert\,s)}\sbr{\, \log \frac{\pi_\theta^{(t)}(a\,\vert\,s) }{\pi_\theta^{(t+1)}(a\,\vert\,s) }\,}
			\\[0.2cm]
			& \geq &\displaystyle \eta_1\mathbb{E}_{s\,\sim\,d^\star} \mathbb{E}_{a\,\sim\,\pi^\star(\cdot\,\vert\,s)} \sbr{\,\left( \hat w^{(t)}\right)^\top\nabla_\theta \log \pi_\theta^{(t)}(a\,\vert\,s)\,} 
			\,-\,
			\beta\frac{\eta_1^2}{2(1-\gamma)^2} \left\Vert \hat w^{(t)}\right\Vert^2
			\\[0.2cm]
			& = & \displaystyle\eta_1\mathbb{E}_{s\,\sim\,d^\star} \mathbb{E}_{a\,\sim\,\pi^\star(\cdot\,\vert\,s)} \sbr{\,\left(\hat w_r^{(t)}\right)^\top\nabla_\theta \log \pi_\theta^{(t)}(a\,\vert\,s)\,}
			\\[0.2cm]
			&&\displaystyle +~\eta_1 \lambda^{(t)}\mathbb{E}_{s\,\sim\,d^\star} \mathbb{E}_{a\,\sim\,\pi^\star(\cdot\,\vert\,s)} \sbr{\,\left(\hat w_g^{(t)}\right)^\top \nabla_\theta \log \pi_\theta^{(t)}(a\,\vert\,s)\,} 
			\,-\,
			\beta\frac{\eta_1^2}{2(1-\gamma)^2} \left\Vert \hat w^{(t)}\right\Vert^2
			\\[0.4cm]
			& = &\displaystyle \eta_1\mathbb{E}_{s\,\sim\,d^\star} \mathbb{E}_{a\,\sim\,\pi^\star(\cdot\,\vert\,s)} A_r^{(t)}(s,a)
			\,+\,
			\eta_1 \lambda^{(t)}\mathbb{E}_{s\,\sim\,d^\star} \mathbb{E}_{a\,\sim\,\pi^\star(\cdot\,\vert\,s)} A_g^{(t)}(s,a) 
			\\[0.2cm]
			&&\displaystyle+~\eta_1\mathbb{E}_{s\,\sim\,d^\star} \mathbb{E}_{a\,\sim\,\pi^\star(\cdot\,\vert\,s)} \!\sbr{\,\left({\hat w_r^{(t)}\!+\!\lambda^{(t)} \hat w_g^{(t)}}\right)^\top \nabla_\theta \log \pi_\theta^{(t)}(a\,\vert\,s) \!-\! \big({A_r^{(t)}(s,a)\!+\!\lambda^{(t)} A_g^{(t)}(s,a)}\big)\,}
			\\[0.2cm]
			&&\displaystyle-~\beta\frac{\eta_1^2}{(1-\gamma)^2} \Big(\left\Vert\hat w_r^{(t)}\right\Vert^2 + \big(\lambda^{(t)}\big)^2 \left\Vert \hat w_g^{(t)}\right\Vert^2\Big)
			\\[0.4cm]
			&\geq&\displaystyle \eta_1(1-\gamma) \big(\,{V_r^\star(\rho) - V_r^{(t)}(\rho)} \,\big)
			\,+\,
			\eta_1(1-\gamma) \lambda^{(t)} \big(\,{ V_g^\star(\rho) \, -\,  V_g^{(t)}(\rho)}\,\big)
			\\[0.2cm]
			&&\displaystyle-~\eta_1  \widehat{\text{\normalfont err}}_r^{(t)} (\pi^\star)
			\displaystyle
			\, -\,
			\eta_1 \lambda^{(t)}  \widehat{\text{\normalfont err}}_g^{(t)} (\pi^\star)
			\,-\,
			\beta\frac{\eta_1^2\, W^2}{(1-\gamma)^2} 
			\,-\,
			\beta\frac{\eta_1^2\, W^2}{(1-\gamma)^2}  \big(\lambda^{(t)}\big)^2
		\end{array}
		\]
		where $\hat w^{(t)} = \hat w_r^{(t)}+\lambda^{(t)} \hat w_g^{(t)}$ for a given $\lambda^{(t)}$, in the last inequality we apply the performance difference lemma, notation of $\widehat{\text{\normalfont err}}_r^{(t)} (\pi^\star)$ and $\widehat{\text{\normalfont err}}_g^{(t)} (\pi^\star)$, and $\Vert \hat w_r^{(t)}\Vert$, $ \Vert \hat w_g^{(t)}\Vert\leq W$.
		
		Rearranging the inequality above leads to
		\[
		\begin{array}{rcl}
			&& \!\!\!\! \!\!\!\! \!\! V_r^\star(\rho) \,-\, V_r^{(t)}(\rho) 
			\\[0.2cm]
			& \leq &\displaystyle \frac{1}{1-\gamma} \frac{1}{\eta_1} \mathbb{E}_{s\,\sim\,d^\star} \sbr{\, D_\text{KL}\left(\pi^\star(\cdot\,\vert\,s)\,\big\Vert\,\pi_\theta^{(t)}(\cdot\,\vert\,s)\right) \,-\, D_\text{KL}\left(\pi^\star(\cdot\,\vert\,s)\,\big\Vert\,\pi_\theta^{(t+1)}(\cdot\,\vert\,s)\right) \,}  
			\\[0.4cm]
			&& \displaystyle +~\frac{1}{1-\gamma} \widehat{\text{\normalfont err}}_r^{(t)} (\pi^\star)  \,+\, 
			\frac{2}{(1-\gamma)^2 \xi } \widehat{\text{\normalfont err}}_g^{(t)} (\pi^\star)
			\,+\, 
			\beta\frac{\eta_1 W^2}{(1-\gamma)^3}  
			\,+\, 
			\beta\frac{4\eta_1W^2}{(1-\gamma)^5\xi^2} 
			\\[0.4cm]
			&& \displaystyle
			-~ \lambda^{(t)} \big(\,{ V_g^\star(\rho) \, -\,  V_g^{(t)}(\rho)}\,\big)
		\end{array}
		\]
		where we utilize $0\leq \lambda^{(t)} \leq 2/((1-\gamma)\xi)$ from the dual update of Algorithm~\ref{alg.sample-based.general}.
		
		Therefore,
		\[
		\begin{array}{rcl}
			&& \!\!\!\! \!\!\!\! \!\!\displaystyle \frac{1}{T} \sum_{t\,=\,0}^{T-1} \big(\, V_r^\star(\rho) \,-\, V_r^{(t)}(\rho) \,\big)
			\\[0.2cm]
			&\leq&\displaystyle \frac{1}{(1-\gamma)\eta_1 T}\sum_{t \, = \,0 }^{T-1} \mathbb{E}_{s\,\sim\,d^\star} \sbr{\, D_\text{KL}\left(\pi^\star(\cdot\,\vert\,s)\,\big\Vert\,\pi_\theta^{(t)}(\cdot\,\vert\,s)\right) \,-\, D_\text{KL}(\pi^\star\left(\cdot\,\vert\,s)\,\big\Vert\,\pi_\theta^{(t+1)}(\cdot\,\vert\,s)\right) \,}  
			\\[0.2cm]
			&&\displaystyle +~\frac{1}{(1-\gamma)T}\sum_{t\,=\,0}^{T-1} \widehat{\text{\normalfont err}}_r^{(t)} (\pi^\star)  \,+\, 
			\frac{2}{(1-\gamma)^2\xi T}\sum_{t\,=\,0}^{T-1}  \widehat{\text{\normalfont err}}_g^{(t)} (\pi^\star)
			\,+\,
			\beta\frac{\eta_1W^2 }{(1-\gamma)^3} 
			\,+\, 
			\beta\frac{4\eta_1W^2}{(1-\gamma)^5\xi^2}  
			\\[0.2cm]
			&& \displaystyle 
			-~\frac{1}{T}\sum_{t\,=\,0}^{T-1}\lambda^{(t)} \big(\,{ V_g^\star(\rho) - V_g^{(t)}(\rho)}\,\big)
			\\[0.2cm]
			&\leq&\displaystyle \frac{\log|A|}{(1-\gamma)\eta_1 T}
			\,+\,
			\frac{1}{(1-\gamma)T}\sum_{t\,=\,0}^{T-1} \widehat{\text{\normalfont err}}_r^{(t)} (\pi^\star)
			\,+\,
			\frac{2}{(1-\gamma)^2\xi T}\sum_{t\,=\,0}^{T-1}  \widehat{\text{\normalfont err}}_g^{(t)} (\pi^\star)
			\\[0.2cm]
			&&  \displaystyle+~\beta\frac{\eta_1W^2}{(1-\gamma)^3} 
			\,+\, 
			\beta\frac{4\eta_1W^2}{(1-\gamma)^5\xi^2} 
			\,+\, 
			\frac{1}{T}\sum_{t\,=\,0}^{T-1}\lambda^{(t)} \big(\,{ V_g^\pi(\rho) \, -\,  V_g^{(t)}(\rho)}\,\big)
		\end{array}
		\]
		where in the last inequality we take a telescoping sum of the first sum and drop a non-positive term. Taking the expectation over the randomness in sampling on both sides of the inequality above yields
		\begin{equation}\label{eq.keyone.sample}
			\begin{array}{rcl}
				&& \!\!\!\! \!\!\!\! \!\! 
				\displaystyle
				\mathbb{E}\sbr{\,
					\frac{1}{T} \sum_{t\,=\,0}^{T-1} \big(\,V_r^\star(\rho) \, -\,  V_r^{(t)}(\rho)\,\big)\,}  
				\,+\, 
				\mathbb{E}\sbr{\, 
					\frac{1}{T}\sum_{t\,=\,0}^{T-1}\lambda^{(t)} \big(\,{ V_g^\star(\rho)\,  -\,  V_g^{(t)}(\rho)}\,\big)\,}
				\\[0.2cm]
				& \leq & \displaystyle
				\frac{\log|A|}{(1-\gamma)\eta_1 T}
				\,+\,
				\frac{1}{(1-\gamma)T}\sum_{t\,=\,0}^{T-1}\mathbb{E}\sbr{\, \widehat{\text{\normalfont err}}_r^{(t)} (\pi^\star)\,}
				\,+\,
				\frac{2}{(1-\gamma)^2\xi T}\sum_{t\,=\,0}^{T-1}\mathbb{E}\sbr{\, \widehat{\text{\normalfont err}}_g^{(t)} (\pi^\star)\,}
				\\[0.2cm]
				&& \displaystyle+~\beta\frac{\eta_1W^2}{(1-\gamma)^3} 
				\,+\, 
				\beta\frac{4\eta_1W^2}{(1-\gamma)^5\xi^2}.
			\end{array}
		\end{equation}
		
		\noindent\textbf{Proving the first inequality}. From the dual update in Algorithm~\ref{alg.sample-based.general}, we have
		\[
		\begin{array}{rcl}
			0 \; \leq \; 
			\rbr{\lambda^{(T)}}^2 &=&\displaystyle \sum_{t \, = \,0 }^{T-1} \big(\,(\lambda^{(t+1)} )^2\,-\,(\lambda^{(t)})^2\,\big)
			\\[0.2cm]
			&\leq& \displaystyle\sum_{t \, = \,0 }^{T-1} \rbr{ \big(\lambda^{(t)} \, -\,  \eta_2\,  \big(\, \hat{V}_g^{(t)}(\rho)\, -\, b\, \big)\big)^2\,-\,(\lambda^{(t)})^2 }
			\\[0.2cm]
			&=& \displaystyle2\eta_2 \sum_{t\,=\,0}^{T-1} \lambda^{(t)} \big(\, b\, -\, \hat{V}_g^{(t)}(\rho)\, \big)
			\,+\,
			\eta_2^2 \sum_{t\,=\,0}^{T-1}\big(\, \hat{V}_g^{(t)}(\rho)\, -\, b\, \big)^2
			\\[0.2cm]
			&\leq&\displaystyle 2\eta_2 \sum_{t\,=\,0}^{T-1} \lambda^{(t)} \big(\, V_g^\star(\rho)\, -\, {V}_g^{(t)}(\rho)\, \big)
			\,+\,
			2\eta_2 \sum_{t\,=\,0}^{T-1} \lambda^{(t)} \big(\, {V}_g^{(t)}(\rho)\, -\, \hat{V}_g^{(t)}(\rho)\, \big)
			\\[0.2cm]
			&&\displaystyle+~ \eta_2^2 \sum_{t\,=\,0}^{T-1}\big(\, \hat{V}_g^{(t)}(\rho)\, -\, b\, \big)^2
		\end{array}
		\]
		where the second inequality is due to the feasibility of the policy $\pi^\star$: $V_g^\star (\rho) \geq b$. Since ${V}_g^{(t)}(\rho)$ is a population quantity and $\hat{V}_g^{(t)}(\rho)$ is an estimate that is independent of $\lambda^{(t)}$ given the past history, $\lambda^{(t)}$ is independent of ${V}_g^{(t)}(\rho)-\hat{V}_g^{(t)}(\rho)$ at time $t$ and thus $\mathbb{E}\big[\,{\lambda^{(t)} \big({{V}_g^{(t)}(\rho)-\hat{V}_g^{(t)}(\rho)}\big)}\,\big] = 0$ due to the fact $\mathbb{E}\big[\,{\hat{V}_g^{(t)}(\rho)}\,\big] = {V}_g^{(t)}(\rho)$; see it in Appendix~\ref{alg.general}. Therefore, 
		\begin{equation}\label{eq.keytwo.sample}
			\displaystyle-\,\mathbb{E}\sbr{\,\frac{1}{T}\sum_{t\,=\,0}^{T-1} \lambda^{(t)} \big(\,V_g^\star(\rho)\, -\, {V}_g^{(t)}(\rho)\,\big)\,}
			\; \leq \;
			\displaystyle\mathbb{E}\sbr{\,\frac{\eta_2}{2T} \sum_{t\,=\,0}^{T-1}\big(\,\hat{V}_g^{(t)}(\rho)\, -\, b\,\big)^2\,}
			\; \leq \;
			\displaystyle\frac{2\eta_2}{(1-\gamma)^2}
		\end{equation}
		where in the second inequality we drop a non-positive term and use the fact
		\[
		\mathbb{E}\sbr{\,\big(\hat{V}_g^{(t)}(\rho)\big)^2\,} 
		\; = \; 
		\displaystyle\text{Var}\sbr{\,\hat{V}_g^{(t)}(s)\,} 
		\,+ \, \rbr{\mathbb{E}\sbr{\,\hat{V}_g^{(t)}(s)\,} }^2
		\; \leq \;
		\displaystyle\frac{2}{(1-\gamma)^2}
		\]
		where the inequality is due to that $\text{Var}\big[\,{\hat{V}_g^{(t)}(s)}\,\big] \leq {1}/(1-\gamma)^2$; see it in Appendix~\ref{alg.general}, and $\mathbb{E}\sbr{\,\hat{V}_g^{(t)}(\rho)\,} = {V}_g^{(t)}(\rho)$, where $0\leq {V}_g^{(t)}(s) \leq {1}/(1-\gamma)$. 
		
		Adding the inequality~\eqref{eq.keytwo.sample} to~\eqref{eq.keyone.sample} on both sides and taking $\eta_1=\eta_2={1}/{\sqrt{T}}$ yield the first inequality.
		
		\noindent\textbf{Proving the second inequality}. From the dual update in Algorithm~\ref{alg.sample-based.general}, we have for any $\lambda\in\Lambda \DefinedAs \big[\, 0,{2}/((1-\gamma)\xi) \, \big]$
		\[
		\begin{array}{rcl}
			&& \!\!\!\!\!\!\!\!\!\!\!\! \mathbb{E}\sbr{\vert\lambda^{(t+1)} - \lambda\vert^2} 
			\\[0.2cm]
			&=& \mathbb{E}\sbr{\,\abr{\calP_\Lambda \rbr{\lambda^{(t)} \, -\,  \eta_2 \big(\, \hat V_g^{(t)}(\rho)\, -\, b\,  \big)}  -\calP_\Lambda(\lambda)}^2\,}
			\\[0.2cm]
			&\overset{(a)}{\leq}& \mathbb{E}
			\sbr{\,
				\big\vert{{\lambda^{(t)} \, -\,  \eta_2 \big(\, \hat V_g^{(t)}(\rho)\, -\, b\,  \big)}  \, -\, \lambda}\big\vert^2
				\,}
			\\[0.2cm]
			&=& \mathbb{E}
			\sbr{
				\,\abr{{\lambda^{(t)}  \, -\, \lambda}}^2
				\,} 
			\,-\,
			2\eta_2\,\mathbb{E}
			\sbr{ \, \big(\, \hat V_g^{(t)}(\rho)\, -\, b\,  \big) \rbr{\lambda^{(t)}  \, -\, \lambda} \,}
			\,+\,
			\eta_2^2\, \mathbb{E}\sbr{\,
				\big(\,{\hat V_g^{(t)}(\rho)\, -\, b }\,\big)^2\, }
			\\[0.2cm]
			&\overset{(b)}{\leq}&\displaystyle \mathbb{E}\sbr{\,\big\vert{{\lambda^{(t)} \, -\, \lambda}}\big\vert^2\,} 
			\,-\,
			2\eta_2\, \mathbb{E}\sbr{\,\big(\, \hat V_g^{(t)}(\rho)\, -\, b\, \big)\big(\, \lambda^{(t)}  \, -\, \lambda\, \big)\,} 
			\,+\,
			\frac{3\eta_2^2}{(1-\gamma)^2}
		\end{array}
		\]
		where $(a)$ is due to the non-expansiveness of projection $\calP_\Lambda$ and $(b)$ is due to $\mathbb{E}\big[\,{({\hat V_g^{(t)}(\rho)-b })^2}\,\big]\leq {2}/{(1-\gamma)^2}+{1}/{ (1-\gamma)^2}$. Summing it up from $t=0$ to $t=T-1$ and dividing it by $T$ yield
		\[
		\begin{array}{rcl}
			0&\leq&
			\displaystyle\frac{1}{T}\, \mathbb{E}\sbr{\,\vert\lambda^{(T)} \, -\,  \lambda\vert^2 \,}
			\\[0.2cm]
			&\leq& \displaystyle\frac{1}{T}\,\mathbb{E}\sbr{\,\big\vert{{\lambda^{(0)}  \, -\, \lambda}}\big\vert^2\,} 
			\,-\,
			\frac{2\eta_2}{T}\sum_{t\,=\,0}^{T-1} \mathbb{E}\sbr{\,\big(\, \hat V_g^{(t)}(\rho)\, -\, b\, \big)\big(\, \lambda^{(t)}  \, -\,  \lambda\, \big)\,} 	
			\,+\,
			\frac{3 \eta_2^2}{(1-\gamma)^2}
		\end{array}
		\]
		which further implies that
		\[
		\mathbb{E}
		\sbr{\,
			\frac{1}{T}\sum_{t\,=\,0}^{T-1} \big(\, V_g^{(t)}(\rho)\, -\, b \, \big)\big(\, \lambda^{(t)}  \, -\, \lambda\, \big)
			\,}
		\;\leq\; 
		\frac{1}{2\eta_2T}\, \mathbb{E}
		\sbr{\,
			\big\vert{{\lambda^{(0)}  \, -\, \lambda}\big\vert}^2 \,} \,+\,
		\frac{2\eta_2}{(1-\gamma)^2}
		\]
		where we use $\mathbb{E}\big[\,{\hat V_g^{(t)}(\rho)} \,\big]= V_g^{(t)}(\rho)$ and $\lambda^{(t)}$ is independent of $\hat V_g^{(t)}(\rho)$ given the past history.
		We now add the above inequality into~\eqref{eq.keyone.sample} on both sides and utilize $V_g^{\star}(\rho) \geq b$:
		\[
		\begin{array}{rcl}
			&& \!\!\!\!\!\!\!\!\!\!\!\!\displaystyle \mathbb{E}\sbr{\,\frac{1}{T} \sum_{t\,=\,0}^{T-1} \big( \, V_r^\star(\rho) \,-\, V_r^{(t)}(\rho) \,\big)\,} \,+\, \lambda\,\mathbb{E}\sbr{\,\frac{1}{T} \sum_{t\,=\,0}^{T-1} \big(\,b\,-\, V_g^{(t)}(\rho)\,\big)\,}
			\\[0.2cm]
			&\leq&\displaystyle \frac{\log|A|}{(1-\gamma)\eta_1 T} 
			\,+\,
			\frac{1}{(1-\gamma)T}\sum_{t\,=\,0}^{T-1} \mathbb{E} \sbr{\,\text{\normalfont err}_r^{(t)} (\pi^\star)\,}
			\,+\,
			\frac{2}{(1-\gamma)^2\xi T}\sum_{t\,=\,0}^{T-1} \mathbb{E}\sbr{\,\text{\normalfont err}_g^{(t)} (\pi^\star)\,}
			\\[0.2cm]
			&& \displaystyle
			+~\beta\frac{\eta_1W^2}{(1-\gamma)^3} 
			\,+\, 
			\beta\frac{4\eta_1W^2}{(1-\gamma)^5\xi^2}
			\,+\,
			\frac{1}{2\eta_2T} \mathbb{E}\sbr{\, \big\vert{{\lambda^{(0)}  -\lambda}}\big\vert^2 \,}
			\,+\,
			\frac{2\eta_2}{(1-\gamma)^2}.
		\end{array}
		\]
		By taking $\lambda = \frac{2}{(1-\gamma)\xi}$ when ${\sum_{t\,=\,0}^{T-1} \big({b -V_g^{(t)}(\rho)}}\big)\geq 0$; otherwise $\lambda=0$, we reach
		\[
		\begin{array}{rcl}
			&& \!\!\!\! \!\!\!\! \!\!\displaystyle \mathbb{E}\sbr{\,V_r^\star(\rho)\,-\, \frac{1}{T} \sum_{t\,=\,0}^{T-1}{ V_r^{(t)}(\rho) }\,} 
			\,+\, 
			\frac{2}{(1-\gamma)\xi}\mathbb{E}\,\sbr{\,b \,-\, \frac{1}{T} \sum_{t\,=\,0}^{T-1} {V_g^{(t)}(\rho)}\,}_+
			\\[0.2cm]
			&\leq&\displaystyle \frac{\log|A|}{(1-\gamma)\eta_1 T} 
			\,+\,
			\frac{1}{(1-\gamma)T}\sum_{t\,=\,0}^{T-1} \mathbb{E}\sbr{\,\text{\normalfont err}_r^{(t)} (\pi^\star)\,}
			\,+\,
			\frac{2}{(1-\gamma)^2\xi T}\sum_{t\,=\,0}^{T-1} \mathbb{E}\sbr{\,\text{\normalfont err}_g^{(t)} (\pi^\star)\,}
			\\[0.2cm]
			&& \displaystyle
			+~\beta\frac{\eta_1W^2}{(1-\gamma)^3} 
			\,+\, 
			\beta\frac{4\eta_1W^2}{(1-\gamma)^5\xi^2}
			\,+\,
			\frac{2}{\eta_2 (1-\gamma)^2\xi^2T}
			\,+\,
			\frac{2\eta_2}{(1-\gamma)^2}.
		\end{array}
		\]
		
		Since $V_r^{(t)}(\rho)$ and $V_g^{(t)}(\rho)$ are linear functions in the occupancy measure~\cite[Chapter~10]{altman1999constrained}, there exists a policy $\pi'$ such that $V_r^{\pi'}(\rho)=\frac{1}{T} \sum_{t\,=\,0}^{T-1}{ V_r^{(t)}(\rho) }$ and $V_g^{\pi'}(\rho)=\frac{1}{T} \sum_{t\,=\,0}^{T-1}{ V_g^{(t)}(\rho) }$. Hence,
		\[
		\begin{array}{rcl}
			&& \!\!\!\! \!\!\!\! \!\! \displaystyle \mathbb{E}\sbr{\,V_r^\star(\rho)\,-\, V_r^{\pi'}(\rho)\,} 
			\,+\,
			 \frac{2}{(1-\gamma)\xi} \mathbb{E}\,\sbr{\,b \,-\,  {V_g^{\pi'}(\rho)}\,}_+
			\\[0.2cm]
			&\leq&\displaystyle \frac{\log|A|}{(1-\gamma)\eta_1 T} 
			\,+\,
			\frac{1}{(1-\gamma)T}\sum_{t\,=\,0}^{T-1} \mathbb{E}\sbr{\,\text{\normalfont err}_r^{(t)} (\pi^\star)\,}
			\,+\,
			\frac{2}{(1-\gamma)^2\xi T}\sum_{t\,=\,0}^{T-1} \mathbb{E}\sbr{\,\text{\normalfont err}_g^{(t)} (\pi^\star)\,}
			\\[0.2cm]
			&& \displaystyle
			+~\beta\frac{\eta_1W^2}{(1-\gamma)^3} 
			\,+\, 
			\beta\frac{4\eta_1W^2}{(1-\gamma)^5\xi^2}
			\,+\,
			\frac{2}{\eta_2 (1-\gamma)^2\xi^2T}
			\,+\,
			\frac{2\eta_2}{(1-\gamma)^2}.
		\end{array}
		\]
		From Lemma~\ref{lem.duality} (ii), we have $\lambda^\star \leq 1/((1-\gamma)\xi)$. Application of Lemma~\ref{lem.constraint} with $C= {2}/((1-\gamma)\xi)$ yields
		\[
		\begin{array}{rcl}
			\mathbb{E}\sbr{\,b \,-\,  {V_g^{\pi'}(\rho)}\,}_+
			&\leq&\displaystyle \frac{\xi\log|A|}{\eta_1 T} 
			\,+\,
			\frac{\xi}{T}\sum_{t\,=\,0}^{T-1} \mathbb{E}\sbr{\,\text{\normalfont err}_r^{(t)} (\pi^\star)\,}
			\,+\,
			\frac{2}{(1-\gamma) T}\sum_{t\,=\,0}^{T-1} \mathbb{E}\sbr{\,\text{\normalfont err}_g^{(t)} (\pi^\star)\,}
			\\[0.2cm]
			&& \displaystyle
			+~\beta\frac{\eta_1\xi W^2}{(1-\gamma)^2} 
			\,+\, 
			\beta\frac{4\eta_1 W^2}{(1-\gamma)^4\xi }
			\,+\,
			\frac{2}{\eta_2 (1-\gamma)\xi T}
			\,+\,
			\frac{2\eta_2\,\xi}{(1-\gamma)}.
		\end{array}
		\]
		which leads to our constraint violation bound if we utilize $\mathbb{E}\sbr{\,\frac{1}{T} \sum_{t\,=\,0}^{T-1} \big({b -V_g^{(t)}(\rho)}\big)\,} = \mathbb{E}\sbr{\,b -  {V_g^{\pi'}(\rho)}\,}$ and taking $\eta_1 =\eta_2={1}/{\sqrt{T}}$.
	\end{proof}
	
	\begin{proof}[Proof of Theorem~\ref{thm.samplecomplexity.general}]
		
		By Lemma~\ref{lem.gap.violation.sample}, we only need to consider the randomness in sequences of $\hat w^{(t)}$ and bound $\mathbb{E}\sbr{\,\text{\normalfont err}_\diamond^{(t)} (\pi^\star)\,}$ for $\diamond = r$ or $g$. Application of the triangle inequality yields
		\begin{equation}\label{eq.err.sample}
			\begin{array}{rcl}
				\widehat{\text{\normalfont err}}_r^{(t)} (\pi^\star)
				& \leq &
				\left\vert
				\mathbb{E}_{s\,\sim\,d_\rho^\star} \mathbb{E}_{a\,\sim\,\pi^\star(\cdot\,\vert\,s)} \sbr{A_r^{(t)}(s,a) \,-\, \left(w_{r,\star}^{(t)}\right)^\top \nabla_\theta \log \pi_\theta^{(t)}(a\,\vert\,s)}
				\right\vert
				\\[0.2cm]
				&& +~\left\vert
				\mathbb{E}_{s\,\sim\,d_\rho^\star} \mathbb{E}_{a\,\sim\,\pi^\star(\cdot\,\vert\,s)} \sbr{ \left(w_{r,\star}^{(t)}\, -\, \hat w_{r}^{(t)}\right)^\top \nabla_\theta \log \pi_\theta^{(t)}(a\,\vert\,s) }
				\right\vert
			\end{array}
		\end{equation}
		where $w_{r,\star}^{(t)} \in \argmin_{\norm{w_r}_2\,\leq\, W} {E}_r^{\nu^{(t)}} (w_r; \theta^{(t)})$. We next bound each term on the right-hand side of~\eqref{eq.err.sample}, separately.
		For the first term,
		\begin{equation}\label{eq.err0g.sample}
			\begin{array}{rcl}
				&& \!\!\!\! \!\!\!\! \!\!	
				\mathbb{E}_{s\,\sim\,d_\rho^\star} \mathbb{E}_{a\,\sim\,\pi^\star(\cdot\,\vert\,s)} 
				\sbr{\,
					A_r^{(t)}(s,a) 
					\, - \,
					\left(w_{r,\star}^{(t)}\right)^\top \nabla_\theta \log \pi_\theta^{(t)}(a\,\vert\,s)
					\,}
				\\[0.2cm]
				& \leq &
				\sqrt{
					\mathbb{E}_{s\,\sim\,d_\rho^\star} \mathbb{E}_{a\,\sim\,\pi^\star(\cdot\,\vert\,s)} \rbr{A_r^{(t)}(s,a) 
						\, - \, \left(w_{r,\star}^{(t)}\right)^\top \nabla_\theta \log \pi_\theta^{(t)}(a\,\vert\,s) }^2
				}
				\\[0.2cm]
				& = &
				\sqrt{{E}_r^{\nu^\star} \left(w_{r,\star}^{(t)} ; \theta^{(t)}\right) }.
			\end{array}
		\end{equation}
		Similarly,
		\begin{equation}\label{eq.err1g.sample}
			\begin{array}{rcl}
				&& \!\!\!\! \!\!\!\!  \!\!
				\mathbb{E}_{s\,\sim\,d_\rho^\star} \mathbb{E}_{a\,\sim\,\pi^\star(\cdot\,\vert\,s)} \sbr{ \left(w_{r,\star}^{(t)}\, -\, \hat w_{r}^{(t)}\right)^\top \nabla_\theta \log \pi_\theta^{(t)}(a\,\vert\,s) }
				\\[0.2cm]
				&\leq&
				\sqrt{
					\mathbb{E}_{s\,\sim\,d_\rho^\star} \mathbb{E}_{a\,\sim\,\pi^\star(\cdot\,\vert\,s)} \sbr{ \rbr{ \left(w_{r,\star}^{(t)}\, -\, \hat w_{r}^{(t)}\right)^\top  \nabla_\theta \log \pi_\theta^{(t)}(a\,\vert\,s) }^2 }
				}
				\\[0.2cm]
				&=&
				\sqrt{ 
					\norm{ w_{r,\star}^{(t)}\, -\, \hat w_{r}^{(t)} }_{\Sigma_{\nu^\star}^{(t)}}^2.
				}
			\end{array}
		\end{equation}
		
		We let $\kappa^{(t)} \DefinedAs \big\Vert{ \left(\Sigma_{\nu_0}^{(t)}\right)^{-1/2} \Sigma_{\nu^\star}^{(t)} \left(\Sigma_{\nu_0}^{(t)}\right)^{-1/2} }\big\Vert_2$ be the relative condition number at time $t$. Thus,
		\begin{equation}\label{eq.err2g.sample}
			\begin{array}{rcl}
				\norm{ w_{r,\star}^{(t)}\, -\, \hat w_{r}^{(t)} }_{\Sigma_{\nu^\star}^{(t)}}^2
				& \leq  &
				\displaystyle
				\norm{\left(\Sigma_{\nu_0}^{(t)}\right)^{-1/2} \Sigma_{\nu^\star}^{(t)} \left(\Sigma_{\nu_0}^{(t)}\right)^{-1/2}}
				\norm{ w_{r,\star}^{(t)}\, -\, \hat w_{r}^{(t)} }_{\Sigma_{\nu_0}^{(t)}}^2
				\\[0.2cm]
				& \overset{(a)}{\leq} & 
				\displaystyle
				\frac{\kappa^{(t)}}{1-\gamma}
				\norm{ w_{r,\star}^{(t)}\, -\, \hat w_{r}^{(t)} }_{\Sigma_{\nu^{(t)}}}^2
				\\[0.2cm]
				& \overset{(b)}{\leq} & 
				\displaystyle
				\frac{\kappa^{(t)}}{1-\gamma}
				\left( {E}_r^{\nu^{(t)}} \big(\hat w_{r}^{(t)} ; \theta^{(t)}\big) \, -\,  {E}_r^{\nu^{(t)}} \big(w_{r,\star}^{(t)} ; \theta^{(t)}\big) \right)
			\end{array}
		\end{equation}
		where we use 
		$(1-\gamma)\nu_0 \leq \nu_{\nu_0}^{\pi^{(t)}} \DefinedAs \nu^{(t)}$ in $(a)$, and we get $(b)$ due to that
		the first-order optimality condition for $w_{r,\star}^{(t)}$:
		\[
		\left( w_r \, -\,  w_{r,\star}^{(t)} 
		\right)^\top
		\nabla_\theta {E}_r^{\nu^{(t)}} \big(w_{r,\star}^{(t)}; \theta^{(t)}\big) 
		\; \geq  \;
		0, 
		\; \text{ for any } w_r \text{ satisfying } \norm{w_r} \,\leq\, W
		\]
		further implies that
		\[
		\begin{array}{rcl}
			&& \!\!\!\! \!\!\!\! \!\!
			{E}_r^{\nu^{(t)}} \big(w_{r}; \theta^{(t)}\big) \, - \, {E}_r^{\nu^{(t)}} \big(w_{r,\star}^{(t)}; \theta^{(t)}\big) 
			\\[0.2cm]
			& = & 	
			\mathbb{E}_{(s,a)\,\sim\,\nu^{(t)}} \left[\,  \rbr{ A_r^{(t)}(s,a) \, -\,   \phi_{s,a}^\top w_{r,\star}^{(t)} \, +\,  \phi_{s,a}^\top  w_{r,\star}^{(t)} \, -\,  \phi_{s,a}^\top w_r  }^2 \,\right]
			\, - \, 
			{E}_r^{\nu^{(t)}} \big(w_{r,\star}^{(t)}; \theta^{(t)}\big) 
			\\[0.4cm]
			& = & 
			2 \left(w_{r,\star}^{(t)} \, -\,  w_r\right)^\top
			\mathbb{E}_{s,a\,\sim\,\nu^{(t)}} \left[\,  \rbr{ A_r^{(t)}(s,a) \, -\,   \phi_{s,a}^\top w_{r,\star}^{(t)}} \phi_{s,a}  \,\right]
			\\[0.2cm]
			&& + \,
			\mathbb{E}_{(s,a)\,\sim\,\nu^{(t)}} \left[\,  \rbr{  \phi_{s,a}^\top w_{r,\star}^{(t)} \, -\,  \phi_{s,a}^\top w_r  }^2 \,\right]
			\\[0.4cm]
			& = &  \left( w_r  \, -\,  w_{r,\star}^{(t)} \right)^\top \nabla_\theta {E}_r^{\nu^{(t)}} (w_{r,\star}^{(t)}; \theta^{(t)})  
			\, + \, 
			\norm{ w_r \, -\,  w_{r,\star}^{(t)} }_{\Sigma_{\nu^{(t)}}}^2
			\\[0.4cm]
			& \geq &  \norm{ w_r \, -\,  w_{r,\star}^{(t)} }_{\Sigma_{\nu^{(t)}}}^2.
		\end{array}
		\]
		
		Taking an expectation over~\eqref{eq.err2g.sample} from both sides yields
		\begin{equation}\label{eq.err1g.sample.2nd}
			\begin{array}{rcl}
				\mathbb{E}\sbr{ 
					\norm{ w_{r,\star}^{(t)} \, - \, w_{r}^{(t)} }_{\Sigma_{\nu^\star}^{(t)}}^2
				}
				& \leq & \displaystyle
				\mathbb{E}
				\sbr{\,  
					\frac{\kappa^{(t)}}{1-\gamma}\,
					\mathbb{E}\sbr{\,
						{E}_r^{\nu^{(t)}} \big(\hat w_{r}^{(t)} ; \theta^{(t)}\big) 
						\, - \,
						{E}_r^{\nu^{(t)}} \big(w_{r,\star}^{(t)} ; \theta^{(t)}\big)\,\bigg\vert\,\theta^{(t)} 
						\,}  
					\,}
				\\[0.4cm]
				& \overset{(a)}{\leq} & \displaystyle 	\mathbb{E}\sbr{  \frac{\kappa^{(t)}}{1-\gamma}\,\frac{2G^2}{\sigma_F(K+1)} }
				\\[0.4cm]
				& \overset{(b)}{\leq} & \displaystyle 	 \frac{2\kappa\, G^2}{ \sigma_F (1-\gamma)(K+1) }
			\end{array}
		\end{equation}
		where $(a)$ is due to the standard SGD result~\citep{lacoste2012simpler}: for $\alpha_k = 2/(\sigma_F (k+1))$, 
		\[
		E_{r,\text{\normalfont est}}^{(t)} 
		\; = \;
		\mathbb{E}
		\sbr{\, 
			{E}_r^{\nu^{(t)}} \big(\hat w_{r}^{(t)} ; \theta^{(t)}\big) 
			\, - \,
			{E}_r^{\nu^{(t)}} \big(w_{r,\star}^{(t)} ; \theta^{(t)}\big) 
			\,} 
		\;\leq\;
		\frac{2G^2}{\sigma_F (K+1)}
		\]
		and $(b)$ follows from Assumption~\ref{as.errors+condtion}. Here, it is straightforward to check the second-order moment of stochastic gradient  $G_{\diamond,k}$ using Assumption~\ref{as.Lipschitzpolicy}:
		\[
		\mathbb{E}\sbr{\, \norm{G_{\diamond,k}}^2 \,}
		\;\leq\; 
		4L_\pi^2 
		\left( W^2 L_\pi^2 
		\,+\, 
		\frac{2}{(1-\gamma)^2} \right)
		\; = \; G^2.
		\]
		
		Substitution of~\eqref{eq.err1g.sample} and~\eqref{eq.err1g.sample.2nd} into the right-hand side of~\eqref{eq.err.sample} yields an upper bound on $\mathbb{E} \big[\,
		\text{\normalfont err}_r^{(t)} (\pi^\star) \,\big]$.
		By the same reasoning, we can establish a similar bound on $\mathbb{E} \big[\,
		\text{\normalfont err}_g^{(t)} (\pi^\star) \,\big]$. 
		Finally, application of these upper bounds to Lemma~\ref{lem.gap.violation.sample} leads to our desired results.
	\end{proof}
	
	\section{Proof of Theorem~\ref{thm.samplecomplexity.loglinear}}\label{pf.samplecomplexity.loglinear}
	
	By $\norm{\phi_{s,a}}\leq B$, for the log-linear policy class, $\log \pi_\theta(a\,\vert\,s)$ is $\beta$-smooth with $\beta = B^2$.
	By Lemma~\ref{lem.gap.violation.sample}, we only need to consider the randomness in the sequence of $\hat w^{(t)}$ and the error bounds for $\mathbb{E}\big[\,{\widehat{\text{\normalfont err}}_r^{(t)}(\pi^\star)}\,\big]$ and $\mathbb{E}\big[\,{\widehat{\text{\normalfont err}}_g^{(t)}(\pi^\star)}\,\big]$. 
	We first use~\eqref{eq.err.sample} and consider the following cases.
	By~\eqref{eq.policyshift} and $A_r^{(t)}(s,a)  = Q_r^{(t)}(s,a) -  \mathbb{E}_{a'\,\sim\,\pi_\theta^{(t)}(\cdot\,\vert\,s)} \left[\,Q_r^{(t)}(s,a')\,\right]$,
	\begin{equation}\label{eq.err0.sample}
		\begin{array}{rcl}
			&& \!\!\!\! \!\!\!\! \!\!	
			\mathbb{E}_{s\,\sim\,d_\rho^\star} \mathbb{E}_{a\,\sim\,\pi^\star(\cdot\,\vert\,s)} 
			\sbr{\,
				A_r^{(t)}(s,a) 
				\, - \, \left(w_{r,\star}^{(t)}\right)^\top \nabla_\theta \log \pi_\theta^{(t)}(a\,\vert\,s)
				\,}
			\\[0.2cm]
			& = & 
			\mathbb{E}_{s\,\sim\,d_\rho^\star} \mathbb{E}_{a\,\sim\,\pi^\star(\cdot\,\vert\,s)} 
			\sbr{\,
				Q_r^{(t)}(s,a) 
				\, - \,
				\phi_{s,a}^\top w_{r,\star}^{(t)}
				\,}
			\\[0.2cm]
			&&  - \,
			\mathbb{E}_{s\,\sim\,d_\rho^\star} \mathbb{E}_{a'\,\sim\,\pi_\theta^{(t)}(\cdot\,\vert\,s)} 
			\sbr{\,
				Q_r^{(t)}(s,a') 
				\, - \,
				\phi_{s,a'}^\top w_{r,\star}^{(t)}
				\,}
			\\[0.2cm]
			& \leq &
			\sqrt{
				\mathbb{E}_{s\,\sim\,d_\rho^\star} \mathbb{E}_{a\,\sim\,\pi^\star(\cdot\,\vert\,s)} \sbr{\rbr{Q_r^{(t)}(s,a) 
						\, - \,
						\phi_{s,a}^\top w_{r,\star}^{(t)}}^2}
			}
			\\[0.2cm]
			&& + \,
			\sqrt{
				\mathbb{E}_{s\,\sim\,d_\rho^\star} \mathbb{E}_{a'\,\sim\,\pi_\theta^{(t)}(\cdot\,\vert\,s)} \sbr{\rbr{Q_r^{(t)}(s,a') \,-\, \phi_{s,a'}^\top w_{r,\star}^{(t)} }^2}
			}
			\\[0.2cm]
			& \leq & 2
			\sqrt{ |A| \,
				\mathbb{E}_{s\,\sim\,d_\rho^\star} \mathbb{E}_{a\,\sim\,\text{Unif}_A} \sbr{ \rbr{Q_r^{(t)}(s,a) \,-\, \phi_{s,a}^\top w_{r,\star}^{(t)} }^2 }
			}
			\\[0.2cm]
			& = &
			2\sqrt{ |A| \, \mathcal{E}_r^{\nu^\star} \big(w_{r,\star}^{(t)} ; \theta^{(t)}\big) }.
		\end{array}
	\end{equation}
	Similarly,
	\begin{equation}\label{eq.err1.sample}
		\begin{array}{rcl}
			&& \!\!\!\! \!\!\!\!  \!\!
			\mathbb{E}_{s\,\sim\,d_\rho^\star} \mathbb{E}_{a\,\sim\,\pi^\star(\cdot\,\vert\,s)} 
			\sbr{\, \left(w_{r,\star}^{(t)}\, -\, \hat w_{r}^{(t)}\right)^\top \nabla_\theta \log \pi_\theta^{(t)}(a\,\vert\,s) 
				\,}
			\\[0.2cm]
			&=&
			\mathbb{E}_{s\,\sim\,d_\rho^\star} \mathbb{E}_{a\,\sim\,\pi^\star(\cdot\,\vert\,s)} 
			\sbr{\, 
				\left(w_{r,\star}^{(t)}\, -\, \hat w_{r}^{(t)}\right)^\top \phi_{s,a} 
				\,}
			\\[0.2cm]
			&& - ~ \mathbb{E}_{s\,\sim\,d_\rho^\star} \mathbb{E}_{a'\,\sim\,\pi_\theta^{(t)}(\cdot\,\vert\,s)} \sbr{ \left(w_{r,\star}^{(t)}\, -\, \hat w_{r}^{(t)}\right)^\top \phi_{s,a'} }
			\\[0.2cm]
			&\leq&
			2 \sqrt{ |A|
				\mathbb{E}_{s\,\sim\,d_\rho^\star} \mathbb{E}_{a\,\sim\,\text{Unif}_A} \sbr{ \rbr{ \left(w_{r,\star}^{(t)}\, -\, \hat w_{r}^{(t)}\right)^\top \phi_{s,a} }^2 }
			}
			\\[0.2cm]
			&=&
			2 \sqrt{ 
				|A|  \norm{ w_{r,\star}^{(t)}\, -\, \hat w_{r}^{(t)} }_{\Sigma_{\nu^\star}}^2
			}
		\end{array}
	\end{equation}
	where $\Sigma_{\nu^\star} \DefinedAs \mathbb{E}_{(s,a)\,\sim\,\nu^\star} \left[\, \phi_{s,a} \phi_{s,a}^\top \,\right]$. By the definition of $\kappa$,
	\begin{equation}\label{eq.err2.sample}
		\norm{ w_{r,\star}^{(t)}\, -\, \hat w_{r}^{(t)} }_{\Sigma_{\nu^\star}}^2
		\; \leq  \;
		\kappa
		\norm{ w_{r,\star}^{(t)}\, -\, \hat w_{r}^{(t)} }_{\Sigma_{\nu_0}}^2
		\; \leq  \;
		\frac{\kappa}{1-\gamma}
		\norm{ w_{r,\star}^{(t)}\, -\, \hat w_{r}^{(t)} }_{\Sigma_{\nu^{(t)}}}^2
	\end{equation}
	where we use 
	$(1-\gamma)\nu_0 \leq \nu_{\nu_0}^{\pi^{(t)}} \DefinedAs \nu^{(t)}$ in the second inequality. We note that $w_{r,\star}^{(t)} \in \argmin_{\norm{w_r}_2\,\leq\, W} \,\mathcal{E}_r^{\nu^{(t)}} (w_r; \theta^{(t)})$. Application of the first-order optimality condition for $w_{r,\star}^{(t)}$ yields
	\[
	\left( w_r \, -\,  w_{r,\star}^{(t)} 
	\right)^\top
	\nabla_\theta \mathcal{E}_r^{\nu^{(t)}} \big(w_{r,\star}^{(t)}; \theta^{(t)}\big) 
	\; \geq  \;
	0, 
	\; \text{ for any } w_r \text{ satisfying } \norm{w_r} \,\leq\, W. 
	\]
	Thus,
	\[
	\begin{array}{rcl}
		&& \!\!\!\! \!\!\!\! \!\!
		\mathcal{E}_r^{\nu^{(t)}} \big(w_{r}; \theta^{(t)}\big) \, -\,  \mathcal{E}_r^{\nu^{(t)}} \big(w_{r,\star}^{(t)}; \theta^{(t)}\big) 
		\\[0.2cm]
		& = & 	
		\mathbb{E}_{s,a\,\sim\,\nu^{(t)}} \left[\,  \rbr{ Q_r^{(t)}(s,a) \, -\,   \phi_{s,a}^\top w_{r,\star}^{(t)} \, +\,  \phi_{s,a}^\top w_{r,\star}^{(t)} \, -\,  \phi_{s,a}^\top w_r  }^2 \,\right]
		\, - \, \mathcal{E}_r^{\nu^{(t)}} \big(w_{r,\star}^{(t)}; \theta^{(t)}\big) 
		\\[0.4cm]
		& = & 
		2 \left(w_{r,\star}^{(t)} \, -\,  w_r\right)^\top 
		\mathbb{E}_{s,a\,\sim\,\nu^{(t)}} \left[\,  \rbr{ Q_r^{(t)}(s,a) \, -\,   \phi_{s,a}^\top w_{r,\star}^{(t)} } \phi_{s,a}  \,\right]
		\\[0.2cm]
		&& + ~
		\mathbb{E}_{s,a\,\sim\,\nu^{(t)}} \left[\,  \rbr{  \phi_{s,a}^\top w_{r,\star}^{(t)} \, -\,  \phi_{s,a}^\top w_r  }^2 \,\right]
		\\[0.4cm]
		& = &  \left( w_r  \, -\,  w_{r,\star}^{(t)} \right)^\top	\nabla_\theta \mathcal{E}_r^{\nu^{(t)}} \big(w_{r,\star}^{(t)}; \theta^{(t)}\big) \, + \, \norm{ w_r \, -\,  w_{r,\star}^{(t)} }_{\Sigma_{\nu^{(t)}}}^2
		\\[0.4cm]
		& \geq &  \norm{ w_r \, -\,  w_{r,\star}^{(t)} }_{\Sigma_{\nu^{(t)}}}^2.
	\end{array}
	\]
	Taking $w_r =  \hat w_r^{(t)}$ in the inequality above and combining it with~\eqref{eq.err2.sample} and~\eqref{eq.err1.sample} yield
	\begin{equation}\label{eq.err3.sample}
		\begin{array}{rcl}
			&& \!\!\!\! \!\!\!\! \!\!
			\mathbb{E}_{s\,\sim\,d_\rho^\star} \mathbb{E}_{a\,\sim\,\pi^\star(\cdot\,\vert\,s)} 
			\sbr{\, \left(w_{r,\star}^{(t)}\, -\, \hat w_{r}^{(t)}\right)^\top \nabla_\theta \log \pi_\theta^{(t)}(a\,\vert\,s) 
				\,}
			\\[0.2cm]
			& \leq & \displaystyle
			2\sqrt{\frac{|A|\kappa}{1-\gamma} \left( \mathcal{E}_r^{\nu^{(t)}} \big(\hat w_{r}^{(t)}; \theta^{(t)}\big) \, -\,  \mathcal{E}_r^{\nu^{(t)}} \big(w_{r,\star}^{(t)}; \theta^{(t)}\big)  \right)}.
		\end{array}
	\end{equation}
	
	We now substitute~\eqref{eq.err0.sample} and~\eqref{eq.err3.sample} into the right-hand side of~\eqref{eq.err.sample} as follows
	\[
	\begin{array}{rcl}
		\mathbb{E} \sbr{\,
			\text{\normalfont err}_r^{(t)} (\pi^\star)\,}
		& \leq &
		2\sqrt{ |A| \,  \mathbb{E} \sbr{\, \mathcal{E}_r^{d^\star} \big(w_{r,\star}^{(t)} ; \theta^{(t)}\big) }\,}
		\, + \,
		2\sqrt{\dfrac{|A|\kappa}{1-\gamma} \mathbb{E} \sbr{\, \mathcal{E}_r^{\nu^{(t)}} \big(\hat w_{r}^{(t)}; \theta^{(t)}\big) - \mathcal{E}_r^{\nu^{(t)}} \big(w_{r,\star}^{(t)}; \theta^{(t)}\big) \,} }
		\\[0.2cm]
		& \leq &
		2\sqrt{ |A| \,  \mathbb{E} \sbr{\, \mathcal{E}_r^{d^\star} \big(w_{r,\star}^{(t)} ; \theta^{(t)}\big) \,}}
		\, + \,
		2\sqrt{\dfrac{|A|\kappa}{1-\gamma} \,\dfrac{2G^2}{\sigma_F (K+1)} }
	\end{array}
	\]
	where the second inequality is due to the standard SGD result~\citep{lacoste2012simpler}: for $\alpha_k = 2/(\sigma_F (k+1))$, 
	\[
	\mathcal E_{r,\text{\normalfont est}}^{(t)} 
	\; = \;
	\mathbb{E}\sbr{\, \mathcal{E}_r^{\nu^{(t)}} \big(\hat w_{r}^{(t)} ; \theta^{(t)}\big) \, -\,  \mathcal{E}_r^{\nu^{(t)}} \big(w_{r,\star}^{(t)} ; \theta^{(t)}\big) \,} 
	\;\leq\;
	\frac{2G^2}{\sigma_F(K+1)}.
	\]
	By the same reasoning, we can find a similar bound on $\mathbb{E} \big[\,{
		\text{\normalfont err}_g^{(t)} (\pi^\star)}\,\big]$. 
	Finally, our desired results follow by applying Assumption~\ref{as.errors} and Lemma~\ref{lem.gap.violation.sample}.

\end{document}